\documentclass[reqno,twoside]{amsart}

\usepackage{amsmath,amsthm,amssymb,amsfonts,mathrsfs,color}
\usepackage{latexsym,esint}
\usepackage{hyperref}

\usepackage{graphicx,pgf,tikz}
\usetikzlibrary{decorations}
\tikzstyle xyax=[thin]
\tikzstyle mlin=[thick]
\tikzstyle slin=[]

\newtheorem{theorem}{Theorem}[section]
\newtheorem{proposition}[theorem]{Proposition}
\newtheorem{lemma}[theorem]{Lemma}
\newtheorem{corollary}[theorem]{Corollary}
\newtheorem{definition}[theorem]{Definition}
\newtheorem{example}[theorem]{Example}
\newtheorem*{theorem*}{Lemma}

\theoremstyle{definition}
\newtheorem{remark}[theorem]{Remark}

\numberwithin{equation}{section}

\newcommand{\res}{\mathop{\hbox{\vrule height 7pt width .5pt depth 0pt
\vrule height .5pt width 6pt depth 0pt}}\nolimits}

\definecolor{purple}{rgb}{0.8,0.01,0.7}

\newcommand{\N}{\mathbb N} 
\newcommand{\R}{\mathbb R} 

\DeclareMathOperator{\Div}{div}

\newcommand{\e}{\varepsilon}
\newcommand{\ol}{\overline}

\newcommand\p{\varphi}
\newcommand\iO{\int_\O}

\newcommand\Oel{{\O_{el}}}
\newcommand\Opl{{\O_{pl}}}

\newcommand{\LL}{{\mathcal L}}
\newcommand{\HH}{{\mathcal H}}
\newcommand{\M}{{\mathcal M}}
\newcommand{\C}{{\mathcal C}}

\newcommand{\CC}{{\mathscr C}}
\newcommand\h{\HH^1}

\let\O=\Omega

\newcommand\pa{\partial}

\author[J.-F. Babadjian]{Jean-Fran\c cois Babadjian}
\author[G. A. Francfort]{Gilles A. Francfort}

\address[J.-F. Babadjian]{Universit\'e Paris-Saclay, CNRS,  Laboratoire de math\'ematiques d'Orsay, 91405, Orsay, France.}
\email{jean-francois.babadjian@universite-paris-saclay.fr}

\address[G.A. Francfort]{Flatiron Institute, 162 Fifth Avenue, NY-NY10010, USA}
\email{gfrancfort@flatironinstitute.org}



\date{\today}

\title[Uniqueness, regularity and characteristic flow]{Uniqueness, regularity and characteristic flow for a non strictly convex singular variational problem}

\begin{document}

\begin{abstract}
This work addresses the question of uniqueness and regularity of the minimizers of a convex but not strictly convex integral functional with linear growth in a two-dimensional setting. The integrand -- whose precise form derives directly from the theory of perfect plasticity -- behaves quadratically close to the origin and grows linearly once a specific threshold is reached. Thus, in contrast with the only existing literature on uniqueness for functionals with linear growth, that is that which pertains to the generalized least gradient,  the integrand is not  a norm. We make use of hyperbolic {conservation} laws hidden in the structure of the problem to tackle  uniqueness. Our argument strongly relies on the regularity of a vector field -- the Cauchy stress in the terminology of perfect plasticity -- which allows us to define characteristic lines, and then to employ  the  method of characteristics. Using the detailed structure of the characteristic landscape evidenced in  our preliminary study  \cite{BF}, we show that this vector field is actually continuous, save for  possibly two  points. The  different behaviors of the energy density at zero and at infinity imply an inequality constraint on the Cauchy stress. Under a {barrier type} convexity assumption on the set where the inequality constraint is saturated, we show that  uniqueness  holds for pure Dirichlet boundary data devoid of any regularity properties, a stronger result than that of uniqueness for a given trace on the whole boundary since our minimizers can fail to attain the boundary data. We also show a partial regularity result for the minimizer.
\end{abstract}

\maketitle

\tableofcontents

\section{Introduction}

This work should be seen as a sequel to, and a culmination of our previous work \cite{BF}, although the viewpoint and the presentation will be  different. We point out from the very beginning  that the setting is two-dimensional and that  the methods we use will not generalize to higher dimensions.

Given a bounded connected open subset $\O$ of $\R^2$ with Lipschitz boundary, we partition $\partial\O$ into the disjoint union of $\partial\O=\partial_D \O\cup \partial_N\O$ where $\partial_D \O$ is open in the relative topology of $\partial\O$ and $\partial_N\O=\partial\O \setminus \partial_D\O$. For a Dirichlet boundary data $w \in L^1(\partial_D \O)$ and a Neumann boundary data $g \in L^\infty(\partial_N\O)$, we consider the following problem of the calculus of variations
\begin{equation}\label{eq.op}
\inf\left\{\iO W(\nabla u)\, dx-\int_{\pa_N \O} gu\, d\h: \; u \in W^{1,1}(\O),\; u=w \mbox{ on }\pa_D \O\right\},
\end{equation}
where the potential $W:\R^2 \to \R$ is explicitly given by
\begin{equation}\label{eq:W}
W(\xi)=\begin{cases}\frac12|\xi|^2 & \text{ if } |\xi|\le 1,\\
|\xi|-\frac12 & \text{ if } |\xi|>1.
\end{cases}
\end{equation}
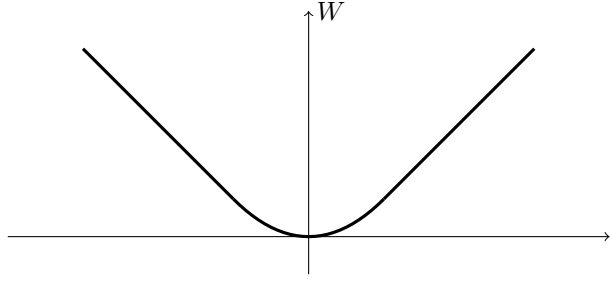
\begin{figure}[htbp]
\scalebox{1}{
\begin{tikzpicture}
\draw[->] (-4, 0) -- (4, 0);
\draw[->] (0, -0.5) -- (0, 3) ;
\draw [very thick, domain=1:3] plot (\x,{\x-1/2});
\draw [very thick, domain=-3:-1] plot (\x,{-\x-1/2});
\draw [very thick, domain=-1:1] plot (\x,{0.5*\x*\x});
\draw  (0.3,3) node{$W$};
\end{tikzpicture}}
\caption{The graph of $W$}
\label{fig:W}
\end{figure}
Because of the linear growth of $W$ at infinity, it is natural to seek $u$ in $W^{1,1}(\O)$. However, it is by now well-known (see {\it e.g.} \cite{FM,BDM}) that this problem needs to be relaxed in the larger space $BV(\O)$ of functions with bounded variation, and that the relaxed problem reads as
\begin{equation}\label{eq.rp}
\min\left\{\iO W(\nabla u)\, dx+ |D^su|(\O)+\int_{\pa_D\O}|w-u|\, d\h-\int_{\pa_N \O} gu\, d\h: \; u\in BV(\O) \right\},
\end{equation}
where, for $u\in BV(\O)$, $\nabla u$ denotes the Lebesgue-absolutely continuous part of $Du$, $D^su$ the Lebesgue-singular part of $Du$, and $|D^su|$ the variation measure of $D^su$.

As a corollary, the relaxed minimization problem  \eqref{eq.rp} has a solution in $BV(\O)$, at least provided that minimizing sequences  of \eqref{eq.op} remain bounded in $W^{1,1}(\O)$. This will be the case for example if the Neumann condition derives from a potential, {\it e.g.} there exists $\tau\in \C^0(\ol \O;\R^2)$ such that
\begin{equation}\label{eq.sl}
\begin{cases}{\rm div\,}\tau=0 & \text{ in }\O,\\
\tau\cdot \nu=g & \mbox{ on }\pa_N\O,\\
\|\tau\|_{L^\infty(\O)}\le\alpha<1&
\end{cases}
\end{equation}
(with $\nu$ the exterior unit normal to $\O$) as could be seen by observing (see {\it e.g.} \cite[Remark 2.10]{FG}) that the boundary term in \eqref{eq.op} can be replaced by
$$\int_{\pa_N\O} gu\, d\h=\int_\O\tau\cdot\nabla u\, dx.$$

\medskip

Our  goal, in both this and the previous paper \cite{BF}, is to adjudicate the uniqueness and regularity of such minimizers. 

As far as uniqueness is concerned our result is that uniqueness  holds true in the pure Dirichlet case, provided that a {barrier type} convexity condition on a well-defined set is satisfied: in essence that  $\{x\in\O: |\nabla u(x)|\ge 1\}$ be convex (see \eqref{eq.def.opl} below).  When departing from the pure Dirichlet case, uniqueness is false as explained further down in this introduction. See Theorems \ref{thm:BF1} and \ref{thm.uniq} for the precise statements of uniqueness and the example evidenced in our previous work (\cite[Subsection 1.2]{BF}) for the failure of uniqueness for mixed boundary conditions. Of course our uniqueness result does not completely solve the uniqueness question because of the convexity condition, even if we do not know of any situation in which it is violated, nor do we know if that set could be made of more than one convex component. Moreover, we contend that such a condition is inherent to the problem at hand, at least if we recall  uniqueness results on related problems available in the literature, as will be explained below. However, this result is also quite different from those available in the literature, this on two main grounds.

First, the energy $W$ is not a norm (a sub-additive, one-homogeneous function). Thus, we cannot appeal  to various uniqueness results (see \cite{JMN})  derived for special classes of Lipschitz domains. Those should be viewed in the footstep of a well-known result  (see \cite[Theorem 4.1]{SWZ}) that establishes existence and uniqueness of continuous minimizers with bounded total variation in the least gradient setting, provided that a continuous trace  is given.

Then, the result provides uniqueness for the Dirichlet problem, and not for the trace problem. In other words, we do not assume that the trace of the test functions ($BV$-functions) coincides with the Dirichlet datum. To our knowledge, the only result in that direction is \cite[Theorem 1.2]{Ledos} which asserts that there is at most one (but maybe no) solution for the Dirichlet problem with continuous data  in the class of uniformly continuous functions.  Here, no regularity restriction of any kind is imposed on the Dirichlet datum or on the solution.

As regards our convexity assumption, note that, in the least gradient setting (or the related setting of a norm), a so-called barrier assumption -- essentially a generalization of strict convexity -- is  imposed on the topology of the domain, so as to attain the relevant uniqueness results (see {\it e.g.} \cite{SWZ, JMN}). The most general assumption known to us is \cite[Theorem 1.1]{Gorny}  which establishes a  partial uniqueness under simple convexity assumptions.  Here, our energy $W$ looks like a least gradient energy provided that the norm of the gradient exceeds $1$. It is thus hardly surprising that we should impose a barrier-like condition on the set of points where the norm of the gradient exceeds $1$, which is not the whole domain but the set defined in \eqref{eq.def.opl} below.
We emphasize that that set, although not {\it a priori} known, is always uniquely determined.

As far as regularity is concerned, the example in our previous work (\cite[Subsection 1.2]{BF}) also demonstrates the possible lack of regularity of the minimizers for mixed boundary conditions. In the pure Dirichlet case and under the already mentioned convexity condition, we establish a partial regularity result for the minimizer $u$ (see Theorem \ref{thm:cont-u}). We also establish the continuity of the stress $\sigma$ in $\O$, except possibly at two points (see Theorem \ref{thm:cont_sigma}). The definition of the stress $\sigma$ is provided below starting with \eqref{eq.rpm}.

\medskip

The premise of our argument is to rewrite \eqref{eq.rp} in a more palatable manner. This is done by remarking that $W$ is actually the infimal convolution of the convex functions $\xi \mapsto \frac12|\xi|^2$ and $\xi \mapsto |\xi|$, {\it i.e.},
\begin{equation}\label{eq:infconv}
W(\xi)=\min_{p\in\R^2}\left\{\frac12|\xi-p|^2+|p|\right\}.
\end{equation}
Consequently, by measurable selection type arguments, \eqref{eq.rp} reads as
 \begin{equation}\label{eq.rpm}
 \min_{ (u,\sigma,p) \in \mathcal A_w}\left\{\iO \frac12 |\sigma|^2\, dx+ |p|(\O\cup\pa_D\O)-\int_{\pa_N\O}gu\, d\h\right\},
 \end{equation}
 where, referring to Section \ref{sec:not} regarding notation and functional spaces,
 \begin{multline*}
\mathcal A_w:=\{(u,\sigma,p) \in BV(\O) \times L^2(\O;\R^2) \times \M(\O\cup\partial_D \O;\R^2): \\
Du=\sigma+p \text{ in } \O, \ p=(w-u)\nu \HH^1\text{ on } \partial_D \O\}.
\end{multline*}

\begin{remark}
Formulation \eqref{eq.rpm} can be interpreted  in the setting of perfect Hencky plasticity, our original underlying motivation for this investigation. In that framework, $u$ is a scalar (anti-plane) displacement, $Du$ is the total strain which is additively decomposed as the sum of an elastic strain denoted by $\sigma$ -- here undistinguishable from  the stress -- and a plastic strain $p$. The stress $\sigma$ is constrained to remain in the closed unit ball. As long as $\sigma$ leaves in the interior of the ball no plastic strain appears ($p=0$) and the material behaves elastically (the displacement $u$ is formally harmonic in the elasticity set $\{|\sigma|<1\}$). Plastic strain lives  in the set $\{|\sigma|=1\}$ where the constraint is saturated. Note that, as usual in variational problems with linear growth, the Dirichlet condition might fail to be satisfied, in which case the boundary jump is energetically penalized. In the context of plasticity, it means that the plastic strain can also charge the boundary at those boundary points where the displacement $u$ does not match the prescribed Dirichlet data $w$. The Neumann data $g$ can be interpreted as a surface force and  \eqref{eq.sl} corresponds to the so called safe-load condition of classical plasticity.

As an aside, Hencky plasticity is itself a static version of  Von Mises plasticity, the canonical model of elasto-plasticity in solid mechanics. Elasto-plasticity, while a somewhat neglected field in the mathematical community, is  the solid equivalent of Navier-Stokes, in that it -- and not a viscosity driven model like visco-elasticity -- is the template for dissipative behavior in  crystalline solids. In such a light, our result should be seen as a first step towards uniqueness of elasto-plastic evolutions, a result   of paramount importance in solid mechanics and materials science.
\hfill\P\end{remark}
 
A proof of the existence of a minimizer for \eqref{eq.rpm} is trivial by the direct method in the calculus of variations. This  also provides an alternative proof of the relaxation of the functional
\begin{multline}\nonumber
 \min\left\{\iO \frac12|\sigma|^2\, dx+ \int_\O |p|\, dx-\int_{\pa_N\O}gu\, d\h: \; \nabla u=\sigma+p \text{ in }\O,\, u=w \mbox{ on }\pa_D\O,\right.\\
 \left.(u,\sigma,p)\in W^{1,1}(\O)\times L^2(\O;\R^2)\times L^1(\O;\R^2)\right\}
 \end{multline}
first obtained in \cite{M} in a more general vectorial setting. 

\medskip

As far as uniqueness is concerned, formulation \eqref{eq.rpm} splits the distributional gradient of $u$ into a part $\sigma$ which is unique and a part $p$ where non-uniqueness may occur. The uniqueness of $\sigma$ is straightforward by strict convexity of $\sigma\mapsto\frac12 |\sigma|^2$ and sub-additivity of $p\mapsto |p|$.
 
We already know that uniqueness cannot hold in this general two-dimensional setting. In \cite[Subsection 1.2]{BF} we produced an example with drastic non-uniqueness for a trapezoidal domain with both Dirichlet and Neumann (smooth) boundary  conditions. Further the example demonstrates that regularity beyond $BV$ is false for the minimizers since their gradient can, for example, have  non-zero Cantor parts. This should be contrasted with the case of  total variation type minimization problems or, more generally, with one-homogeneous, strictly convex variational problems with linear growth (see {\it e.g.} \cite{Miranda,JMN,BS,Mercier}) for which some statement of regularity or uniqueness can be vindicated. 

Here, the necessary and sufficient Euler-Lagrange conditions read as
 $$\begin{cases}
{\rm div}\sigma=0 & \text{ in }\O,\\
\sigma\cdot\nu=g  & \text{ on }\pa_N\O,\\
|\sigma|\leq 1 & \text{  in }\O,\\
Du=\sigma+p & \text{ in }\O, \\
p=(w-u)\nu \HH^1& \text{ on } \partial_D\O,\\
|p|=\sigma\cdot p  & \text{ in }\O\cup \partial_D\O,
\end{cases}
$$
 where, in the last equality, the duality pairing $\sigma\cdot p$ has to be interpreted in a suitable measure theoretic way (see \eqref{def_Sigma_p} below). Formally, the previous equations imply that $p=\sigma|p|$ and thus, that $Du=\sigma \mu$ where $\mu=|p|+\LL^2$. A natural conservation law for $\mu$ arises by taking the curl of the previous equality. It leads to the following continuity equation 
 $${\rm div}(\sigma^\perp\mu)=0$$
 satisfied by $\mu$. It suggests that hyperbolic methods should give information on the behavior of the measure $\mu$ (hence on $p$ and $u$), at least along characteristic lines which are solutions to the ODE
\begin{equation}\label{eq:char-curve}
\dot \gamma(t)=\sigma^\perp(\gamma(t)), \quad t \in \R.
\end{equation}
One of the main issues is to lend a meaning to the concept of solution to the previous ODE since the vector field $\sigma$ is not regular. One can establish $H^1_{\rm loc}$-regularity for $\sigma$, a regularity that takes us beyond the classical Cauchy-Lipschitz or Cauchy-Peano theories. In this Sobolev setting, one could appeal to the more sophisticated Lagrangian flow theory originally initiated in \cite{DPL} and later developed in \cite{Amb,CD}. Unfortunately, this theory does not apply either because of a lack of control of the so-called compressibility constant which roughly corresponds to the quantity $\|{\rm curl}\sigma\|_\infty=\|{\rm div} \sigma^\perp\|_\infty$. 

In our previous paper \cite{BF}, we focussed on the set $\{|\sigma|=1\}$ assuming that it has non-empty interior. In that case, an additional conservation law for $\sigma$, typical in micromagnetism, arises
$${\rm div}\sigma=0, \quad |\sigma|=1.$$
 The use of entropy methods as in \cite{DKMO,JOP} allows one to exhibit characteristic line segments in $\{|\sigma|=1\}$ along which both $\sigma$ and $u$ must remain constant at least locally. We also give a very precise structure of that set, proving that any convex open subset $\O'$ of $\{|\sigma|=1\}$ decomposes into countably many pairwise disjoint open fans  $\mathbf F_{\hat z}$ -- the intersection of an open cone with $\O'$; see \eqref{eq:fan} below for a precise definition -- with an apex $\hat z$ on $\pa \O$ on which $\sigma$ behaves like a vortex centered at $\hat z$, {\it i.e.} 
$$\sigma(x)=\pm\frac{(x-\hat z)^\perp}{|x-\hat z|},$$
together with pairwise disjoint closed convex sets $\mathbf C$ on which $\sigma$ is continuous except on exceptional line segments $S$ that must lie on the boundary of  $\mathbf C$. Assuming the convexity of the set $\{|\sigma|=1\}$ we were able to provide a complete description of the distribution of those characteristic lines. See \cite[Theorem 1.3]{BF} or Theorem \ref{thm:structOp}  below for a summary of these results.
 
In this paper, we  adopt a more global viewpoint. We denote by $\Opl$ the interior of the set where $|\sigma|=1$ which we call {\em the (possibly) plastic region}, and by $\Oel=\O \setminus \overline \O_{pl} $ the {\em elastic region} (see \eqref{eq:O1} and \eqref{eq:OplOel} for precise definitions). We additionally (and admittedly restrictively) assume that 
\begin{equation}
\label{eq.def.opl}
\{|\sigma|=1\}\mbox{ is convex,}
\end{equation}
 so that $\Opl$ is convex as well. Under those assumptions $\Opl$, if not empty, is precisely the set to which the results of \cite{BF} apply.

We first exhibit in Section \ref{sec.ex} two examples for which $\Opl$ is either empty, or else a convex set strictly contained in $\O$.  In both cases uniqueness holds. Then, in Subsection \ref{subsec.empty}, we
quickly dispatch the case for which $\{|\sigma|=1\}$ is a convex set with empty interior (that is a line segment).
Subsection  \ref{sec.nei} recalls and refines the geometry of the characteristics in $\Opl$ which was analyzed in details in our previous paper. In Proposition \ref{prop:level-set} we establish that those characteristics are precisely the level sets of the minimizers in $\Opl$. 

Section \ref{sec.con} is entirely devoted to the already mentioned continuity results for $\sigma$ and $u$ (Theorems \ref{thm:cont_sigma} and \ref{thm:cont-u}). The continuity of $\sigma$ might be deemed predictable because, as will be seen in Lemma \ref{lem:whatinOel} and Theorem \ref{thm:struct-sigma-u}  below, $\sigma$ is continuous on both sets $\Oel$ and $\Opl$. Yet it is not so obvious that $\sigma$, and not only $\sigma\cdot \nu$, should be continuous across their shared boundary with normal $\nu$. The partial continuity of $u$ is, for its part, completely new, especially in view of the lack of regularity of the Dirichlet boundary condition $w$ along $\pa\O$.
 
Section \ref{sec.pdbc} is devoted to the proof of  the uniqueness of the minimizer of  \eqref{eq.rp} in the case of purely Dirichlet boundary conditions, that is when $\pa_D\O=\pa\O$, this under the assumption of convexity of $\Opl$ and the additional restriction that $\O$ be  $\C^{1,1}$. This is the object of Theorem \ref{thm.uniq}.  It is clear that the assumption on $\O$ is of a technical nature. Whether   the convexity assumption on $\Opl$ is equally so is unclear to us at present, but, as already alluded to, it conceptually fits  the usual barrier assumptions for uniqueness in the least gradient problem.

The final Section \ref{sec.cf} is devoted to a description of the geometry of the characteristics curves which,
thanks to the continuity of the stress $\sigma$, are well defined -- while maybe not unique -- on the whole of $\O$, this by the Cauchy-Peano theorem. We specifically investigate how such characteristics end up crossing, or not, the boundary $\Sigma$ between $\Oel$ and $\Opl$.  That last section is to be viewed as an effort toward a completion of the classification of the characteristic curves initiated in \cite{BF}.

\medskip

Summing up, the originality of this work is in our opinion twofold. On the one hand,  new regularity results  for minimizers of \eqref{eq.rpm} -- or, equivalently, for solutions of the scalar Hencky plasticity problem \eqref{eq.rp} -- are derived; those go beyond the classical $H^1_{\rm loc}$-regularity of $\sigma$. In Theorem \ref{thm:cont_sigma}, we establish the continuity of $\sigma$, except maybe at two single points in the interior of $\O$; those correspond to  non-differentiability points of the interface $\Sigma$ between the elastic and plastic parts of the domain which are not crossed by characteristic line segments. For its part, Theorem \ref{thm:cont-u}  establishes the continuity of the minimisers $u$  at all points of $\O$ swept by characteristic lines intersecting $\Sigma$. Both results  heavily rely on  the hyperbolic structure undergirding the problem. On the other hand, we prove the first generic uniqueness results in the case of  pure Dirichlet boundary data (Theorems \ref{thm:BF1} and \ref{thm.uniq});  again, those go beyond related results in the least gradient setting. 

Throughout, spatial hyperbolicity is key, resulting in  a surprising interplay between variational and hyperbolic structures.

\section{Notation and preliminaries}\label{sec:not}

The Lebesgue measure in $\R^n$ is denoted by $\mathcal L^n$  and the $s$-dimensional Hausdorff measure by $\mathcal H^s$.

From here onward the space dimension is set to $2$. If $a$ and $b \in \R^2$,  $a \cdot b$ denotes the Euclidean scalar product, and  $|a|:=\sqrt{a \cdot a}$. Further, we denote by $a^\perp=(-a_2,a_1)$ the rotation of $a=(a_1,a_2)$ by  $\pi/2$. The open (resp. closed) ball of center $x$ and radius $\rho$ is denoted by $B_\rho(x)$ (resp. $\overline B_\rho(x)$). 

\medskip

Departing from convention, we find it  most convenient to denote throughout by $(a,b)$ one of the line segments  $[a,b[$, $]a,b]$ or $[a,b]$.

\medskip

In all that follows, $\Omega \subset \R^2$ is {(at the least)} a bounded and Lipschitz open set, $\partial_D\O \subset \partial\O$ is open in the relative topology of $\partial\O$, and $\partial_N\O:=\partial\O \setminus \partial_D \O$. We use standard notation for Lebesgue and Sobolev spaces. For $X$ a locally compact set in $\R^2$, we denote by $\mathcal{M}(X;\R^2)$ (resp. $\M(X)$)  the space of bounded Radon measures in $X$ with values in $\R^2$ (resp. $\R$), endowed with the  norm $|\mu|(X)$, where $|\mu|\in \mathcal{M}(X)$ is the variation of the measure $\mu$. The space $BV(\Omega)$ of functions of bounded variation in $\Omega$ is made of all functions $u \in L^1(\Omega)$ such that their distributional gradient $Du\in \mathcal M(\Omega;\R^2)$. Sobolev embedding shows that $BV(\O)\subset L^2(\O)$. 

\medskip

We recall that, if $\O$ is bounded with Lipschitz boundary and $\sigma \in L^2(\O;\R^2)$ with $\Div \sigma \in L^2(\O)$, its normal trace, denoted by $\sigma\cdot\nu$, is well defined as an element of $H^{-1/2}(\partial \Omega)$. If further $\sigma \in L^\infty(\Omega;\R^2)$, then $\sigma\cdot \nu \in L^\infty(\partial \Omega)$ with $\|\sigma\cdot\nu\|_{L^\infty(\partial \O)}\leq \|\sigma\|_{L^\infty(\O;\R^2)}$ (see {\it e.g.}~\cite[Theorem 1.2]{A}). 

\medskip

According to \cite[Definition 1.4]{A} and \cite[Section 6]{FG}, we define a generalized notion of  duality pairing between $\sigma$ and a measure $p$ as follows:
\begin{definition}
\label{def_Sigma_p}
Let $\O$ be a bounded open set with Lipschitz boundary, $\pa_D\O$ be a relatively open subset of $\pa\O$ and $\partial_N\O=\partial\O \setminus \partial_D\O$. For every $\sigma \in L^\infty(\Omega;\R^2)$ with $\Div\sigma \in L^2(\O)$, $(u,e,p) \in BV(\Omega) \times  L^2(\Omega;\R^2) \times  \mathcal{M}(\Omega \cup \partial_D\O;\R^2)$ and $w \in W^{1,1}(\O)$ such that $D u = e+p$ in $\O$ and $p=(w-u)\nu \HH^{1}$ on $\partial_D \O$, we define the distribution $\left[\sigma\cdot p \right]\in \mathcal D'(\R^2)$ by
\begin{multline}\label{eq:duality}\langle\left[\sigma\cdot p \right], \p\rangle =  \int_\O \p\sigma\cdot (\nabla w-e)\; dx  
+ \int_\O (w-u)\sigma \cdot\nabla\p\; dx\\
+\int_{\Omega} (w-u) (\Div \sigma) \p \, dx +\int_{\partial_N\O}(\sigma\cdot\nu) (u-w)\p\, d\HH^1 \quad \text{ for all } \p\in \mathcal{C}^{\infty}_c(\R^2).
\end{multline}
\end{definition}

By appropriate smooth approximation of $\sigma$ through local translations and convolution (see {\it e.g.} \cite[Section 6]{FG} with the additional simplification that $\sigma\cdot\nu \in L^\infty(\partial\O)$), it could be  shown  that $[\sigma\cdot p]$ is actually a bounded Radon measure  in $\R^2$ satisfying
\begin{equation}\label{eq:convex-ineq}
|[\sigma\cdot p]| \leq \|\sigma\|_{L^\infty(\O;\R^2)} |p| \quad \text{ in } \mathcal M(\R^2)
\end{equation}
and {with total mass obtained} by taking $\p\equiv 1$ in \eqref{eq:duality} (see \cite{FG}). Moreover, if $\sigma \in \mathcal C^0(\O;\R^2)$,
$$\langle\left[\sigma\cdot p \right], \varphi\rangle =  \int_\O \varphi \sigma \cdot\, dp=\int_\O \varphi \sigma \cdot \frac{dp}{d|p|} \, d|p|\quad \text{ for all }\varphi\in \C_c(\O),$$
where $\frac{dp}{d|p|}$ stands for the Radon-Nikod\'ym derivative of $p$ with respect to its variation $|p|$. For all of the above, see \cite[Section 6]{FG} in the vectorial case.

\medskip

Finally we establish the following result which will be used several times throughout. It is a direct application of  {\it e.g.} \cite[Theorem 1]{JK}.

\begin{lemma}\label{lem.u=cst}
Let $U \subset \R^2$ be a {bounded open domain}, and  $u$ be a 1-Lipschitz harmonic function in $U$ such that
$$\left|\partial_\nu u\right|=1 \quad \HH^1\text{-a.e. on }\Gamma,$$
where $\Gamma$ is a relatively open and connected {Lipschitz} subset of $\pa U$. Then $u$ remains constant on $\Gamma$.
\end{lemma}

\begin{proof} 
Because $|\nabla u|\le 1$ in $U$ and $|\partial_\nu u|= 1$ on $\Gamma$, \cite[Theorem 1]{JK}) implies that the tangential derivative $\partial_\tau u$, which exists $\h$-a.e. on {$\Gamma$} because {$\Gamma$} is a Lipschitz curve, is $0$ $\h$-a.e. on $\Gamma$.  Indeed, we deduce from that theorem that, as $\e \to 0$, 
$$\begin{cases}
\partial_\tau u\left(x-\e\nu(x)\right) \to \partial_\tau u(x)\\
\partial_\nu u\left(x-\e\nu(x)\right) \to \partial_\nu u(x)
\end{cases}
\qquad \text{ for }\h\mbox{-a.e.  }x \in \Gamma.$$
But, since $\Gamma$ is a Lipschitz curve, the generalized area formula  (see {\it e.g.} \cite[Theorem 2.91]{AFP}) implies that $\h(u({\Gamma}))=0$. Since $u$ is continuous on $\overline\O$, hence on $\Gamma$, this cannot happen unless $u$ is constant on $\Gamma$. 
\end{proof}

\section{Statement of the problem}\label{sec.ex}

\subsection{The minimization problem and the Euler-Lagrange equations}

 Assume that $w \in W^{1,1}(\R^2)$ -- so that its trace on $\pa_D\O$ belongs to $L^1(\pa_D\O)$ -- and $g \in L^\infty(\pa_N\O)$ satisfying \eqref{eq.sl}. Recalling the introduction, we consider the following minimization problem
\begin{equation}\label{eq:minim} 
\min \Big\{ \frac12 \int_\O |\sigma|^2\, dx + |p|(\O\cup \partial_D\O)-\int_{\pa_N\O}gu\, d\h : \; (u,\sigma,p) \in \mathcal A_w\Big\}
\end{equation}
{where 
\begin{multline*}
\mathcal A_w:=\{(u,\sigma,p) \in BV(\O) \times L^2(\O;\R^2) \times \M(\O\cup\partial_D \O;\R^2): \\
Du=\sigma+p \text{ in } \O, \ p=(w-u)\nu \HH^1\text{ on } \partial_D \O\}.
\end{multline*}
is an affine subspace of $BV(\O) \times L^2(\O;\R^2) \times \M(\O\cup\partial_D \O;\R^2)$, hence weak* closed in that space.}

The direct method in the calculus of variations ensures the existence of solutions $(u,\sigma,p) \in \mathcal A_w$ to the variational problem \eqref{eq:minim}. Moreover, it is shown in \cite[Section 3]{BF} that the Euler-Lagrange equations associated with the relaxed minimization problem \eqref{eq:minim} are the following necessary and sufficient first order conditions:
\begin{equation}
\label{eq:plast}
\begin{cases}
{\rm div}\sigma=0 & \text{ in }{H^{-1}(\O)},\\[1mm]
\sigma\cdot\nu=g  & \text{ a.e. on }\pa_N\O,\\[1mm]
|\sigma|\leq 1 & \text{ a.e. in }\O,\\[1mm]
Du=\sigma+p & \text{ in }\M(\O;\R^2), \\[1mm]
p=(w-u)\nu \HH^1& \text{ in } \M(\partial_D \O;\R^2),\\[1mm]
|p|=[\sigma\cdot p]  & \text{ in }\M(\O\cup \partial_D\O).
\end{cases}
\end{equation}
Further, $\sigma$ is unique and, by \cite[Theorem 4.1]{BF}, itself a revisiting of \cite[Theorem 4.1]{BeF}  and of \cite{S},
\begin{equation}\label{eq.reg-sigma}
\sigma\in H^1_{\rm loc}(\O;\R^2).
\end{equation}
Note that  \eqref{eq:plast}  and \eqref{eq.reg-sigma} have been derived in \cite[Section 3]{BF} and \cite[Theorem 4.1]{BF} respectively, under the assumptions that  $\pa_N\O=\emptyset$ and that   $w\in H^1(\R^2)$. However, an inspection of that proof shows that \eqref{eq:plast}  holds in our present setting while \cite[Theorem 4.1]{BF}, hence \eqref{eq.reg-sigma}, still holds true since this result is local.

\begin{remark}
Observe that $u$ is a solution of the original variational problem \eqref{eq.rp}  if and only if $(u,\sigma,p)\in \mathcal A_w$ is a solution of \eqref{eq:minim}. Indeed, using the infimum convolution formula \eqref{eq:infconv}, together with a measurable selection argument, for any $v\in BV(\O)$, the corresponding  $\tau$, $q$ such that $(v,\tau,q)\in \mathcal A_w$ are obtained through
$$\begin{array}{rcl}W(\nabla v(x)) &=& \frac12|\tau(x)|^2 + |\nabla v(x)-\tau(x)| \quad \LL^2\text{-a.e. in }\O,
\\[2mm]q&=&(\nabla v -\tau) \LL^2\res\O +D^s v+ (w-v)\nu \HH^1\res \partial_D\O.\end{array}$$
\vskip-.7cm\hfill\P
\end{remark}

\medskip

For linguistic convenience and because, as explained in \cite[Section 1]{BF}, \eqref{eq:minim} can be viewed as a problem of Hencky plasticity, we will refer to {$u$ as the {\em displacement}},  $\sigma$ as the {\em stress} and  $p$ as the {\em plastic strain} throughout the rest of this paper. Hereafter,  the last equation of \eqref{eq:plast} will be labeled the {\it flow rule} and it will appear at various places in different forms.

\begin{remark}\label{rmk:jump-flow-rule}
The following holds true:

(i) Exactly  as in \cite[Lemma 3.8]{FG},  if $\Gamma \subset \overline \O$ is locally the graph of a Lipschitz function, then $\sigma\cdot \nu \in L^\infty (\Gamma)$ and
$$[\sigma\cdot p]\res \Gamma=(\sigma\cdot\nu) (u^+-u^-) \HH^1 \res \Gamma,$$
where $u^+$ and $u^-$ are the {one-sided} traces of $u$ on $\Gamma$ oriented by the normal unit vector $\nu$, and $\sigma\cdot\nu$ is the normal trace of $\sigma$ on $\Gamma$. Thus, the flow rule {(the last equation of \eqref{eq:plast})} localized on $\Gamma$ reads
$$(\sigma\cdot\nu) (u^+-u^-)=|u^+-u^-| \quad \HH^1\text{-a.e. on }\Gamma.$$
Since by definition $u^+ \neq u^-$ on $J_u$, {the jump set of $u$}, we infer that $\sigma\cdot\nu=\pm 1$ $\HH^1$-a.e. on $\Gamma \cap J_u$. This applies also if $\HH^1(\Gamma\cap\partial_D\O)>0$, replacing $u^+$ by $w$ on that part of $\Gamma$.

\smallskip

(ii) Since $\sigma \in H^1_{\rm loc}(\O;\R^2)$,  it admits a precise representative defined $\rm{Cap }_p$-quasi everywhere for any $p<2$ hence $\HH^s$-almost everywhere in $\O$ for any $s>0$ (see {\it e.g.} \cite[Sections 4.7, 4.8]{EG}). In the sequel we will identify $\sigma$ with its precise representative which is thus defined outside a Borel set {$\mathcal N \subset \O$} of zero Hausdorff dimension.

\smallskip

(iii) Note also that, if $\Gamma \subset \O$ is as in (i), the normal trace $\sigma\cdot\nu$ coincides $\HH^1$-a.e. on $\Gamma$ with {the scalar product in $\R^2$} of (the trace of) $\sigma$ with the normal $\nu$ to $\Gamma$ since $\sigma \in H^1_{\rm loc}(\O;\R^2)$.

\smallskip

(iv) Arguing as in \cite{A2,DMDSM,FGM,BM}, it is possible to express the flow rule (the last equation of \eqref{eq:plast}) by means of the quasi-continuous representative of the stress, still denoted by $\sigma$, which is $|p|$-measurable. We get
\begin{equation}\label{eq:frptwise}
\sigma (x) \cdot \frac{dp}{d|p|}(x)=1 \quad \text{for $|p|$-a.e. $x\in \O$}
\end{equation}
or still $p=\sigma |p|$  in $\M(\O)$.
\hfill\P
\end{remark}

\subsection{The elastic and plastic regions}

Let us define the {\it saturation set} as
\begin{equation}\label{eq:O1}
\O_1:=\{x \in \O \setminus \mathcal N : \; |\sigma(x)|= 1\},
\end{equation}
where $\mathcal N \subset \O$ is the exceptional set of zero Hausdorff dimension of Remark \ref{rmk:jump-flow-rule}-(ii). Because $\sigma$ is unique, this set is well defined. In the remainder of this paper we  assume that $\O_1$  satisfies the following 

\bigskip

\noindent{\bf Hypothesis (H).} \qquad\qquad\qquad\emph{The saturation set $\O_1$ is convex.}

\bigskip
 
We next define 
\begin{equation}\label{eq:OplOel}
\Opl:={\rm int}(\O_1), \quad \Oel:=\O \setminus \overline\O_1.
\end{equation}

Under hypothesis {\bf (H)}, the set $\Opl$ is a (possibly empty) convex open set. It is  henceforth  referred to as the {\it (possibly) plastic set}, although it is not to be confused with the the support of the measure $p$ (see {\it e.g.} Theorem \ref{thm:BF1} and Example \ref{ex:1} below). 

\begin{remark}
As a consequence of Lemma \ref{lem:whatinOel} and of Theorem 
\ref{thm:struct-sigma-u} below, the set $\mathcal N$ must lie in $\partial\Opl\cap\O$. It will actually be proved in Theorem \ref{thm:cont_sigma} that that set consists of at most two points.
\hfill\P
\end{remark}

The open set $\Oel$ is similarly  referred to as the {\it elasticity set} because, whenever it is not empty, all solutions are purely elastic in $\Oel$. Indeed, 

\begin{lemma}\label{lem:whatinOel}
If $\Oel \neq \emptyset$, then ${u\in \C^\infty(\Oel)}$ is harmonic, {$\sigma=\nabla u$} and $|\nabla u|<1$ in $\Oel$.
\end{lemma}

\begin{proof}
Observe that $\Oel \subset \mathcal N \cup \{x \in \O \setminus \mathcal N : \; |\sigma(x)|<1\}$. Since $\HH^1(\mathcal N)=0$  and $u\in BV(\O)$, then {$|p|(\mathcal N) \leq \int_{\mathcal N}|\sigma|\, dx + |Du|(\mathcal N)=0$}. Moreover, using the flow rule \eqref{eq:frptwise},  $|p|(\{x \in \O \setminus \mathcal N : \; |\sigma(x)|<1\})=0$. As a consequence, $|p|(\Oel)=0$ and thus  $Du=\sigma\LL^2$ in $\Oel$. Since $\Oel$ is a (nonempty) open set, we infer that $u \in H^1(\Oel)$ with $\nabla u=\sigma$ a.e. in $\Oel$. Using the first equation in \eqref{eq:plast}, we deduce that $u$ is harmonic in $\Oel$. In particular, $u$ (and thus $\sigma=\nabla u$ as well) is smooth in $\Oel$, hence $\mathcal N \cap \Oel=\emptyset$ and $|\nabla u|<1$ in $\Oel$.
\end{proof}

\begin{remark}
Independently of our convexity hypothesis {\bf (H)}, \cite[Theorem 2.1.2]{FS}, or, to be exact, its scalar analogue, shows the existence of an open set $\O_0 \subset \O$ such that $u \in \C^{0,\alpha}(\O_0)$ for some $\alpha \in (0,1)$, $|\sigma|<1$ in $\O_0$ and $|\sigma|=1$ $\LL^2$-a.e. in $\O \setminus \O_0$. As a consequence of the flow rule \eqref{eq:frptwise} expressed in a pointwise form, we get that $|p|(\O_0)=0$ and thus,
$$
\begin{cases}
p=0 & \text{ in }\O_0,\\
Du=\sigma \LL^2 & \text{ in }\O_0,\\
\Delta u=0 & \text{ in }\O_0.
\end{cases}
$$
The function $u$ being harmonic in $\O_0$, we infer that $u \in  \C^\infty(\O_0)$. 

In view of  assumption {\bf (H)},   the same conclusion is ensured in our setting without appealing to that theorem.
\hfill\P
\end{remark}

\begin{remark}
The generalized least gradient problem (see {\it e.g.} \cite{JMN}) consists in minimizing functionals of the form   
$$u \in BV(\O) \mapsto \int_\O \varphi(\nabla u)\, dx$$
where $\p$ is a given norm, say in $\R^2$,
under a prescribed trace $g \in \C^0(\partial\O)$. A natural condition which ensures uniqueness of a minimizer with that trace is a so-called {\it barrier condition}. Roughly speaking this condition ensures the positivity of a generalized mean curvature of $\partial\O$ related to $\varphi$. This condition is in turn related to the strict convexity of $\O$. 

In our case, the integrand $W$ only behaves like a norm for $|\nabla u|\geq 1$, that is  where the inequality constraint on $\sigma$ is active, {\it i.e.} $|\sigma|=1$. It is natural to expect some sort of convexity property of $\O_1$. Unfortunately, in contrast with to \cite{JMN} where the assumption is solely on the set $\O$,  our case requires  assumption {\bf (H)} on the {\it a priori} unknown set $\O_1$.  We do not know if {\bf (H)} is satisfied in general but do not have a counterexample.
\hfill\P
\end{remark}

\medskip

Although the stress $\sigma$ is always unique, it is not in general so for the displacement $u$ and the plastic strain $p$. This has been evidenced in \cite[Subsection 1.2]{BF} where infinitely many solutions are constructed for a mixed boundary value problem for which $\Opl=\O$. We now complement this example by producing a unique plastic strain (hence a unique solution $(u,\sigma,p)$ to the minimization problem \eqref{eq:minim}), but with a region $\Opl$ which is {a convex open set {\it strictly} contained in $\O$}. Boundary conditions can be picked to be pure Dirichlet, or of mixed type. 
\begin{figure}[htbp]
\scalebox{1}{\begin{tikzpicture}

 \draw[style=thick]  (5,0) arc (0:180:5) ; \draw[color=black] (3,0) arc(0:180:3) ; \draw[color=blue][<->] (5.1,0) arc (0:180:5.1) ;
\fill[color=gray!20] (3,0) arc(0:180:3) ;

\draw [color=black!80](-1.65,1.2)node{\framebox{\tiny{{\it plastic region}}}};
\draw[style=thick] (-5,0) -- (5,0) ;

 \draw(2.7,1.3) node[right]{$r=a$};
\draw(2.9,3) node[right]{$r=b$};

\draw [color=red][style=thick] (-3,0)-- (2.1213,2.1213);
\draw [style=thin] (0,0)-- (2.1213,2.1213);

\draw[color=blue][<->] (-5,-.1)--(-3,-.1);
\draw[color=blue][<->] (-3,-.3)--(5,-.3);

\draw (-1,1.9)node {\color{red}$\sigma=\vec{e}_{\frac\theta2}$};
\draw (-1,4)node {\color{red}$u_{el}=2\sqrt{ar}\sin{\frac\theta2}$};
\draw (-1,3.4)node {\color{red}$\sigma=\nabla u_{el}=\sqrt{a/r}(\sin{\frac\theta2},\cos{\frac\theta2})$};
\draw( .9,1) node{\color{red}$u\equiv 2a\sin{\frac\theta2}$};
\draw[color=red][->](-.1,.95)--(-.3,1.1);

\draw[color=blue] (-5,-.6) node[right] {$\displaystyle\frac{\partial u}{\partial\theta}=0$};
\draw[color=blue] (0,-.5) node[right] {$u=0$};
\draw[color=blue] (4.8,2.5) node[right] {$u=2\sqrt {ab} \sin{\theta/2}$};
\draw (0,0) node[above]{O}; \draw (0,0) node {$\bullet$} ;

\draw (.5,.2) node {\small$\theta$}; \draw (0.3,0) arc (0:55:.2) ;
\draw (-1.95,.18) node {\small$\theta/2$}; \draw (-2.4,0) arc (0:22:.6) ;

\end{tikzpicture}}\caption{The MacClintock example.}
\label{fig:mcC}
\end{figure}
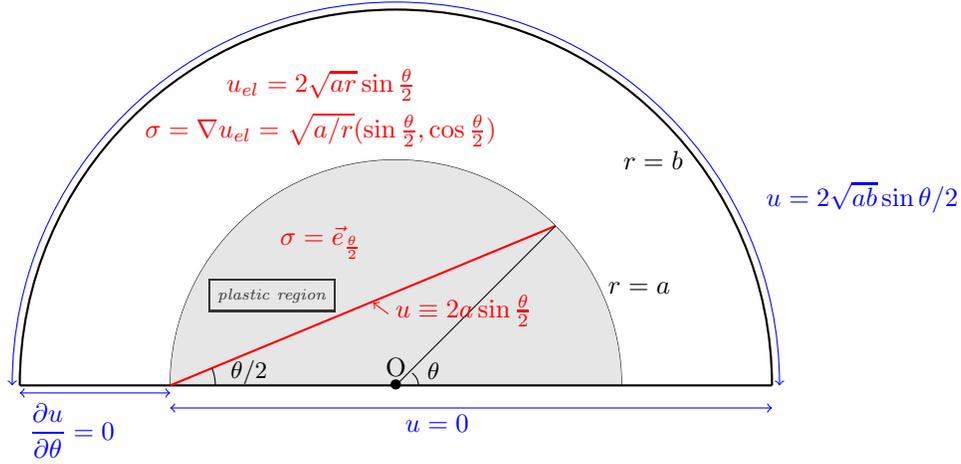

\begin{example}[{\bf The modified MacClintock example}]\footnote{This example is motivated by the solution given by F. A. MacClintock to the elasto-plastic field around the crack tip of a semi-infinite straight crack in so-called mode III.}\label{ex.Mc}
{\rm Consider a half-disk of radius $b$ with  boundary conditions as shown on Figure \ref{fig:mcC}. Note that the boundary condition on $\theta=\pi$ switch from homogeneous Neumann for $a\le r\le b$ to homogeneous Dirichlet for $r\le a$. The solution is elastic ($|\sigma(x)|<1$) in the half-annulus $a<r<b$. Then the plastic region is $r\le a$ and it corresponds to a fan centered at $(r=a,\theta=\pi)$, so that the associated stress field is $\sigma= \vec{e}_{\frac\theta2}$ in terms of the angle $\theta$. According to Remark \ref{rmk:jump-flow-rule}-(i) there can be no jump at $r=a$ because the normal $\vec{e}_r$ is not alined with $\sigma=\vec{e}_{\frac\theta2}$. Further, $u$ must remain constant along the characteristic line segment given in polar coordinates by $](a,\pi),(a,\theta)[$ (see \cite[Theorem 6.2]{BF}). Thus it is equal to $2a\sin{\frac\theta2}$ along that line segment. The solution is therefore unique.

Remark that we could have imposed a Dirichlet boundary condition on the line segment $](b,\pi),(a,\pi)[$ in lieu of the Neumann boundary condition. That boundary condition should then be $w=2\sqrt{ar}$ and there would thus be a jump in $w$ at $(r=a,\theta=\pi)$. By Gagliardo's theorem $w$ would still be the trace on $\pa \O$ of a  $W^{1,1}(\R^2)$ function so that  our analysis equally applies to the pure Dirichlet case.\hfill\P}
\end{example}

\subsection{When the saturation set has empty interior}\label{subsec.empty} 

In that case, the solution is purely elastic inside the full domain $\O$, except possibly on a segment separating $\O$ into two connected components.

\begin{theorem}\label{thm:BF1}
{Let $\O \subset\R^2$ be a bounded open set with Lipschitz boundary,} $w\in L^1(\pa \O)$ and $u \in BV(\O)$ be a solution of \eqref{eq.rp}. Under hypothesis {\bf (H)}, if, further, $\O_1$ has empty interior in $\O$, then 
\begin{itemize}
\item either $\O_1=\emptyset$ and $u \in \C^\infty(\O) \cap \C^{0,1}(\overline{\O})$ satisfies
\begin{equation}\label{eq:eq1}
\begin{cases}
-\Delta u= 0 & \text{ in }\O,\\
|\nabla u|< 1 & \text{ in }\O,\\
\partial_\nu u=g & \text{ on }\partial_N\O,\\
(w-u)\partial_\nu u=|w-u| & \text{ on }\partial_D\O;
\end{cases}
\end{equation}
\item or $\O_1=S$ for some open line segment $S$ separating $\O$ into two connected components denoted by $\O^\pm$.
{Then, $u \in \C^\infty(\O^\pm) \cap \C^{0,1}(\O^\pm \cup S)$, and denoting by $u^\pm$ the one-sided traces of $u$ on $S$, it satisfies}
\begin{equation}\label{eq:eq2}
\begin{cases}
-\Delta u= 0 & \text{ in }\O^\pm,\\
|\nabla u|< 1 & \text{ in }\O^\pm,\\
u^\pm \text{ are constant on }S,\\
\partial_\nu u=g & \text{ on }\partial_N\O,\\
(w-u)\partial_\nu u=|w-u| & \text{ on }\partial_D\O.
\end{cases}
\end{equation}
\end{itemize}
In both cases, if in addition $\partial_D\O=\partial\O$, then $u$ is unique.
\end{theorem}

\begin{proof}
If $\O_1$ is a convex set with empty interior, there exists a line segment $S$ such that $\O_1=S \cap \O$. Without loss of generality, we assume that $0 \in S$ and denote by  $\nu$ a (constant) unit vector orthogonal to $S$. {Let $L$ be the connected component of $(\nu^\perp \R) \cap \O$ containing $\O_1$ and  $\O^\pm$ be the two connected components of $\O \setminus L$.} By definition, $\Oel=\O \setminus \overline S$ and, by Lemma \ref{lem:whatinOel}, $u \in \C^\infty(\O \setminus \overline S)$ is harmonic in $\O \setminus \overline S$. Moreover, 
$$Du=\nabla u \LL^2 + (u^+-u^-)\nu \HH^1\res S,$$
where $u^\pm$ are the one-sided traces of $u$ on $L \cap \O$ according to this orientation. Further, {by Remark \ref{rmk:jump-flow-rule}-(i)}, the flow rule yields
\begin{equation}\label{eq:flbis}
(\sigma\cdot\nu) (u^+-u^-)=|u^+-u^-| \quad \HH^1\text{-a.e. on }S.
\end{equation}
Since $u \in BV(\O) \subset L^2(\O)$ and $D u\res \O^\pm=\sigma\LL^2\res \O^\pm$ with $\sigma \in L^2(\O;\R^2)$, then $u \in H^1(\O^\pm)$ with $\nabla u=\sigma$ in $\O^\pm$, and $u^\pm \in H^{1/2}(L \cap \O)$. Since $\O^\pm \subset \Oel$,  Lemma \ref{lem:whatinOel} implies that $|\nabla u|=|\sigma|<1$ in $\O^\pm$, hence $u\in W^{1,\infty}(\O^\pm)$. Thus $u$ can be extended by continuity to a  $u \in \C^{0,1}_{\rm loc}({\O}^\pm\cup L)$-function. In particular, the traces $u^\pm$ of $u$ are continuous on {$L$} and thus, the set 
$$J=\{ x \in {L} : \; u^+(x) \neq u^-(x)\},$$
which is included in $\ol S$, is relatively open in {$L$}, hence included in $\O_1=S\cap\O$. 

\medskip

If $J \neq \emptyset$, { $J=\bigcup_{j \in \N} J_j$ where $\{J_j=\; ]\alpha_j,\beta_j[\}_{j \in \N}$ are pairwise disjoints open line segments and the flow rule \eqref{eq:flbis}} yields $|\sigma\cdot \nu|=1$ on $J_j$.
Moreover, by definition \eqref{eq:O1} of $\O_1$, $|\sigma|=1$ on $\O_1$. We thus get that $\sigma=\pm\nu$ on $J_j$. Since $\sigma \in H^{1/2}(J_j;\R^2)$, \cite[Lemma A.2]{BF} yields $\sigma=\nu$ or $\sigma=-\nu$ on $J_j$. Then Lemma \ref{lem.u=cst} ensures that $u^+$ and $u^-$ are both  constant on $J_j$. By continuity of $u^\pm$ on {$L$} and using that $u^+(\alpha_j)=u^-(\alpha_j)$ and $u^+(\beta_j)=u^-(\beta_j)$, this is possible only if {$J=L$} and $u^\pm=c^\pm$ on {$L$}, for some constants $c^\pm \in \R$ with $c^+\neq c^-$. Hence \eqref{eq:eq2}.

\medskip

If instead $J=\emptyset$, we get that $u \in H^1(\O^\pm)$ satisfies $u^+=u^-$ on {$L$}. This ensures that actually $u \in H^1(\O)$. Moreover, since $\nabla u =\sigma \in H^1_{\rm loc}(\O;\R^2)$, we deduce that $u \in H^2_{\rm loc}(\O)$. We thus obtain from Green's formula that for all $\varphi \in \C^\infty_c(\O)$,
\begin{multline*}
\int_\O u\Delta \varphi\, dx = \int_\O \varphi \Delta u\,dx + \int_S \left(u^+(\partial_\nu \varphi) -\p(\partial_\nu u)^+\right) d\HH^1\\ -\int_S \left(u^-(\partial_\nu \varphi) -\p(\partial_\nu u)^-\right) d\HH^1= \int_{\O\setminus S} \varphi \Delta u\,dx=0.
\end{multline*}
We thus infer that $u$ is actually harmonic in all of $\O$, hence of class $\C^\infty$ on that set. 

Classical properties of harmonic functions show that $|\nabla u|^2$ is subharmonic, {\it i.e.} $-\Delta (|\nabla u|^2) \leq 0$ in $\O$. We already know that $|\nabla u|<1$ in $\Oel=\O \setminus \overline S$. Let us now consider $x_0 \in \overline S$ and $r>0$ small enough so that $B_r(x_0) \subset\subset \O$. The mean value property yields that
$$|\nabla u(x_0)|^2 \leq \fint_{B_r(x_0)} |\nabla u|^2\, dx.$$
Since $|\nabla u|<1$ $\LL^2$-a.e. in $\O$, we get that $\fint_{B_r(x_0)} |\nabla u|^2\, dx<1$, hence $|\nabla u(x_0)|<1$. Thus $|\nabla u|<1$ in all of $\O$, $\O_1=\emptyset$ and since $\O$ has Lipschitz boundary, we deduce that $u$ extends to a Lipschitz-continuous function in $\overline\O$. Hence \eqref{eq:eq1}.

\medskip

It remains to show the uniqueness of $u$ in the pure Dirichlet case $\partial_D\O=\partial\O$. In the  case where $\O_1=\emptyset$, we first notice that there exists an $\HH^1$-measurable $A \subset \partial\O$ such that $\HH^1(A)>0$ and $|\sigma\cdot\nu|<1$ $\HH^1$-a.e. on $A$. Else, using that $\sigma=\nabla u$, we get that 
\begin{equation}\label{eq:0930}
|\sigma\cdot \nu|=|\partial_\nu u|=1 \quad \HH^1\text{-a.e. on }\partial\O,
\end{equation}
and according to Lemma \ref{lem.u=cst}, $u$ must remain constant on $\pa\O$. But then the maximum principle implies that $u$ is constant on the convex (hence connected) set $\O$ which contradicts the starting assumption \eqref{eq:0930}. So that case does not occur. Let $u_1$ and $u_2$ be two solutions of  \eqref{eq:plast}. By uniqueness of $\sigma=\nabla u_1=\nabla u_2$, there exists a constant $c \in \R$ such that  $u_2=u_1+c$. As $|\sigma\cdot\nu|<1$ $\HH^1$-a.e. on $A$, the flow rule 
$$(w-u_i)\sigma\cdot\nu=|w-u_i| \quad \HH^1\text{-a.e. on }A$$
ensures that $u_1=u_2=w$ $\HH^1$-a.e. on $A$, which yields $c=0$ and $u_1=u_2$ in $\O$.

The second case $\O_1=S$ can be treated similarly, arguing separately on $\O^+$ and $\O^-$ {and using that $u^\pm$ are constant on {$S=L$}}.
\end{proof}

When $\O_1 =\emptyset$ Theorem \ref{thm:BF1} shows the uniqueness of the displacement $u$, hence of the plastic strain $p$ although the Dirichlet boundary condition might fail to be satisfied. We now give an example of a plastic strain concentrated on a set of $0$-volume on the Dirichlet boundary, showing that such a situation can indeed occur.

\begin{example}\label{ex:1}
{\rm Consider a circular annulus $\O$ of inner radius $a$ and outer radius $b$ subject to the following Dirichlet boundary conditions:
$$\begin{cases}
w(r=a)=\alpha\\w(r=b)=\beta
\end{cases}
$$ with $|\beta-\alpha|>a\ln(b/a)$, and look for a radially symmetric stress in \eqref{eq:plast}. The only possible divergence free stress is of the form $\sigma=\frac dr \vec{e}_r$. Then, for $|\sigma|$ to be less than or equal to $1$ on $\O$ we must have $|d|\le a$ so that, if some plasticity is desired, $|d|=a$. For $r>a$, we have $|\sigma|<1$, so $\sigma=\nabla u$ and $u(r)= \pm a\ln r+k(\theta)$. The boundary condition $u(r=b)=\beta$ must be met, so that $k(\theta)= \beta\mp a\ln b$.

According to Remark \ref{rmk:jump-flow-rule}-(i), at $r=a$, $\mp (\alpha- u(r=a))=|\alpha-u(r=a)|$ where $u(r=a)$ is the value of the elastic solution at $r=a$, that is $\pm a\ln{a/b}+ \beta$. 
We get 
$$\pm(\beta-\alpha\mp a\ln{b/a})=|\beta-\alpha\mp a\ln{b/a}|.$$
If $\beta-\alpha>a\ln(b/a)$, then $d=a$ while, if $\beta-\alpha<-a\ln(b/a)$, $d=-a$. In both cases $u$ jumps at $r=a$ with an associated $p=(u(a)-\alpha)\delta_{r=a}$.

Summing up, we have 
\begin{equation}\label{eq:sol-pb1}
u=\pm a \ln r/b+\beta,\quad p=(u(a)-\alpha)\delta_{r=a}
\end{equation}
as a solution to \eqref{eq:plast}. We claim that this is the only solution. Indeed, $u(r),\; r>a$ is unique since $\nabla u$ is unique and its value at $r=b$ is given. So, $p$ must concentrate on $r=a$ and, because of the flow rule (the last equation in \eqref{eq:plast}), it must be given by its expression in \eqref{eq:sol-pb1}.
}
\end{example}

\subsection{When the saturation set has non empty interior}\label{sec.nei}

From now on, we assume that $\O_1$ has nonempty interior so that the set $\Opl$ is a nonempty convex open set. In the rest of this work, we denote by
\begin{equation}\label{eq:def-Sigma}\Sigma:=(\partial\Oel \cap \partial\Opl) \cap \O
\end{equation}
the interface between the elastic and plastic parts.

\begin{lemma}
The sets $\Oel$, $\Opl$ and $\Sigma$  satisfy
\begin{equation}\label{def:interface}
\Oel=\O \setminus \overline \O_{pl} , \quad \Opl=\O \setminus \overline \O_{el} , \quad \Sigma=\O \cap \partial \Opl=\O \cap \partial \Oel.
\end{equation} 
\end{lemma}

\begin{proof}
Since $\O_1$ is convex, according to \cite[Proposition 3.45]{BC} and \eqref{eq:OplOel}, we have $\overline\O_1=\overline \O_{pl} $, $\Opl={\rm int}(\overline \O_{pl} )$ and $\Oel=\O \setminus \overline \O_{pl} $.

First, as $\O \setminus \overline \O_{el}  \subset \overline \O_{pl}  \cap \O$, then $\O \setminus \overline \O_{el} \subset {\rm int}(\overline \O_{pl} )=\Opl$. Moreover, if $x \in \O \cap \overline \O_{el}  \cap \Opl$, there exists $r>0$ such that $B_r(x) \subset \Opl$ and $y \in B_r(x)$ such that $y \in \Oel$. Then $y \in \Oel \cap \Opl$ which is impossible. This implies that $\Opl \subset \O \setminus\overline \O_{el} $, hence  $\Opl = \O \setminus\overline \O_{el} $.

Recalling that $\O \cap \overline \O_{pl}  = \O \setminus \Oel$ and $\O \setminus \Opl = \O \cap \overline \O_{el} $, we get that $\O \cap \partial \Opl = \O \cap \partial \Oel=\Sigma$.  
\end{proof}

{According to Theorems 5.1 and 5.6 in \cite{BF} (applied to $\O_p$ now denoted by $\Opl$),  the following rigidity properties of $u$ and $\sigma$ hold true in $\Opl$.

\begin{theorem}\label{thm:struct-sigma-u}
The stress $\sigma$ is locally Lipschitz in $\Opl$ and $\sigma$ is constant along all line segments
$L_y \cap \Opl$, where, for $y \in \Opl$, 
$$
L_y:=y+\R \sigma^\perp(y)
$$
is called a  {\it characteristic line}. Moreover, there exists an $\HH^1$-negligible set $Z \subset \Opl$ such that $\LL^2(\Opl \cap (\bigcup_{z \in Z} L_z))=0$ and $u$ is constant along $L_x \cap \Opl$ for all $x \in \Opl \setminus (\bigcup_{z \in Z} L_z)$.
\end{theorem}}

\begin{remark}\label{rem:cont-|sigma|-Sigma}
Since $\Opl$ is an open set with Lipschitz boundary and $\Sigma$ is an open subset of $\partial\Opl$ in the relative topology of $\partial\Opl$, we infer that $\Sigma$ is locally the graph of a Lipschitz function (see Propositions 2.4.4 and 2.4.7 in \cite{HP}). Moreover, as $\sigma \in H^1_{\rm loc}(\O;\R^2)$, we get that 
$$|\sigma|=1 \quad \HH^1\text{-a.e. on } \Sigma.$$
Indeed, since $\sigma \in H^1_{\rm loc}(\O;\R^2) \cap L^\infty(\O;\R^2)$,  $|\sigma|^2 \in H^1_{\rm loc}(\O)$. Using that $|\sigma|=1$ $\LL^2$-a.e. in $\Opl$ by definition of that set, and that $\Sigma = \O \cap \partial\Opl$, we obtain that the trace of $|\sigma|$ satisfies $|\sigma|=1$ $\HH^1$-a.e. on $\Sigma$. 

In particular, at least when $\Sigma$ is smooth enough, since $\Delta\sigma =0$ in $\Oel$, then $|\sigma| \in \C^0(\Oel\cup\Sigma)$ (see \cite[Theorem A3.3]{brezis.nirenberg}).  We will actually prove  a stronger result, see Theorem \ref{thm:cont_sigma}, namely, that $\sigma$ (and not only its modulus) is continuous in $\Omega$ except possibly at two single points of $\Sigma$, and that $\sigma \in \C^0(\O;\R^2)$, provided $\Sigma$ has a well defined normal at these points. In particular, we will have, in the notation of Theorem \ref{thm:cont_sigma}, that $|\sigma(x)|=1$ for all $x \in \Sigma \setminus \mathcal Z$.
\hfill\P
\end{remark}

We now further investigate the properties satisfied by $u$ in the plastic zone $\Opl$. First,

\begin{lemma}\label{lem:loc-const}
The function $u$ cannot be locally constant in $\Opl$.
\end{lemma}

\begin{proof}
Assume by contradiction that $u$ in constant in an open set $U \subset \Opl$. As a consequence, $Du=0$ and $p=-\sigma\LL^2$ in $U$. The flow rule (stated in Remark \ref{rmk:jump-flow-rule}-(iv)) yields in turn that
$$-|\sigma|^2=\sigma\cdot p=|p|=|\sigma| \quad \text{ in } U$$
which is impossible since $|\sigma|=1$ in $U \subset \Opl$. 
\end{proof}

As part of the  proof of Theorem   \ref{thm.uniq}, we will be examining the superlevel sets $\{u>\lambda\}$ of the function $u$ in $\Opl$. It is  known that the reduced boundary of the superlevel sets of solutions to least gradient problem are minimal surfaces (see {\it e.g.} \cite{JMN,Miranda,SWZ}). In our case, we demonstrate a much stronger structure of the level sets of $u$ in the plastic zone: they are characteristic line segments. The following result states a kind of ``monotonicity'' property of $u$ in $\Opl$ across the characteristic line segments. 
 
\begin{proposition}\label{prop:level-set}
Let $\lambda \in \R$ be such that 
\begin{equation}\label{eq:hyp-ls}
0<\LL^2(\{u>\lambda\} \cap \Opl)<\LL^2(\Opl).
\end{equation} Then there exists {an open half-plane} $H_\lambda$ whose boundary is a characteristic line $L_{x_\lambda}$, for some $x_\lambda \in \Opl$, and 
\begin{equation}\label{eq:1426}
\begin{cases}
u> \lambda \quad \LL^2\text{-a.e. in }\Opl \cap H_\lambda,\\
u< \lambda \quad \LL^2\text{-a.e. in }\Opl \setminus {\overline H_\lambda}.
\end{cases}
\end{equation}
In particular, the essential boundary of the sets {$\{u>\lambda\}$ (resp. $\{u\geq \lambda\}$)} in $\Opl$ is precisely $\partial H_\lambda\cap\Opl$. Moreover, $\sigma(x_\lambda)$ is the inner unit normal to $H_\lambda$. 
\end{proposition}

Note that, in view of Lemma \ref{lem:loc-const},  a $\lambda$ that satisfies \eqref{eq:hyp-ls} always exists.

\begin{proof}
We denote by $u^*$ the precise representative of $u \in BV(\Opl)$  defined outside a set $Z_u \subset \Opl$ with $\HH^1(Z_u)=0$ (see {\it e.g.} \cite[Corollary 3.80]{AFP}). We can assume without loss of generality that $Z_u$ contains the exceptional set $Z$ introduced in the statement of Theorem \ref{thm:struct-sigma-u}. Further, setting $N_u:=\bigcup_{z \in Z_u} (\Opl \cap L_z)$ and $\O_u:=\Opl \setminus N_u$, we can also assume that $\LL^2(N_u)=0$. 

Indeed, this is true of $Z$. As far as $Z_u$ is concerned, by \cite[Proposition 5.7]{BF}, for all $x_0 \in \Opl$, there exists an open neighborhood $U$ of $x_0$ contained in $\Opl$, a square $Q=(-r,r)^2$ and a bi-Lipschitz mapping $\Phi:Q \to U$ with the property that $\Phi^{-1}$ maps the characteristic line segments in $U$ into a family of vertical parallel lines. Specifically, for all $x \in Q$, there exists a unique $t \in (-r,r)$ such that
$$\Phi^{-1}(L_x \cap U)=\{t\} \times (-r,r).$$
Setting $P(t,s):=t$,  the set $\hat Z_u:=P(\Phi^{-1}(Z_u \cap U))$ is $\HH^1$-negligible because $\Phi^{-1}$ and $P$ are Lipschitz. Since $\hat Z_u \subset (-r,r)$, this reads as $\LL^1(\hat Z_u)=0$. Thus,
$$\Phi^{-1}\left( \bigcup_{z \in Z_u} (L_z \cap  U)\right)=\bigcup_{t \in \hat Z_u} \{t\} \times (-r,r)=\hat Z_u \times (-r,r).$$
By Fubini's Theorem, $\LL^2(\hat Z_u \times (-r,r))=0$ and because $\Phi$ itself is Lipschitz,  
$$\LL^2\left(\bigcup_{z \in Z_u} (L_z \cap U)\right)=0.$$
The desired result is obtained by moving the point $x_0$.

\medskip

Let $L$ be a characteristic line, {$x_0 \in L$, and $\xi_0:=\sigma(x_0)$} be the constant value of $\sigma$ on $L \cap \Opl$. {For all $\xi \in \mathbb S^1$ with $\xi_0 \cdot \xi>0$, we} define
$$(\Opl)_y^\xi=\{t \in \R : \; y+t\xi \in \Opl\},$$
which is an open interval by convexity of $\Opl$. For $\h$-a.e. $y \in  L \cap \Opl$, we further set
$$u_y^\xi(t):=u(y+t\xi) \quad \text{ for $\LL^1$-a.e. $t \in (\Opl)_y^\xi$.}$$
By \cite[Theorem 3.107]{AFP}, $u_y^\xi \in BV((\Opl)_y^\xi)$ for $\HH^1$-a.e. $y \in L \cap \Opl$ and the following disintegration result holds:
\begin{equation}\label{eq:slicing}
Du \cdot \xi\res\O^\xi_L=\HH^1\res (L \cap \Opl) \otimes Du_y^\xi
\end{equation}
where $\O^\xi_L:=\{x\in\O: \exists y\in L \cap \Opl \mbox{ s.t. } x\in (\Opl)_y^\xi\}$.
Note that, in that reference, only sections of $u$ in a direction orthogonal to $L$ are considered. However, an inspection of the proof shows that the argument leading to \eqref{eq:slicing} (which in essence relies  on Fubini's Theorem) remains unchanged for non-orthogonal directions. Further, there exist a finite number $p$  of $\xi_i$'s with  $\xi_0 \cdot \xi>0$ such that 
\begin{equation}\label{eq:333}
\O=\cup_{i=1}^p\O^{\xi_i}_L.
\end{equation}

We claim that $\sigma\cdot\xi >0$ in $\Opl$. Let $y \in L \cap \Opl$, since $\sigma$ is constant along $L$ and equal to {$\xi_0$, $\sigma(y)\cdot\xi=\xi_0 \cdot \xi>0$}. If, for some $t \in (\Opl)_y^\xi$, $\sigma(y+t\xi)\cdot \xi=0$, the characteristic line $L_{y+t\xi}$ is   parallel to $\xi$ and passes through $y \in \Opl$ which is impossible since line segments cannot intersect inside $\Opl$ (see \cite[Proposition 5.5]{BF}). The continuity of $\sigma$ in $\Opl$ implies that $\sigma(y+t\xi)\cdot \xi > 0$ for all $t \in (\Opl)_y^\xi$. Any point ${x \in \O}$ can be written as  $x=y+t\xi$ for some {$y \in L \cap \Opl$, $\xi \in \mathbb S^1$ with $\xi \cdot \xi_0>0$ and $t \in (\Opl)_y^\xi$, hence $\sigma(x)\cdot \xi>0$ for all $x \in \O$}. 

Since, by the flow rule (see Remark \ref{rmk:jump-flow-rule}-(iv)), $Du=\sigma \mu$ where $\mu=\LL^2+|p|$, $Du\cdot \xi= \sigma\cdot\xi\ \mu$.  Because $|p|$ is a non negative measure,  $\sigma\cdot\xi\ \LL^2 \leq Du\cdot\xi$ and Fubini's Theorem together with \eqref{eq:slicing} ensures that for $\h$-a.e. $y \in  L \cap \Opl$, 
\begin{equation}\label{eq:1531}
Du_y^\xi\cdot  \xi \ge \sigma_y^\xi \cdot \xi \LL^1>0 \quad \text{ in }(\Opl)_y^\xi,
\end{equation}
where $\sigma_y^\xi(t):=\sigma(y+t\xi)$. 

\medskip

{\sf Case 1.} Assume first that $\lambda \in u^*(\O_u)$ so that there exists $x_\lambda \in \O_u \subset \Opl$ with $u^*(x_\lambda)=\lambda$. By Theorem \ref{thm:struct-sigma-u} and since $Z_u \supset Z$, we know that $u^*$ is constant along $L_{x_\lambda} \cap \Opl$, hence $u^*(x)=\lambda$ for all $x \in L_{x_\lambda} \cap \Opl$. Let $H_\lambda$ be the open half-plane such that $\partial H_\lambda=L_{x_\lambda}$ and containing $\sigma(x_\lambda)$. Set $\xi_0:=\sigma(x_\lambda) \in \mathbb S^1$ and $L:=L_{x_\lambda}$. According to \eqref{eq:1531}, for $\HH^1$-a.e. $y \in L \cap \Opl$ and all $\xi$'s with $\xi_0\cdot\xi>0$, 
$$u_y^\xi < \lambda \quad \LL^1\text{-a.e. on } (\Opl)_y^\xi \cap \R^*_-, \quad u_y^\xi > \lambda \quad \LL^1\text{-a.e. on } (\Opl)_y^\xi \cap \R^*_+$$
so that, by \eqref{eq:333}, \eqref{eq:1426} holds. Note that this first case does not use hypothesis \eqref{eq:hyp-ls}.

\medskip

{\sf Case 2.} Next if $\lambda \not\in u^*(\O_u)$, we define
$$\lambda^+:=\inf\{s>\lambda : \; s \in u^*(\O_u)\} \geq \lambda.$$
Note that if $\lambda^+=\infty$, then $\LL^2(\{u > \lambda\} \cap \Opl)=0$ so that we can assume that $\lambda^+<\infty$. We claim that ${\{u\geq \lambda\}} \cap \Opl=\{u\geq \lambda^+\} \cap \Opl$ up to an $\LL^2$-negligible set. This is obvious if $\lambda^+=\lambda$. If however, $\lambda^+>\lambda$, by definition of $\lambda^+$,  $u^*(\O_u) \cap [\lambda,\lambda^+)=\emptyset$, hence (up to a set of zero $\LL^2$-measure)
$${\{u\geq \lambda\} }\cap \Opl={\{u^*\geq \lambda\} }\cap \O_u =\{u^*\geq \lambda^+\} \cap \O_u=\{u\geq \lambda^+\} \cap \Opl.$$
By definition of the infimum, a decreasing sequence $\{s_n\}_{n \in \N}$ in $u^*(\O_u)$ is  such that $s_n>\lambda$, with $s_n \searrow \lambda^+$. According to Case 1, there exist points $x_n \in \Opl$, characteristic lines $L_n=L_{x_n}$ and open half-spaces $H_n$ satisfying
$$
\begin{cases}
\partial H_n=L_n;\;
H_n\ni\sigma(x_n);\\
u < s_n\ \LL^2\mbox{ -a.e. in }\Opl \setminus {\overline H_n};\;
u > s_n\ \LL^2\mbox{ -a.e. in }\Opl \cap H_n.
\end{cases}
$$

Let  $[a_n,b_n]:=L_n \cap \overline \O_{pl}$. Up to a subsequence, we can suppose that $a_n \to a$, $b_n \to b$ for some $a$, $b \in \partial\Opl$ and the closed line segment $[a_n,b_n]$ converges in the sense of Hausdorff to a closed line segment $[a,b]$. If $[a,b] \subset \partial\Opl$, either $\LL^2(\{u \leq \lambda^+\} \cap \Opl)=0$, or $\LL^2(\{u \geq \lambda^+\} \cap \Opl)=0$. This implies that $\LL^2(\{u \leq \lambda\} \cap \Opl)=0$ or $\LL^2(\{u>\lambda\} \cap \Opl)=0$ which is against \eqref{eq:hyp-ls}. As a consequence $[a,b]$ is not contained in $\partial\Opl$ and by convexity of $\Opl$, $]a,b[ \; \subset \Opl$. 

By continuity of $\sigma$ in $\Opl$,  $]a,b[$ must be a characteristic line segment $L_{x_\lambda} \cap \Opl$ orthogonal to  $\sigma(x_\lambda)$, where $x_\lambda$ is any arbitrary point in $]a,b[$. Let $H_\lambda$ be the {open} half-space containing $\sigma(x_\lambda)$;  $u \leq \lambda^+$ $\LL^2$-a.e. in $\Opl \setminus \ol H_\lambda$ and $u \geq \lambda^+$ $\LL^2$-a.e. in $\Opl \cap H_\lambda$. Recalling that $u^*(\O_u) \cap [\lambda,\lambda^+)=\emptyset$, and that $\LL^2(\{u \neq u^*\})=\LL^2(\Opl \setminus \O_u)=0$, we conclude that  $u \leq \lambda$ $\LL^2$-a.e. in $\Opl \setminus {\overline H_\lambda}$ and $u \geq \lambda$ $\LL^2$-a.e. in $\Opl \cap H_\lambda$.  Using \eqref{eq:1531} with $\xi=\sigma(x_\lambda)$ and $L=L_{x_\lambda}$, we infer that \eqref{eq:1426} holds.
\end{proof}

We  recall one of the main {results} of \cite{BF} (see \cite[Theorem 1.3]{BF}). To this aim, we need to introduce some notation. Given two vectors $v_1$ and $v_2 \in \R^2$, we denote by 
$C(v_1,v_2):=\{\alpha v_1+\beta v_2 :\; \alpha>0, \, \beta>0\}$ the half open cone  generated by $v_1$ and $v_2$. {A boundary fan with apex $\hat z \in \partial\Opl \cap \partial\O$} is an open subset of $\Opl$ of the form
\begin{equation}\label{eq:fan}
\mathbf F_{\hat z}=\Opl \cap (\hat z + C(v_1,v_2)).
\end{equation}

\begin{theorem}\label{thm:structOp}
The set $\Opl$ can be written as the following pairwise disjoint union
\begin{equation}\label{eq:omegapl}
\Opl=\bigcup_{i \in I} \mathbf F_i \cup \bigcup_{\lambda \in \Lambda}(L_{x_\lambda} \cap \Opl) \cup \bigcup_{j \in J} \mathbf C_j,
\end{equation}
for some (possibly) uncountable set $\Lambda$ and at most countable sets $I$, $J$, where 
\begin{itemize}
\item {$\{L_{x_\lambda}\}_{\lambda \in \Lambda}$ is a family of pairwise disjoint characteristic lines} passing though $x_\lambda \in \Opl$;
\item {$\{\mathbf F_{\hat z_i}\}_{i \in I}$ is a family of pairwise disjoint} open boundary fans with apex $\hat z_i \in \partial\Opl \cap \partial \O$;
\item {$\{\mathbf C_j\}_{j \in J}$ is a family of pairwise disjoint convex sets}, closed in the relative topology of $\Opl$ and with non empty interior.
\end{itemize}
Moreover, denoting by 
$$\mathscr F:=\bigcup_{i \in I} \mathbf F_i,$$
then $\{L_{x_\lambda} \cap \Opl\}_{\lambda \in \Lambda}$ (resp. $\{ \mathbf C_j\}_{j \in \N}$) are the connected components of
\begin{equation}\label{def:C}
\CC:=\Opl \setminus \mathscr F
\end{equation}
with empty (resp. nonempty) interior.
\end{theorem}

We also recall the following
\begin{definition}\label{def.dcopl}{\rm
A point $x \in \partial \Opl$ is a {\it characteristic boundary point} if $x \not\in L_z$ for all $z\in \Opl$. 

We denote by $\partial^c \Opl$ the set of all characteristic boundary points.}
\end{definition}

The results of \cite[Theorem 6.11, Remark 6.12]{BF} provides a precise structure of the connected components of $\mathscr C$.

\begin{theorem}\label{prop:structOp2}
Let $\mathbf C$ be a connected component of $\mathscr C$ with nonempty interior. If $\mathbf C \neq \Opl$, then
\begin{itemize}
\item[(i)] Either $\partial \mathbf C=L \cup \Gamma$ with $L$ an open characteristic line segment and $\Gamma$  a connected closed set in $\partial\Opl$. In that case, $\Gamma=\Gamma_1 \cup \Gamma_2 \cup S$ where $\Gamma_1$ and $\Gamma_2$ are connected and $S=\partial \mathbf C \cap \partial^c \Opl=:\partial^c \mathbf C$ is a closed line segment (possibly reduced to a single point) that separates $\Gamma_1$ and $\Gamma_2$. Further, each point of $\Gamma_1$ (resp. $\Gamma_2$) is traversed by a characteristic line segment which will re-intersect $\partial\Opl$ on $\Gamma_2$ (resp. $\Gamma_1$).
\item[(ii)] Or $\partial \mathbf C=L \cup L' \cup \Gamma \cup \Gamma'$ where $L$ and $L'$ are open characteristic line segments, while $\Gamma$ and $\Gamma'$ are two disjoint connected closed sets in $\partial\Opl$. Further each point of $\Gamma$ (resp. $\Gamma'$) is traversed by a characteristic line segment which will re-intersect $\partial\Opl$ on $\Gamma'$ (resp. $\Gamma$). In that case, $\partial \mathbf C \cap \partial^c\Opl=\emptyset$.
\end{itemize}

If however $\mathbf C=\Opl$, then $\partial \Opl= \Gamma_1\cup\Gamma_2\cup S\cup S'$  where $\Gamma_1$ and $\Gamma_2$ are connected and $S$, $S'$ are the (only) connected components of $\partial^c \Opl$. They are disjoint closed line segments (possibly reduced to a single point). Further, each point of $\Gamma_1$ (resp. $\Gamma_2$) is traversed by a characteristic line segment which will re-intersect $\partial\Opl$ on $\Gamma_2$ (resp. $\Gamma_1$).
\end{theorem}

\begin{remark}\label{rem:bdryfancc}
Let $L$ be a characteristic line segment such that $L \cap \overline\O_{pl} \subset \partial\mathbf C$ and $H$ be the closed half-plane with $\partial H\supset L$ that contains $\mathbf C$. For each $\e>0$, we set $H_\e:=\{x \in \R^2 : \; {\rm dist}(x,H) \leq \e\}$. Then 
$$\Opl \cap (H_\e \setminus H) \cap \mathscr F\neq \emptyset,$$
otherwise the connected set {$[\mathbf C \cup (\Opl \cap (H_\e \setminus H)]$} would be disjoint from $\mathscr F$ and strictly contain $\mathbf C$ which would contradict that $\mathbf C$ is a connected component of $\mathscr C$. In other words, the characteristic line segment $L$ is either the boundary of a fan, or the accumulation point for the Hausdorff distance of boundaries of fans. In particular, denoting by $L \cap \overline\O_{pl}=[a,b]$ for some $a$, $b \in \R^2$, then (up to exchanging $a$ and $b$), there exists a sequence $\{z_n\}_{n \in \N}$ of apexes of boundary fans such that $z_n \to a$.

A similar argument would show that, for all $\lambda \in \Lambda$, $L_{x_\lambda} \cap \Opl$ is either the boundary of a fan, or the accumulation point for the Hausdorff distance of boundaries of fans. 
\hfill\P
\end{remark}

\begin{lemma}\label{lem:f_cont}
Let $\mathbf C$ be a connected component of $\mathscr C$ with non-empty interior such that $\partial^c \mathbf C \neq \emptyset$ (case (i) of Theorem \ref{prop:structOp2}). Let $\Gamma_1$ and $\Gamma_2$ be the connected components of $\partial\mathbf C$ defined in that Theorem. Set,  for any $x\in\Gamma_1$, $f(x)$ as the unique intersection point of $L_x$ with $\Gamma_2$. Then $f$ is a homeomorphism from $\Gamma_1$ to $\Gamma_2$.
\end{lemma}

\begin{proof}
According to Theorem \ref{prop:structOp2}-(i), the mapping $f$ is well defined and one to one. It is enough to check that $f$ is continuous on $\Gamma_1$ since a similar argument will lead to the continuity of $f^{-1}$ on $\Gamma_2$.

Let $x \in \Gamma_1$ and  $\{x_n\}_{n \in \N}$ be a sequence in $\Gamma_1$ be such that $x_n \to x$. For each $n \in \N$, there exists $\theta_n \in \R$ such that $f(x_n)=x_n+\theta_n \sigma^\perp(x_n)$. The sequence $\{\theta_n\}_{n \in \N}$ being bounded, it converges, up to a subsequence, to some limit $\theta \in \R$. Moreover, by continuity of $\sigma$ in $\overline{\mathbf C} \setminus \partial^c\mathbf C$ (see \cite[Theorem 6.22]{BF}), 
$f(x_n) \to x +\theta \sigma^\perp(x)=:y\in L_x \cap \overline \Gamma_2$. Since by Theorem  \ref{prop:structOp2}-(i) $L_x$ intersects $\Gamma_1$ and $\Gamma_2$ it follows that, actually, $y \in L_x \cap \Gamma_2$, which proves that $y=f(x)$.
\end{proof}

We next show that $\sigma$ does not change orientation inside all connected components $\mathbf C$ of $\mathscr C$.

\begin{lemma}\label{lem:orientation}
Consider  $\mathbf C$ a connected component of  $\mathscr C$ with nonempty interior and let $L$ be a characteristic line such that $L\cap \overline \O_{pl} \subset\partial{\mathbf C}$. Then, either 
$$\mathring{\mathbf C}=\{x \in \mathring{\mathbf C} : \;  \sigma(x)\cdot(y-x) > 0 \text{ for all }y \in L\cap \overline \O_{pl} \}$$
or
$$\mathring{\mathbf C}= \{x \in \mathring{\mathbf C} : \;  \sigma(x)\cdot(y-x) < 0 \text{ for all }y \in L\cap \overline \O_{pl} \}.$$
\end{lemma}

\begin{proof}
Define
$$\mathbf C^+:=\{x \in \mathring{\mathbf C} : \;  \sigma(x)\cdot(y-x) \geq 0 \text{ for all }y \in L\cap \overline \O_{pl} \}.$$
By continuity of $\sigma$ in $\mathring{\mathbf C}$, we have that $\mathbf C^+$ is closed in $\mathring{\mathbf C}$.

If $\sigma(x)\cdot (y-x)=0$ for some $x \in \mathring{\mathbf C}$ and $y \in L\cap \overline \O_{pl} $, then $L_x$ would be parallel to $y-x$ and the segment $[x,y]$ would be contained in $L_x$. But then both characteristics $L$ and $L_x$ would intersect at $y$ which would lead to a contradiction: indeed this is not possible if $y \in \Opl$ according to \cite[Proposition 5.5]{BF}, while, if $y \in \partial\Opl$, we would have constructed  a boundary fan contained in $\mathbf C$, which is  impossible since $\CC$ contains no boundary fans (recall \eqref{def:C}). This implies that 
$$\mathbf C^+=\{x \in \mathring{\mathbf C} : \; \sigma(x)\cdot(y-x) > 0 \text{ for all }y \in L\cap \overline \O_{pl} \}.$$

Let $x \in \mathbf C^+$ and $\e_0:= \min_{y \in L\cap \overline \O_{pl} }\sigma(x)\cdot(y-x)>0$. By continuity of $\sigma$ in $\mathring{\mathbf C}$, there exists $0<\delta < 3\e_0/4$ such that $B_\delta(x) \subset \mathring{\mathbf C}$ and, if $x' \in B_\delta(x)$, then 
$$|\sigma(x)-\sigma(x')|\leq \frac{\e_0}{4{\rm diam}(\mathbf C)}.$$ Thus, 
$$\sigma(x')\cdot (y-x')\geq \sigma(x)\cdot (y-x')-\frac{\e_0}{4{\rm diam}(\mathbf C)}|y-x'| \geq \sigma(x)\cdot (y-x) - \frac{\e_0}{4} - \delta\geq \frac{3\e_0}{4}-\delta>0$$
which proves that $x' \in \mathbf C^+$. As a consequence $\mathbf C^+$ is open which is possible only if $\mathbf C^+=\emptyset$ or  $\mathbf C^+=\mathring{\mathbf C}$ since $\mathring{\mathbf C}$ is connected. If $\mathbf C^+=\mathring{\mathbf C}$, then we are done. If $\mathbf C^+=\emptyset$, it means that 
$$\mathring{\mathbf C}=\{x \in \mathring{\mathbf C} : \;  \sigma(x)\cdot(y-x) < 0 \text{ for some }y \in L\cap \overline \O_{pl} \}.$$
Assume by contradiction that there exists $y' \in L\cap \overline \O_{pl} $ such that $ \sigma(x)\cdot(y'-x)>0$. For all $t \in [0,1]$, define $y_t:=ty+(1-t)y' \in L\cap \overline \O_{pl} $ because $L\cap \overline \O_{pl} $ is a closed line segment. The mapping $t \in [0,1] \mapsto \sigma(x) \cdot (y_t-x)$ is continuous, while $\sigma(x)\cdot(y_0-x) > 0$ and $ \sigma(x)\cdot(y_1-x) < 0$. The intermediate valued Theorem implies that, for some $t_0 \in \;]0,1[$, $ \sigma(x)\cdot(y_{t_0}-x) = 0$ which is impossible. Consequently, $\sigma(x)\cdot (y'-x)<0$ for all $y' \in L\cap \overline \O_{pl} $, which completes the proof of the lemma in that case as well.
\end{proof}

We can similarly explicit the structure of boundary fans. Recalling the definition \eqref{eq:fan} of a (boundary) fan $\mathbf F_{\hat z}$ with apex $\hat z\in\pa\O\cap\pa\Opl$, we set, for $i=1$, $2$, $L_i=\hat z+\R v_i$,
$$t_i:=\sup\{ t \geq 0 :\; \hat z+tv_i \in \partial\mathbf F_{\hat z}\}$$
and $a_i:=\hat z+t_i v_i$ so that $L_i \cap \overline \O_{pl} =[\hat z,a_i]$. By convexity of $\Opl$,
$$\partial\mathbf F_{\hat z}=[\hat z,a_1]\cup\Gamma\cup [\hat z,a_2],$$
where $\Gamma$ is an open connected set in $\partial\mathbf F_{\hat z}$. 

\begin{proposition}\label{prop:struct-fans}
Let $\mathbf F_{\hat z}$ be a boundary fan with apex $\hat z$ and generatrices $v_1$, $v_2 \in \R$ as in \eqref{eq:fan}, and let $\partial^c \mathbf F_{\hat z}:=\partial \mathbf F_{\hat z} \cap \partial^c\Opl$. Then,
\begin{itemize}
\item Either $\partial^c \mathbf F_{\hat z}=\; ]\hat z,a_1] \, \cup \; ]\hat z,a_2]$;
\item Or $\partial^c \mathbf F_{\hat z}=\; ]\hat z,a_1]$ and $L_2$ is a characteristic line (or the converse);
\item Or $\partial^c \mathbf F_{\hat z}=\emptyset$ and both $L_1$, $L_2$ are characteristic lines.
\end{itemize}
\end{proposition}

\begin{proof}
Since, by \cite[Lemma 6.1]{BF}, $\Gamma$ is traversed by characteristic line in $\mathbf F_{\hat z}$ which also passes through $\hat z$, we deduce that $\partial^c\mathbf F_{\hat z} \subset \; ]\hat z,a_1] \; \cup \; ]\hat z,a_2]$. The conclusion follows observing that, if $]\hat z,a_i[\; \subset \Opl$, then $L_i$ is a characteristic line.
\end{proof}

In view of {Theorem \ref{prop:structOp2} and Proposition \ref{prop:struct-fans}}, if follows that $\partial^c\Opl$, if not empty, is the union of pairwise disjoint line segments possibly reduced to a single point. The following result will imply as a corollary that the characteristic boundary $\partial^c\Opl$ has at most {\it two} connected components. 

\begin{lemma}\label{lem:char-bound}
 There does {\bf not} exist three pairwise disjoint  nonempty characteristic lines $L_1 \cap  \overline\O_{pl}$, $L_2 \cap   \overline\O_{pl}$ and $L_3 \cap   \overline\O_{pl}$ such that
\begin{equation}\label{eq:Si}
L_j \cap \overline \O_{pl}\subset H_i \quad \text{ for all } i \neq j,
\end{equation}
where $H_1$, $H_2$, $H_3$ are half-planes with $\partial H_i=L_i$ for $i=1,2,3$.
\end{lemma}

\begin{proof}
$\partial \Opl$ is a closed, connected set with finite $\HH^1$ measure, therefore, according to \cite[Proposition C-30.1]{david},  it is arcwise connected and  there exists a $1$-periodic Lipschitz continuous mapping 
\begin{equation}\label{eq:star}
\gamma:[0,1] \to \R^2 \text{ such that }\gamma(0)=\gamma(1) \text{ and }\partial \Opl =\gamma([0,1]).
\end{equation}

Let us define the set $$\mathbf C:=\Opl \cap H_1 \cap H_2 \cap H_3.$$ 
It is nonempty, convex, and its boundary contains the three open characteristic line segments $L_i \cap \Opl$. Note that $L_i\cap\partial\Opl=\{x_i,x'_i\}$, where both points $x_i$ and $x'_i$ lie in $\partial\Opl$, so that $L_i\cap\Opl=\; ]x_i,x'_i[$. Since  $L_1 \cap \overline \O_{pl}$, $L_2 \cap \overline \O_{pl}$ and $L_3 \cap \overline \O_{pl}$  are pairwise disjoint, the points $x_1$, $x'_1$, $x_2$, $x'_2$, $x_3$ and $x'_3$ are pairwise distinct. Setting $x_i=\gamma(r_i)$ and $x'_i=\gamma(r'_i)$, we must have
$${0 \leq r'_1 < r_2 <  r'_2 < r_3< r'_3 < r_1 < 1}$$
upon an appropriate choice of $\gamma(0)$. By convexity of $\mathbf C$, the middle points $y_i:={(x_i+x'_i)}/2$ belong to $L_i\cap\Opl$ and consequently, $L_{y_i}=L_i$. Moreover, for $i \neq j$, the closed segments $[y_i,y_j] \subset \Opl$ cannot be contained in a characteristic line $L_x$, for some $x \in \Opl$, otherwise $L_x$ and $L_i$ (resp. $L_j$) would intersect at $y_i$ (resp. $y_j$), which is not possible in view of  \cite[Proposition 5.5]{BF}.

For any $t \in \; ]0,1[$, define $y(t):=ty_3+(1-t)y_1 \in \; ]y_1,y_3[$. The intersection points of $L_{y(t)}$ with $\partial \Opl$, respectively denoted  by $\gamma(s_t)$ and $\gamma(s'_t)$, satisfy
$$s'_t\in [r'_3,r_1], \quad s_t \in [r'_1,r_2] \cup [r'_2,r_3].$$

Define 
$$\underline t:=\sup\{t \in [0,1]: \; L_{y(t)}\cap\gamma([r'_1,r_2])\neq \emptyset\},\quad \overline   t:=\inf \{t \in [0,1]: \;  L_{y(t)}\cap\gamma([r'_2,r_3]) \neq \emptyset\}.$$ 
If, for all $y \in \; ]y_1,y_3[$, the characteristic line $L_y$ intersects $\gamma([r'_2,r_3])$, then, by continuity of $\sigma$ on $[y_1,y_3] \subset \Opl$, we would have that $L_{y_1}\cap\gamma([r'_2,r_3])\neq \emptyset$ which is impossible since $L_{y_1}=L_1$ is disjoint from $\gamma([r'_2,r_3])$. Therefore the set $\{t \in [0,1]: \; L_{y(t)}\cap\gamma([r'_1,r_2])\neq \emptyset\}$ is not empty and $\underline t>0$. A similar argument also shows that $\overline   t<1$. 
Further $\underline t\le\overline   t$ otherwise we could find two points $y$, $y'\in\; ]y_1,y_3[$ such that $L_y$ and $L_{y'}$ would intersect inside $\Opl$, which is impossible by  \cite[Proposition 5.5]{BF}.

 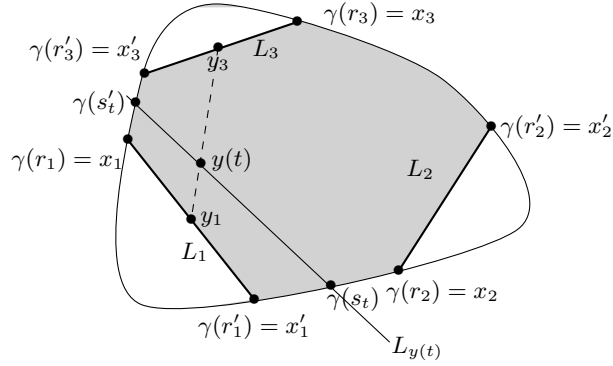
\begin{figure}[hbtp]
\scalebox{1}{\begin{tikzpicture}
\fill[color=gray!35]  plot[smooth cycle] coordinates {(-1,-1) (4,0) (3,2) (0,3) (-1,2)};

\fill[color=white]  plot coordinates {(2.42,-0.55) (4.5,-0.55) (4.5,1.37) (3.65,1.37)  };
\fill[color=white]  plot coordinates {(-3,1.19) (-1.18,1.19)  (.5,-.93) (-2,-1.33)  };
\fill[color=white]  plot coordinates {(-0.97,2.05) (-0.97,2.95)  (1.07,2.96) (1.07,2.76)  };

\draw plot[smooth cycle] coordinates {(-1,-1) (4,0)  (3,2) (0,3) (-1,2)};
\draw[style=thick] (2.42,-0.55) -- (3.65,1.37) ;\draw(2.4,.8) node[right]{\small $L_2$};

\draw[style=thick] (-1.18,1.19) -- (.5,-.93) ; \draw(-.6,-.3) node[right]{\small $L_1$};
\draw  (-1.18,1.19) node{{\small $\bullet$}};
\draw (-1.98,1.2) node[below]{{\small $\gamma(r_1)=x_1$}};

\draw  (.5,-.93) node{{\small $\bullet$}};
\draw (.5,-1) node[below]{{\small $\gamma(r'_1)=x'_1$}};

\draw  (1.07,2.76) node{{\small $\bullet$}};
\draw (1.3,2.86) node[right]{{\small $\gamma(r_3)=x_3$}};

\draw[style=thick] (-0.96,2.07) -- (1.07,2.76) ; \draw (0.35,2.4) node[right]{\small $L_3$};
\draw  (-0.96,2.07) node{{\small $\bullet$}};
\draw  (-0.87,2.4) node[left]{{\small $\gamma(r'_3)=x'_3$}};

\draw  (2.42,-0.55) node{{\small $\bullet$}};
\draw (3,-0.55)  node[below]{{\small $\gamma(r_2)=x_2$}};
\draw  (3.65,1.37) node{{\small $\bullet$}};
\draw  (3.65,1.37)  node[right]{{\small $\gamma(r'_2)=x'_2$}};

\draw (0.02,2.42) node[below]{{\small $y_3$}};
\draw (0.02,2.42) node{{\small $\bullet$}};
\draw (-0.34,0.13) node{{\small $\bullet$}};
\draw (-0.34,0.13) node[right]{{\small $y_1$}};
\draw[dashed] (0.02,2.42)--(-0.34,0.13);
\draw (0.2,1.2) node[below]{{\small $y(t)$}};
\draw (-0.21,.87) node{{\small $\bullet$}};

\draw (-1.2,1.78)--(2.3,-1.5);

\draw(2.2,-1.58) node[right]{{\small $L_{y(t)}$}};
\draw (-1.08,1.69) node[left]{{\small $\gamma(s'_t)$}};
\draw (-1.08,1.69) node{{\small $\bullet$}};
\draw (1.82,-0.65) node[below]{{\small $\gamma(s_t)$}};
\draw (1.52,-0.75) node{{\small $\bullet$}};

\end{tikzpicture}}
\caption{\small  Case of three characteristic characteristic line segments $L_1$, $L_2$, $L_3$.}
\label{fig:isoceles}
\end{figure}

\medskip

Let $\{t_n\}_{n \in \N}$ be a maximizing sequence in $[0,1]$ such that $L_{y(t_n)}\cap\gamma([r'_1,r_2]) \neq \emptyset$ for all $n \in \N$ and $t_n \to \underline t$. Since $\gamma(s'_{t_n}) \in L_{y(t_n)}=y(t_n)+\R \sigma^\perp(y(t_n))$, there exists $\theta_n \in \R$ such that
$$\gamma(s'_{t_n})=y(t_n)+\theta_n \sigma^\perp(y(t_n)),$$
where $\{\gamma(s'_{t_n})\}_{n \in \N}$ is a sequence in $\gamma([r'_3,r_1])$ (hence bounded) and $\{\theta_n\}_{n \in \N}$ is a bounded sequence since $|\sigma^\perp(y(t_n))|=1$ for all $n\in \N$. Therefore, up to a further subsequence $\gamma(s'_{t_n}) \to \underline x \in \gamma([r'_3,r_1])$ and $\theta_n \to \theta$. Thus, using that  $\sigma$ is continuous in $[y_1,y_3] \subset \Opl$, it follows that
$$\underline x=y(\underline t)+\theta \sigma^\perp(y(\underline t)),$$
hence $\underline x \in L_{y(\underline t)} \cap \gamma([r'_3,r_1])$ and $L_{y(\underline t)}$ intersects $\gamma([r'_1,r_2])$. Note that since $\underline t \in \; ]0,1[$, then $\underline y:=y(\underline t) \in\; ]y_1,y_3[$. If $\underline x=\gamma(r'_3)$, then $\underline x$ would be the intersection point of two distinct characteristic lines $L_3$ and $L_{y(\underline t)}$, {hence the apex of a boundary fan $\mathbf F_{\underline x}$ containing $L_2 \cap \Opl$ which is impossible by \cite[Lemma 6.1]{BF}}. On the other hand, if  $\underline x=\gamma(r_1)$, using the continuity of $\sigma$ and the fact that $L_{y(t)}$ intersects $\gamma([r'_2,r_3])$ for all $t>\underline t$, we could construct a sequence $t'_n\searrow \underline t$ such that, for some $\theta'_n \in\R$,
$$\gamma(s'_{t'_n})=y(t'_n)+\theta'_n \sigma^\perp(y(t'_n)).$$
Passing to the limit, we would get that $\gamma(s'_{t'_n})\to \underline x'$ and $\theta'_n \to \theta'$
with $\underline x'\in \gamma([r'_3,r_1])$ and 
$$
\underline x'=y(\underline t)+\theta' \sigma^\perp(y(\underline t)).
$$
Thus $\underline x'\in L_{y(\underline t)}\cap \gamma([r'_3,r_1])$ so $\underline x'=\underline x$ {and we would get that $\underline x$ is the apex of a boundary fan, still denoted by $\mathbf F_{\underline x}$, with, again, the property that $L_2 \cap \Opl \subset \mathbf F_{\underline x}$.} Therefore, $\underline x \in  \gamma(]r'_3,r_1[)$. 

A similar argument shows that $\overline   t \in\; ]0,1[$, $\overline   y:=y(\overline   t) \in\; ]y_1,y_3[$, $L_{\overline   y} \cap \gamma(]r'_2,r_3[)\neq \emptyset$ and $L_{y(\overline   t)}$ intersects $\gamma([r'_2,r_3])$ { as well as $\gamma(]r'_3,r_1[)$}.

Since $r_2<r'_2$, then $\underline t < \overline   t$, otherwise $L_{\underline y}$ and $L_{\overline   y}$ would intersect inside $\Opl$ at a single point $\underline y=\overline   y$ since $L_{\underline y} \cap \gamma(]r'_1,r_2[)\neq \emptyset$ and  $L_{\overline   y} \cap \gamma(]r'_2,r_3[) \neq \emptyset$. But this is impossible by  \cite[Proposition 5.5]{BF}.

Denote by $H_{\underline y}$ and $H_{\overline y}$ the open half-planes with boundary $L_{\underline y}$ and $L_{\overline   y}$ that do not contain the points $y_1$ and $y_3$ respectively. The  region $\mathbf C':=\mathbf C\cap H_{\underline y}\cap H_{\overline y}$ contains the characteristic line segment $L_{y(t)} \cap \Opl$ for all $t\in \; ]\underline t,\overline   t[$. Such a line segment cannot intersect  $L_{\underline y}\cap \Opl$ and $L_{\overline   y}\cap \Opl$ by  \cite[Proposition 5.5]{BF}, it cannot intersect the  connected boundaries $\gamma(]r'_1,r_2[)$ and $\gamma(]r'_2,r_3[)$ by construction, and it cannot intersect the open line segment  $]\gamma(r_2),\gamma(r'_2)[\; = L_2 \cap \Opl$ by \cite[Proposition 5.5]{BF}. The line $L_{y(t)}$ must therefore intersect the point $\gamma(r_2)$ (resp. $\gamma(r'_2)$). {This  is again impossible since there would be a boundary fan containing $L_3 \cap \Opl$ (resp. $L_1 \cap \Opl$), in contradiction with \cite[Lemma 6.1]{BF}.}
\end{proof}

\begin{corollary}\label{cor:char-bound}
The set $\partial^c\Opl$ is the union of at most two pairwise disjoint line segments possibly reduced to a single point.
\end{corollary}

\begin{proof}
Assume, by contradiction, that $S_1$, $S_2$, $S_3$ are three distinct pairwise disjoint {nonempty} line segments with $S_j=(a_j,b_j)$ in $\partial^c \Opl$ (so, in our notation, with possibly $a_j=b_j$). 

Recalling the mapping $\gamma$ of \eqref{eq:star}, we renumber the $S_j$ and exchange $a_j$ with $b_j$ if necessary so that
$$a_i=\gamma(s_i),\; b_j=\gamma(t_j) \text{ for }j\in \{1,2,3\}\mbox{ with } { 0\leq t_1\leq s_2\leq t_2\leq s_3\leq t_3 \leq s_1\leq 1}.$$
We further set
$$\Gamma_1:=\gamma(]t_1,s_2[), \quad \Gamma_2:=\gamma(]t_2,s_3[), \quad \Gamma_3:=\gamma(]t_3,s_1[).$$

If $t_1=s_2$, then $\Gamma_1=\emptyset$ and $b_1=a_2$. As a consequence of {Theorem \ref{prop:structOp2} and Proposition \ref{prop:struct-fans}}, we must have that $S_1=[a_1,b_1[$ and $S_2=\; ]a_2,b_2]$ and the point $b_1=a_2$ is the apex of a boundary fan $\mathbf F$ which, because of the convex character of $\Opl$ must be  $\Opl$  itself. This is however not possible since, by \cite[Lemma 6.1]{BF}, we would get that every point of $S_3$ is traversed by a characteristic line, a contradiction with the fact that $S_3\subset \partial^c\Opl$. This argument shows that $t_1 \neq s_2$, and we prove similarly that $t_2 \neq s_3$, and $t_3 \neq s_1$. 

\medskip

According to Theorem \ref{prop:structOp2} and to Proposition  \ref{prop:struct-fans} and because of the convex character of $\Opl$, there exist three characteristic lines $L_1$, $L_2$, $L_3$ and associated open half planes $H_1$, $H_2$, $H_3$ with $\partial H_i=L_i$ for $i=1,2,3$ such that (see Figure \ref{fig:isoceles2})
$$S_i \subset \R^2 \setminus H_i, \quad S_j \subset H_i \quad \text{ for all } i \neq j.$$
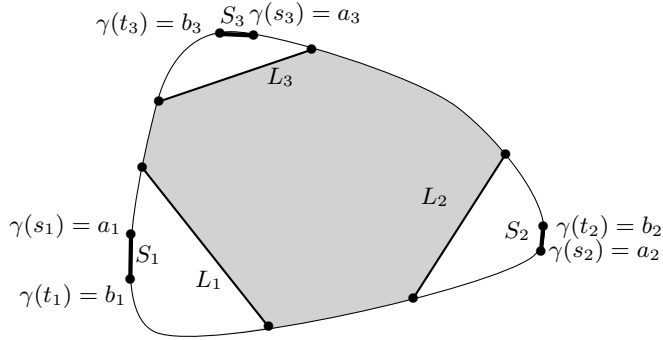
\begin{figure}[hbtp]
\scalebox{1}{\begin{tikzpicture}
\fill[color=gray!35]  plot[smooth cycle] coordinates {(-1,-1) (4,0) (3,2) (0,3) (-1,2)};

\fill[color=white]  plot coordinates {(2.42,-0.55) (4.5,-0.55) (4.5,1.37) (3.65,1.37)  };
\fill[color=white]  plot coordinates {(-3,1.19) (-1.18,1.19)  (.5,-.93) (-2,-1.33)  };
\fill[color=white]  plot coordinates {(-0.97,2.05) (-0.97,2.95)  (1.07,2.96) (1.07,2.76)  };

\draw plot[smooth cycle] coordinates {(-1,-1) (4,0)  (3,2) (0,3) (-1,2)};
\draw[style=thick] (2.42,-0.55) -- (3.65,1.37) ;\draw(2.4,.8) node[right]{\small $L_2$};

\draw[style=thick] (-1.18,1.19) -- (.5,-.93) ; \draw(-.6,-.3) node[right]{\small $L_1$};
\draw  (-1.18,1.19) node{{\small $\bullet$}};

\draw  (.5,-.93) node{{\small $\bullet$}};

\draw  (1.07,2.76) node{{\small $\bullet$}};

\draw[style=thick] (-0.96,2.07) -- (1.07,2.76) ; \draw (0.35,2.4) node[right]{\small $L_3$};
\draw  (-0.96,2.07) node{{\small $\bullet$}};

\draw  (2.42,-0.55) node{{\small $\bullet$}};
\draw  (3.65,1.37) node{{\small $\bullet$}};

\draw (-2.2,.7) node[below]{{\small $\gamma(s_1)=a_1$}};
\draw  (-1.33,.3) node{{\small $\bullet$}};
\draw  (-1.34,-.3) node{{\small $\bullet$}};
\draw[style=  ultra thick] (-1.33,.3)--(-1.34,-.3);
\draw (-2.1,-.24) node[below]{{\small $\gamma(t_1)=b_1$}};
\draw(-1.4,0.03) node[right]{{\small $S_1$}};

\draw (4.1,.08) node[right]{{\small $\gamma(s_2)=a_2$}};
\draw  (4.12,.08) node{{\small $\bullet$}};
\draw  (4.15,.4) node{{\small $\bullet$}};
\draw[style=  ultra thick] (4.12,.08)--(4.15,0.4);
\draw (4.2,.4) node[right]{{\small $\gamma(t_2)=b_2$}};
\draw(4.1,.35) node[left]{{\small $S_2$}};

\draw (0.99,2.95) node[above]{{\small $\gamma(s_3)=a_3$}};
\draw  (0.3,2.95) node{{\small $\bullet$}};
\draw  (-0.15,2.97) node{{\small $\bullet$}};
\draw[style=  ultra thick] (0.3,2.95)--(-0.15,2.97);
\draw (-0.25,3.1) node[left]{{\small $\gamma(t_3)=b_3$}};
\draw(0.32,3.2) node[left]{{\small $S_3$}};

\end{tikzpicture}}
\caption{\small  Case of three characteristic boundaries $S_1$, $S_2$, $S_3$ surrounded by characteristic line segments.}
\label{fig:isoceles2}
\end{figure}
{Note also that $L_1 \cap L_2 \cap \overline \O_{pl}=\emptyset$, otherwise denoting by $x$ the intersection point of $L_1 \cap \overline\O_{pl}$ and $L_2 \cap \overline\O_{pl}$, it would follow that $x$ is the apex of a boundary fan containing $L_3 \cap \Opl$ which is not possible by \cite[Lemma 6.1]{BF}. A similar argument actually shows that 
$L_i \cap L_j \cap \overline \O_{pl}=\emptyset$ for all $i \neq j$. But this geometrical configuration is not allowed by Lemma \ref{lem:char-bound}.}
\end{proof}

\begin{remark}\label{rk.closedness}
In view of {Theorem \ref{prop:structOp2}, Proposition \ref{prop:struct-fans}} and Corollary \ref{cor:char-bound},
the characteristic boundary $\partial^c\Opl$ must be the union of at most two  line segments (maybe reduced to single points) and the end points of those line segments must belong to $\partial^c\Opl$ unless they are the apex of a {boundary} fan.
\hfill\P\end{remark}

If we further  assume that $\O$ is a convex set  the  following additional property holds true.

\begin{proposition}\label{prop:convx}
If $\O$ is convex and $S=(a,b)$ is a connected component of $\partial^c\Opl$, then either $]a,b[\, \cap \Sigma=\emptyset$ or $]a,b[\; \subset \Sigma$.
\end{proposition}

\begin{proof}
Suppose that $]a,b[\, \cap \Sigma\ne \emptyset$ so that, in particular $\Oel $ is not empty. 

Then, either $S=\; ]a,b]$ and $a\in \partial\O$ must be the apex of a boundary fan. Since there is a point in $]a,b] \cap \O$, then,  by convexity of $\O$, $]a,b[\, \subset \O$.  So, $]a,b[\; \subset \O \cap \partial\Opl=\Sigma$.

Or $S=[a,b]$. Assume by contradiction that $]a,b[\, \setminus \Sigma \neq \emptyset$. {By convexity of $\Opl$ and} since $\Sigma \subset \partial\Opl$, there exists $x \in \; ]a,b[$ such that $]a,x] \cap \Sigma=\emptyset$ and $]x,b] \subset \Sigma$. In particular ${]a,x]} \subset \partial\O$. The convexity of $\O$ (hence also of $\overline\O$) shows that $\O$ is contained in one of the open half-spaces $H$ such that $[a,b] \subset \partial H$. But this implies that $]x,b] \subset \partial \O$ which is not possible since $]x,b] \subset \O$.
\end{proof}

\section{Continuity of the solutions}\label{sec.con}

The goal of this section is to establish new continuity results for the solutions $\sigma$ and $u$ of the minimization problem \ref{eq:minim}. Regarding the stress $\sigma$, we show a continuity property in the full domain $\O$ (see Theorem \ref{thm:cont_sigma}) except at two  points at most. Those would be located on the interface  $\Sigma$ between $\Oel$ and $\Opl$ and, at those points,  the normal cone to $\Sigma$ is not degenerate (there are many normals). To this aim, we need to improve the results of \cite{BF} with an accurate account of the behavior of $\sigma$ at characteristic points of $\Sigma$. As for the displacement $u$, we already know that it is smooth (because harmonic) in $\Oel$ and that it cannot jump at all non characteristic points of the interface, as a result of the flow rule and of the convexity of $\Opl$. We improve this property by showing that $u$ cannot jump on the whole of the interface $\Sigma$ and that it is continuous on the portion of $\Opl$ spanned by the characteristic line segments that intersect $\Sigma$ (see Theorem \ref{thm:cont-u}).

\subsection{Continuity of the stress}\label{subsec.con}

In this subsection, we investigate the continuity of $\sigma$ at the interface $\Sigma$ under assumption {\bf (H)}. Recalling Definition \ref{def.dcopl} for  $\partial^c\Opl$ and  \eqref{eq:def-Sigma} for the definition of $\Sigma$, the main result of this section is the following
\begin{theorem}\label{thm:cont_sigma}
Under assumption {\bf (H)}, there exists an exceptional set $\mathcal Z \subset \Sigma \cap \partial^c\Opl$, containing at most two points, such that $\sigma \in \C^0(\O \setminus \mathcal Z;\R^2)$.
\end{theorem}

We first show a partial continuity property of $\sigma$ in $\O \setminus \partial^c\Opl$ which will be  improved at a later stage.

\begin{proposition}\label{prop:cont}
The function $\sigma$ is continuous in $\O \setminus \partial^c\Opl$.
\end{proposition}

\begin{proof}
{\sf Step 1.} We first show that $\sigma \in \mathcal C^0((\Opl\cup \Sigma) \setminus \partial^c\Opl;\R^2)$. It suffices to consider points $x_0 \in \Sigma \setminus \partial^c\Opl=\partial \Opl\cap\O \setminus \partial^c \Opl$ since we already know, by \cite[Theorem 5.1]{BF}, that $\sigma$ is locally Lipschitz continuous inside $\Opl$. First we note that $x_0$ cannot be the apex of a boundary fan because it is inside $\O$. So there is exactly one characteristic line $L_{x_0}$ passing through $x_0$ and, 
as in \cite{BF},  we define $\sigma(x_0)$ as the constant value of $\sigma$ along that characteristic line. We now prove that such an extension is continuous at $x_0$. The proof below is nearly identical to that of \cite[Theorem 6.22]{BF}, but it has to incorporate the case of {boundary} fans which were not in the original proof, which is why we detail it below.

Let $\{x_n\}_{n \in \N}$ be a sequence in $(\Opl\cup\Sigma)\setminus \partial^c \Opl$ such that $x_n \to x_0$. The point $x_n$ is then crossed by a {unique} characteristic line $L_{x_n}$ intersecting $\partial\Opl$ at two points $a_n$ and $b_n$ with, up to a subsequence, $a_n \to a$ and $b_n \to b$. As a consequence, the closed segment $S_n:=L_{x_n}\cap \overline \O_{pl} =[a_n,b_n]$ converges in the sense of Hausdorff to the segment $S=[a,b]$. In particular, since $x_n \in S_n$, then $x_0 \in S$. Moreover, since $a$, $b$, $x_0 \in \partial\Opl$ and $]a,b[$ is an open (possibly empty) line segment inside the convex set $\Opl$, $x_0$ must coincide with either $a$ or $b$. We assume without loss of generality that $x_0=a$. 
{By convexity of $\Opl$, there exists $y_0 \neq x_0$ with $y_0 \in \partial\Opl \cap L_{x_0}$. Let us denote by $H$ one of  half-planes with $\partial H=L_{x_0}$. Up to a further subsequence, we can assume that $S_n \subset H$ for all $n \in \N$. Since $a_n \to a=x_0$, it follows that for yet another subsequence, $a_{n+1}$ belongs to the arc $\mathcal C_n$ in $\partial\Opl$ joining $a_n$ and $x_0$. According to Lemma \ref{lem:char-bound}, for all $n \in \N$, the points $b_{n+1} \not\in \mathcal C_n$ otherwise $\Opl$ would contain the three characteristics line segments $L_{x_0}\cap \Opl$, $S_n=L_{x_n} \cap \Opl$ and $S_{n+1}=L_{x_{n+1}} \cap \Opl$, which is impossible. Moreover, since $S_n$ and $S_{n+1}$ cannot intersect in $\Opl$ by \cite[Proposition 5.5]{BF}, we deduce by convexity of $\Opl$ that either $|a_{n+1}-b_{n+1}| > |a_n-b_n|$ for all $n \in \N$, or $|a_{n+1}-b_{n+1}|\ge |x_0-y_0|$ for all $n \in \N$.  In both cases, we conclude that} $a \neq b$.

\medskip

{\sf Case I:} Assume  that there exists $\delta>0$ such that for all $n \in \N$,
$$\max_{z \in S_n} {\rm dist}(z,\partial\Opl) \geq \delta,$$
then there exists $z_n \in S_n$ such that ${\rm dist}(z_n,\partial\Opl) \geq \delta$ and, up to a subsequence, $z_n \to z$ for some $z \in S$ with ${\rm dist}(z,\partial\Opl) \geq \delta$. Since $x_n \in L_{z_n}$, there exists $\theta_n \in \R$ such that
$$x_n=z_n+\theta_n\sigma^\perp(z_n).$$
Note that, up to a further subsequence, $\theta_n \to \theta \in \R$ and thus, by continuity of $\sigma$ in $\Opl$, we have $x_0=z+\theta\sigma^\perp(z)$ which ensures that $x_0 \in L_z$. Thus, using that $\sigma$ is constant along characteristics and, once again, the continuity of $\sigma$ in $\Opl$, we get that
$$\sigma(x_n)=\sigma(z_n) \to \sigma(z)=\sigma(x_0).$$

\medskip

{\sf Case II:} Assume next that, for some subsequence,
$$\max_{z \in S_n} {\rm dist}(z,\partial\Opl) \to 0.$$
By Hausdorff convergence, for all $z \in S$, there exists a sequence $\{z_n\}_{n \in \N}$ with $z_n \in S_n$ and $z_n \to z$. Thus, ${\rm dist}(z_n,\partial\Opl) \to 0$ which ensures that $S \subset \partial \Opl$. Then $S=[a,b]$ is a closed line segment contained in $\partial\Opl$. If there exists $y \in \; ]a,b[ \; \setminus \partial^c\Opl$, the characteristic line $L_y$ must intersect $S_n$ in $\Opl$ for $n$ large enough because $S_n$ Hausdorff-converges to $[a,b]$. As $S_n=L_{x_n} \cap \overline \O_{pl} $, this is not possible according to \cite[Proposition 5.5]{BF}. {Thus $]a,b[\; \subset\partial^c\Opl$ and $x_0 =a \not\in \partial^c\Opl$. Since $]a,b[\, \neq \emptyset$,} according to Remark \ref{rk.closedness},  $x_0$ must be the apex of a boundary fan, which cannot be so because $x_0\in\O$. This second case never occurs.

\medskip

{\sf Step 2.} We now show that $\sigma \in \mathcal C^0((\Oel\cup \Sigma) \setminus \partial^c\Opl;\R^2)$. 
Note that $\Sigma\setminus\partial^c\Opl$ is relatively open in $\Sigma$  in view of Remark \ref{rk.closedness}. {Also, since $\sigma \in H^1_{\rm loc}(\O;\R^2)$, $\sigma$ has a trace $g$ on $\Sigma$ which belongs to $H^{1/2}_{\rm loc}(\Sigma;\R^2)$.  Using now that $\sigma \in \C^0((\Opl \cup \Sigma) \setminus \partial^c \Opl;\R^2)$, we thus recover that 
\begin{equation}\label{cont.sigma}
g(x)=\sigma(x) \quad \text{ for $\HH^1$-a.e. $x \in \Sigma \setminus \partial^c\Opl$,}
\end{equation}
hence $g$ (has a representative which) is continuous on $\Sigma \setminus \partial^c\Opl$.} Since $u$ is harmonic in the (open) elastic region $\Oel$,  $u \in \C^\infty(\Oel)$. By convexity of $\Opl$, $\Sigma$ is locally the graph of a Lipschitz function. Let $B$ be a ball centered at $x_0 \in \Sigma \setminus \partial^c\Opl$ such that $\overline B \subset \O$ and $B \cap \Sigma=B\cap\partial \Oel=B\cap\partial\Opl \setminus \partial^c \Opl$. Recalling that $\sigma=\nabla u$ in $\Oel$, we get that $\sigma \in H^1(B \cap \Oel;\R^2)$ is a solution of
$$
\begin{cases}
\Delta \sigma=0 & \text{ in } B \cap \Oel,\\
\sigma=g & \text{ on }B \cap \partial\Oel,
\end{cases}
$$
with $g \in \C^0(B \cap \partial \Oel;\R^2)$. As a consequence, if $\p\in \C_c^\infty(\R^2;[0,1])$ is a cut-off function such that $\varphi=1$ in a neighborhood of $x_0$ and ${\rm Supp}(\varphi)\subset B$, then ${\tilde \sigma}:=\p \sigma \in H^1(B \cap \Oel;\R^2)$ satisfies
\begin{equation}\label{eq.truncation}
\begin{cases}
{\Delta \tilde \sigma}=\tilde f & \text{ in } B \cap \Oel,\\
{\tilde \sigma}=\tilde g & \text{ on }\pa (B \cap \Oel),
\end{cases}
\end{equation}
where $\tilde f:=\sigma\Delta \p +2(\nabla\sigma)\nabla\p$ and $\tilde g=\varphi g$ in $\partial\Oel \cap B$ and $\tilde g=0$ in $\partial B \cap \Oel$. Note that $\tilde g$ is continuous on $\pa (B \cap \Oel)$ and $\tilde f \in L^\infty(B \cap\Oel;\R^2)+L^2(B \cap \Oel;\R^2)$. Moreover, since $B \cap \partial\Oel$ is locally the graph of a Lipschitz function, then $B \cap \Oel$ is an open set with Lipschitz boundary which thus satisfies the (exterior) cone condition. Applying {\it e.g.} \cite[Theorem 8.30]{GT}, we conclude that $\tilde \sigma$ is continuous in $\overline{B \cap \O}_{el}$. Since $\p\equiv 1$ in a neighborhood of $x_0$, $\sigma$ must then be continuous at $x_0$ and varying $x_0$ in $\Sigma\setminus\partial^c\Opl$ we conclude that 
\begin{equation}\label{eq.cont.sig--el-reg}
\sigma \in \mathcal C^0((\Oel \cup \Sigma) \setminus \partial^c\Opl;\R^2).
\end{equation}

\medskip

{\sf Step 3.} Since $\sigma \in H^1_{\rm loc}(\O;\R^2)$, $\sigma \in \mathcal C^0((\Oel\cup \Sigma) \setminus \partial^c\Opl;\R^2)\cap \mathcal C^0(( \Opl\cup \Sigma) \setminus \partial^c\Opl;\R^2)$, we deduce that $\sigma \in \mathcal C^0(\O \setminus \partial^c\Opl;\R^2)$.
\end{proof}

\begin{remark}\label{rem.cont-bdary} The result of Proposition \ref{prop:cont} can be extended with the same argument as in the first step to $\overline\O \setminus (\partial^c\Opl\cup(\pa\Opl\cap \mathcal F))$ where $\mathcal F:=\bigcup_{i \in I}\{\hat z_i\}$ and $\hat z_i \in \partial \Opl \cap \partial \O$ is the apex of the boundary fan $\mathbf F_{\hat z_i}$ {in Theorem \ref{thm:structOp}.}\hfill\P
\end{remark}

We next wish to improve the previous result by establishing that $\sigma$ is also continuous across non-degenerate connected components of $\partial^c\Opl$ in $\Sigma$. According to \eqref{eq:omegapl}, if $x_0 \in \Sigma \cap \partial^c\Opl$ then either $x_0 \in \partial^c \Opl \cap \Sigma \cap \partial \mathbf F_{\hat z_i}$ for some {$i \in I$}, or $x_0 \in \partial^c\Opl \cap \Sigma \cap \partial \mathbf C_j$ for some {$j \in J$}. By \cite[Theorem 6.2]{BF}, we already know that $\sigma|_{\Opl} \in \C^\infty(\overline{\mathbf F}_{\hat z_i} \setminus \{\hat z_i\};\R^2)$. In particular, since $\hat z_i$ cannot belong to $\Sigma \subset\O$, then $\sigma|_{\Opl}$ is continuous in $\mathbf F_{\hat z_i} \cup (\Sigma \cap \partial \mathbf F_{\hat z_i})$. 

We now extend this property to connected components $\mathbf C$ of $\mathscr C$ with nonempty interior.
We already know that $\sigma$ is continuous in $\overline{\mathbf C}\setminus \partial^c \mathbf C$ thanks to \cite[Theorem 6.22]{BF}. We improve this result in the case where $\partial^c \mathbf C$ is not reduced to a single point. 

\begin{lemma}\label{lem:segment}
Let $\mathbf C$ be a connected component of $\mathscr C$ with nonempty interior defined in \eqref{def:C} such that
$S=\partial^c\mathbf C=\partial\mathbf C \cap \partial^c\Opl=[a,b]$ with $a\ne b$. 
Then $\sigma$ is continuous in $\overline{\mathbf C}$.
\end{lemma}

\begin{proof}
Recall that, by \cite[Lemma 6.19]{BF} (see also Proposition \ref{prop:structOp2}), $\partial \mathbf C \cap \partial \Opl \setminus \partial^c \mathbf C$ has two connected components $\Gamma_1$ and $\Gamma_2$. We will assume that, {\it e.g.}, $a\in \partial\Gamma_1$,  $b\in \partial\Gamma_2$. Furthermore, all characteristic lines that intersect $\partial \mathbf C\setminus \partial^c \mathbf C$ must intersect both $\Gamma_1$ and $\Gamma_2$. 
Also, {$\partial \mathbf C \cap \overline \O_{pl}  \supset L \cap \overline \O_{pl} $} for some characteristic line $L$. According to Lemma \ref{lem:orientation}, we can assume without loss of generality that
\begin{equation}\label{eq:orient}
\mathring{\mathbf C}=\{x \in \mathring{\mathbf C} : \;  \sigma(x)\cdot(y-x) < 0 \text{ for all }y \in L\cap \overline \O_{pl} \},
\end{equation}
the other case being identical. We claim that $\sigma$ extends by continuity to $S$ by setting $\sigma:=\nu$ on $S$ {where $\nu$ is the (constant) outer unit normal so $\mathbf C$ on the segment $S$.}
\medskip

{\sf Case I:} Take $x \in \; ]a,b[$, and let $\{x_n\}_{n \in \N}$ be a sequence in $\overline{\mathbf C}$ such that $x_n \to x$. 

If  $x_n \in \mathring{\mathbf C}$ for $n$ large enough, then  the characteristic line $L_{x_n}$ intersects $\partial \mathbf C$ at two points $a_n \in \Gamma_1$ and $b_n \in \Gamma_2$. Up to a (not relabeled)  subsequence $a_n \to a'\in\overline\Gamma_1$, $b_n \to b'\in\overline\Gamma_2$ and the segment $[a_n,b_n]$ converges in the sense of Hausdorff to the segment $[a',b']$. {Moreover, exactly as in Step 1, Case II of the proof of Proposition \ref{prop:cont}, $]a',b'[ \; \subset\pa^c\Opl$}. Since $x_n \in [a_n,b_n]$ and $x_n \to x$, we deduce that $x \in [a',b']$. As $x$, $a$, $b$, $a'$ and $b' \in \partial\mathbf C$ and $\mathbf C$ is convex, we get that $[a,b] \subset [a',b'] \subset \partial\mathbf C$. Using that $[a,b]$ is maximal {(see \cite[Proposition 6.8-(ii)]{BF})} we obtain that $a'=a$ and $b'=b$. 

Since $\sigma(x_n)$ is orthogonal to $L_{x_n}\supset [a_n,b_n]$, it follows that any limit $\xi$ of $\sigma(x_n)$ must be orthogonal to $[a,b]$, hence $\xi=\e\nu$ for some $\e=\pm 1$ possibly depending on the {subsequence of $\{x_n\}_{n\in \N}$}. Recalling \eqref{eq:orient}, $\sigma(x_n)\cdot (y-x_n)<0$ for all $y \in L \cap \overline \O_{pl}$ and all $n \in \N$. Passing to the limit yields $\e\nu \cdot (y-x)\leq 0$ for all $y \in L \cap \overline \O_{pl} $. By convexity of $\mathbf C$ we must have that $\e=1$ {so that $\xi=\nu$ is independent of the subsequence}. Thus we can extend $\sigma$ by continuity to $]a,b[$ with value $\nu$ on $S$. 

Otherwise, for a (not relabeled) subsequence $x_n \in \partial \mathbf C$, and for $n$ large enough, $x_n \in \; ]a,b[$ (we use here that $x \in \; ]a,b[$). Since the extension of $\sigma$ to $S$ is precisely equal to $\nu$ on $]a,b[$, we get that $\sigma(x_n)=\sigma(x)=\nu$. Thus $\sigma$ extends by continuity to $]a,b[$ by setting $\sigma:=\nu$.

\medskip

{\sf Case II:} Assume then that $x=a$ (the case $x=b$ can be treated similarly). Let  $\{x_n\}_{n \in \N}$ be a sequence in $\overline{\mathbf C}$ such that $x_n \to a$. If $x_n \in \partial \mathbf C$ for $n$ large enough, it suffices to consider the case where $x_n \in \Gamma_1$ because $\sigma=\nu$ is constant on $]a,b[$. As $x_n \in \Gamma_1$, then $x_n$ is traversed by a characteristic line $L_{x_n}$ which also intersect $\Gamma_2$ at some point $b_n$. Arguing as in the previous case, we infer that $b_n \to b$ and thus that $[a_n,b_n]$ converges in the sense of Hausdorff to the segment $[a,b]$. On the other hand, if $x_n \in \mathring{\mathbf C}$, then the characteristic line $L_{x_n}$ must intersect $\partial \mathbf C$ at two points $a_n \in \Gamma_1$ and $b_n \in \Gamma_2$ which satisfy, up to a subsequence $a_n \to a$, $b_n \to b$ and the segment $[a_n,b_n]$ converges in the sense of Hausdorff to the segment $[a,b]$.

So, in both cases we are back to the setting of Case I and we conclude that $\sigma(x_n) \to \nu$ for $x_n \to x$. Thus $\sigma$ extends by continuity to the full closed segment $S=[a,b]$ by setting $\sigma:=\nu$ on $S$.

\medskip 

Note that, if instead of \eqref{eq:orient}, we have $\mathring{\mathbf C}=\{x \in \mathring{\mathbf C} : \;  \sigma(x)\cdot(y-x) > 0 \text{ for all }y \in L\cap \overline \O_{pl} \}$ (see Lemma \ref{lem:orientation}), then $\sigma$ will extend by continuity to $S$ upon setting $\sigma:=-\nu$ on $S$.
\end{proof}

{\begin{remark}\label{rem:sigma-segment}
According to the proof of Lemma \ref{lem:segment}, {Theorem \ref{prop:structOp2}, Proposition \ref{prop:struct-fans}} and \cite[Theorem 6.2]{BF}, we get that, if $S$ is a non degenerate connected component of $\partial^c\Opl$, then $\sigma=\e\nu$ on $S$ where $\e=\pm 1$ and $\nu$ is a (constant) unit normal to $S$.
\end{remark}}

Finally we partially extend the continuity of $\sigma$ to the degenerate connected components of $\partial^c \Opl$ with a well-defined normal.

\begin{lemma}\label{lem:single_point}
Let $\mathbf C$ be a connected component of $\mathscr C$ with nonempty interior defined in \eqref{def:C} such that
$$S=\partial^c\mathbf C=\partial\mathbf C \cap \partial^c\Opl=\{a\}$$
for some $a$, and such that $\mathbf C$ has a well-defined outer unit normal $\nu(a)$ at $a$. Then $\sigma$ is continuous in $\overline{\mathbf C}$.\end{lemma}

\begin{proof}
Let $L$ be a characteristic line segment such that $\partial \mathbf C \cap \overline \O_{pl} \supset L \cap \overline \O_{pl} $. As before, thanks to Lemma \ref{lem:orientation}, we can assume without loss of generality that \eqref{eq:orient} holds. We claim that $\sigma$ extends by continuity to $\{a\}$ by setting $\sigma(a)=\nu(a)$.

According to \cite[Proposition 6.8-(i)]{BF}, we already know that, for any sequence $\{x_n\}_{n\in \N}$ in $(\partial \mathbf C \cap \partial\Opl) \setminus\partial^c\mathbf C$ with  $x_n\to a$, and say $x_n\in \Gamma_1$, {there exists a subsequence (not relabeld) and $\e=\pm 1$ (possibly depending on the subsequence) such that $\sigma(x_n) \to \e \nu$}.  By continuity of $\sigma$ on $(\partial \mathbf C \cap \partial\Opl) \setminus\partial^c\mathbf C$, we have {\it e.g.} $\sigma(x_n)\cdot (y-x_n) \leq 0$ for all $y \in L \cap \overline \O_{pl} $ and all $n$'s. Passing to the limit in $n$, $\e\nu(a) \cdot (y-a)\leq 0$ for all $y \in L \cap \overline \O_{pl} $, hence $\e=1$ and does not depend on the particular subsequence of $\{x_n\}_{n \in \N}$.

If now $x_n\in \mathring{\mathbf C} \to a$, then an argument identical to that of Case I in Lemma \ref{lem:segment} shows that $L_{x_n}\cap \mathbf C=:[a_n,b_n],$ with $a_n$, $b_n\in \partial \mathbf C \cap \partial\Opl \setminus\partial^c\mathbf C$, Hausdorff-converges to $\{a\}$, so that in particular $a_n\to a$ and $b_n\to a$. But then, by the previous considerations,
$${|\sigma(x_n)-\nu(a)|=|\sigma(a_n)-\nu(a)|} \to 0,$$
since {$a_n \in L_{x_n}$ and} $a_n\in (\partial\mathbf C\cap\partial\Opl) \setminus \partial^c\mathbf C$ and $a_n\ne a$.
\end{proof}

As   Example \ref{ex.tr} below shows, if $\mathbf C$ admits several normals at some isolated characteristic boundary point, then $\sigma$ might not be continuous at that point. 

\begin{example}\label{ex.tr}
{\rm 
Let $T$ be the triangle with vertices $(0,0)$, $a_0:=(0,1)$ and {$b_0:=(1/2,1/2)$}. For all $n \in \N$, we define the points
$$a_n=(0,2^{-n}), \quad b_n=(2^{-n-1},2^{-n-1}).$$
In the triangle $T_n=(a_n,a_{n+1},b_n)$ we consider a fan with apex in $b_n$, while in the triangle $T'_n=(a_{n+1}, b_n,b_{n+1})$ we consider a fan with apex $a_{n+1}$. The fans are oriented in such a way that the resulting function $\sigma$ is continuous across two adjacent triangles. {More precisely, $\sigma$ is defined as
$$
\sigma(x)=
\begin{cases}
\quad\;\frac{(x-b_n)^\perp}{|x-b_n|} & \text{ if } x \in T_n,\\[2mm]
-\frac{(x-a_{n+1})^\perp}{|x-a_{n+1}|} & \text{ if } x \in T'_n.
\end{cases}$$}

We have thus constructed a function $\sigma \in W^{1,\infty}_{\rm loc}(T;\R^2)$ such that
$$|\sigma|=1, \quad {\rm div}\sigma=0 \quad \text{ in } T.$$ 
It is straightforward, with the help of \cite[Theorem 6.2]{BF}, to construct an explicit solution to \eqref{eq:plast} on $T$ with appropriate boundary conditions so that such a $\sigma$ is indeed the associated stress. 

Let $\e>0$ small and $\mathbf C=(0,a_\e,b_\e) \subset T$ be a sub-triangle with $a_\e$, $b_\e \in  \; ]a_0,b_0[$, $|a_\e-a_0| \leq \e$, and $|b_\e-b_0|\leq \e$. Let $x_n$ we the intersection point between the segments $[0,a_\e]$ and $[a_{n},b_n]$, and $y_n$ be the intersection point between the segments $[0,b_\e]$ and $[a_{n+1},b_n]$. Both sequences satisfy $x_n \to (0,0)$ and $y_n \to (0,0)$. Then $\sigma \in \C^0(\overline{\mathbf C}\setminus \{0\};\R^2)$, while 
$${\lim_{n\to \infty}\sigma(x_n) =\frac{(a_0-b_0)^\perp}{|a_0-b_0|}, \quad \lim_{n\to \infty}\sigma(y_n) =-\frac{(b_0-a_1)^\perp}{|b_0-a_1|}.}$$

The point $(0,0) \in \partial^c\mathbf C$ is an isolated characteristic boundary point and all vectors  that belong to the normal cone to $\mathbf C$ at $(0,0)$, are limits of a sequence $\{\sigma(z_n)\}_{n \in \N}$ for some $z_n \in T_\e$ with $z_n \to (0,0)$.
\begin{figure}[htbp]
\scalebox{1}{
\begin{tikzpicture}
\draw [color=blue]  (0,0)-- (1.5,1.5); 
\draw [color=blue] (0,3)-- (1.5,1.5); 
\draw  [color=blue](0,0)-- (0,3);
\draw (0,3) node{{$\cdot$}};
\draw (1.5,1.5) node{{$\cdot$}};
\draw(-.4,3) node{\small{$a_0$}};
\draw(1.8,1.5) node{\small{$b_0$}};
\draw (0,1.5)--(1.5,1.5);

\draw (0, 1.5) node{{$\cdot$}};
\draw (-.4,1.5) node{\small{$a_1$}};
\draw (0.75, 0.75) node{{$\cdot$}};
\draw (1.05,0.75) node{\small{$b_1$}};
\draw (0,1.5)--(0.75,0.75);
\draw (0, .75) node{{$\cdot$}};
\draw (-.4,.75) node{\small{$a_{2}$}};
\draw (0,.75)--(0.75,0.75);
\draw (.375, .375) node{{$\cdot$}};
\draw (.85,.375) node{\small{$b_{2}$}};
\draw (0, .75)--(0.375,0.375);
\draw (0, .375) node{{$\cdot$}};
\draw (0, .375)--(0.375,0.375);
\draw (0.1875,0.1875) node{{$\cdot$}};
\draw (0, .375)--(0.1875,0.1875);
\draw (0, .187) node{{$\cdot$}};
\draw (0,0.1875)--(0.1875, 0.1875);
\draw (0,0.1875)--(.09375, .09375);
\draw (.09375,.09375) node{{$\cdot$}};
\draw (0,.09375)--(.09375,.09375);

\draw  (.3,0.95) node{\tiny{$T_1$}};
\draw  (.38,.6) node{\tiny{$T'_{1}$}};

\draw [color=blue](1.3,2.3)node{$T$};

\end{tikzpicture}}
\caption{An example of non-continuity of $\sigma$.}
\label{fig:trian}
\end{figure}
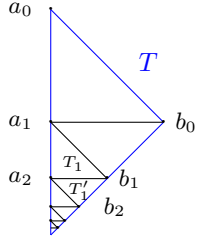

An elementary computation would demonstrate that the $L^2$-norm of $\nabla \sigma$ blows up on $T$
so that the situation described in this example cannot  happen if $T\subsetneq \O$. {This also provides an example distinct from that of boundary fans demonstrating} that one cannot hope to get $H^1$-regularity up to the boundary for $\sigma$. The  $H^1_{\rm loc}$-regularity is in a sense optimal.
\hfill\P}
\end{example}

We are now in position to complete the proof of Theorem \ref{thm:cont_sigma}.

\begin{proof}[Proof of Theorem \ref{thm:cont_sigma}]
From Corollary \ref{cor:char-bound}, we know that $\partial^c\Opl$ contains at most two connected components reduced to a single point that we call $z_1$ and $z_2$. We define $\mathcal Z$ by setting  $z_i \in \mathcal Z$ if $z_i \in \Sigma$ and the normal cone to $\Opl$ at $z_i$ is not reduced to a single direction. The conclusion of Theorem \ref{thm:cont_sigma} is then just a concatenation of the previous results. Indeed Proposition \ref{prop:cont}, Lemmas \ref{lem:segment} and \ref{lem:single_point} together with the fact, originally established in \cite[Theorem 6.2]{BF}, that, if $S=\; ]a,b]\subset\pa^c\Opl\cap\O$, then $\sigma|_{\Opl} \in \C^\infty(\overline{\mathbf F}_{a} \setminus \{a\};\R^2)$ where $\mathbf F_{a}$ is a {boundary fan} with apex $a\in\pa\O$, imply that 
\begin{equation}\label{eq:contOmegapl}
\sigma \in \C^0((\Opl\cup\Sigma)\setminus\mathcal Z;\R^2).
\end{equation}

\medskip

Let $\omega$ be an open subset of $\O$ such that $\mathcal Z \cap \overline \omega=\emptyset$.
Since $\sigma \in H^1(\omega;\R^2)$, then its trace on $\Sigma \cap \omega$, denoted by $g$, is continuous on that set. Using that $\sigma|_{\Oel}=\nabla u|_{\Oel}$ and that $u$ is harmonic in $\Oel$, we infer that
$$\begin{cases}
\Delta \sigma=0 & \text{ in }\Oel \cap \omega,\\
\sigma=g & \text{ on } \Sigma \cap \omega.
\end{cases}$$
The same argument as that in Step 2 of the proof of Proposition \ref{prop:cont} allows us to conclude that $\sigma$ is continuous on $(\Oel \cup \Sigma) \cap \omega$, and thus 
\begin{equation}\label{eq:contOmegael}
\sigma \in \C^0((\Oel\cup\Sigma)\setminus \mathcal Z;\R^2).
\end{equation}
Combining \eqref{eq:contOmegapl}, \eqref{eq:contOmegael} together with $\sigma \in H^1_{\rm loc}(\O;\R^2)$ leads to $\sigma \in \C^0(\O \setminus \mathcal Z;\R^2)$.
\end{proof}

\begin{remark}\label{rem:cont-sig--mulcomp}{
An immediate adaptation of the proof of Theorem \ref{thm:cont_sigma} shows that a similar result holds provided $\Opl$ is the finite union of convex sets with pairwise disjoint closures. More precisely, assume that
$$\Opl=\bigcup_{i=1}^N \O_{pl}^i,$$
where $ \O_{pl}^1,\ldots,\O_{pl}^N$ are open convex sets satisfying $\overline \O_{pl}^i \cap\overline \O_{pl}^j =\emptyset$ for all $1 \leq i \neq j \leq N$. Then, there exists an exceptional set $\mathcal Z \subset \bigcup_{i=1}^N \partial^c \O_{pl}^i$ made of at most $2N$ points such that $\sigma$ is continuous on $\O \setminus \mathcal Z$. 
\hfill\P}
\end{remark}

\begin{remark}\label{rem.sconvex}
If $\Opl$ is strictly convex then its boundary contains no flat parts and thus, $\partial^c\Opl$ can only be at most two isolated points. So, recalling Remark \ref{rem.cont-bdary}, we conclude that $\sigma$ is continuous on $\overline \O_{pl} $, except maybe on the countable set $\mathcal F \cup \mathcal Z$ and that only $\mathcal Z$ can be inside $\O$.\hfill\P
\end{remark}

\subsection{Continuity of the displacement}

In this subsection we investigate the continuity properties of the displacement. Although $u$ is only with bounded variation in $\Opl$, we will show that it is continuous  at all points of $\Opl$  swept by characteristic lines passing through a point of $\Sigma$ (see \eqref{eq:def-Sigma}), that is on the the set
\begin{equation}\label{eq:omega}
\omega:=\{x \in \Opl : \; \exists \, y \in \Sigma \setminus \partial^c\Opl \text{ such that } x \in L_y\}.
\end{equation}

We propose to prove the following partial continuity property of the displacement(s):
\begin{theorem}\label{thm:cont-u}
Under assumption {\bf (H)}, $u$ (has a representative which) is continuous in $\Oel \cup \Sigma \cup \omega$. Moreover, for all $x \in \omega$, $u$ is constant along the characteristic line segment $L_x \cap \Opl$.
\end{theorem}

\begin{remark}
Theorem \ref{thm:cont-u} is the best result that one can hope for when boundary data  are merely in $L^1(\pa\O)$. {Indeed consider $\O$ to be a fan of the form $\{x=(r\cos\theta,r\sin\theta): 0\le r\le 1, 0\le\theta\le \pi/2\}$ and let $h=[0,\pi/2] \to \R$ be a monotonically increasing and non-continuous function with $h'>1$. We set $w(1,\theta)=h(\theta)$ for $\theta \in [0,\pi/2]$ and $w(r,0)=h(0)$ (resp. $w(r,\pi/2)=h(\pi/2)$) for $r \in (0,1]$. Then, according to  \cite[Theorem 6.2]{BF}, the unique solution associated to the Dirichlet boundary data $w \in L^1(\partial\O)$} will be such that $\O=\Opl$ ($\Oel=\Sigma=\omega=\emptyset$) and $u(r,\theta)\equiv h(\theta)$ for all $(r,\theta) \in (0,1) \times (0,\pi/2)$, hence discontinuous along the characteristic lines passing through the discontinuity points of the form $(1,h(\theta))$ with $h$ discontinuous at $\theta$. 
\hfill\P
\end{remark}

We already know that $u \in \C^\infty(\Oel) \cap W^{1,\infty}(\Oel)$. Since  $\Sigma$ is locally the graph of a Lipschitz function, we deduce, as in the proof of Theorem \ref{thm:BF1}, that  
\begin{equation}\label{eq.u-cont-el+sig}
u \in \C^0(\Oel \cup \Sigma). 
\end{equation}
It thus enough to show that $u$ is continuous in $\omega$ and that $u^+=u^-$ on $\Sigma$, where $u^+$ (resp. $u^-$) denotes the trace of $u|_\Oel$ (resp. $u|_\Opl$) on $\Sigma$. Note that previous argument shows that $u^+$ is continuous on $\Sigma$.

Since $\Sigma$ is open in the relative topology of $\partial\Opl$, it has at most countably many connected components. Arguing separately with each connected component, we can assume without loss of generality that $\Sigma$ is connected. Let $g:[0,1] \to \Sigma$ be a one-to-one Lipschitz mapping such that $g(]0,1[) =\Sigma$.

\medskip

\begin{lemma}\label{lem:cont_u}
Let $0 \leq s_0 < t_0 \leq 1$ be such that $g(]s_0,t_0[) \cap \partial^c\Opl=\emptyset$. Define
$$\Sigma_0=g(]s_0,t_0[), \quad \omega_0:=\{x \in \Opl : \; \exists \; y \in \Sigma_0  \text{ such that } x \in L_y\}.$$
Then $u$ has a representative which is continuous in $\Oel \cup \Sigma_0 \cup \omega_0$.
\end{lemma}

\begin{proof}
By definition of $\omega_0$, for all $x \in \omega_0$, there exists a unique characteristic line segment $L_x$ passing through $x$ and intersecting $\Sigma_0$ at a unique point $z_x$. Note that $z_x$ cannot be the apex of a fan since $z_x \in \O$. We define
$$\hat u(x)=u^+(z_x).$$

We claim that the function $\hat u$ is continuous in $\omega_0$. To that effect,  consider a point $x \in \omega_0$. Since $u^+$ is continuous at $z_x \in \Sigma_0$, for all $\e>0$ there exists $\delta>0$ such that $|u^+(y)-u^+(z_x)|\leq \e$ for all $y \in \overline B_\delta(z_x) \cap \Sigma_0$. Let 
$$A_\delta=\bigcup_{y \in \overline B_\delta(z_x) \cap \Sigma} L_y \cap \omega_0.$$
Clearly, $A_\delta$ has non-empty interior and $x \in A_\delta$. Moreover, since characteristic lines do not intersect inside $A_\delta$, for all $y \in A_\delta$ we have $z_y \in \overline B_\delta(z_x)$ and thus
$$|\hat u(y) - \hat u(x)|=|u^+(z_y) - u^+(z_x)| \leq \e,$$
which proves the continuity of $\hat u$ on $\omega_0$ and that, by construction $\hat u|_{\Sigma_0}=u^+$.

Then, $\hat u$ belongs to the  equivalence class of $u$, {\it i.e.}, $\hat u(x)=u(x)$ for $\LL^2$-a.e. $x \in \omega_0$. Here is why. By Theorem \ref{thm:cont_sigma}, $\sigma$ is continuous on $\Sigma_0$ since $\Sigma_0 \cap \partial^c\Opl=\emptyset$. Moreover, by convexity of $\Opl$, for every $x \in \Sigma_0$, the characteristic line $L_x$ is not tangent to $\Sigma_0$. Thus $|\sigma\cdot\nu|<1$ $\HH^1$-a.e. on $\Sigma_0$, and the flow rule in item (i) of Remark \ref{rmk:jump-flow-rule} implies that $u^+=u^-$ $\HH^1$-a.e. on $\Sigma_0$. We can thus find an $\HH^1$-negligible set $Z_1 \subset \Sigma_0$ such that $u^+=u^-$ everywhere on $\Sigma_0 \setminus Z_1$. 

Then $N_1:=\bigcup_{z \in Z_1} (L_z \cap \Opl)$ is $\LL^2$-negligible. It is enough, for that purpose, to show that, for all $s_0<s_1<t_1<t_0$,
\begin{equation}\label{eq:0956}
\LL^2\left(\bigcup_{z \in Z_1 \cap g(]s_1,t_1[)} (L_z \cap \Opl)\right)=0.
\end{equation}

For $\delta < \frac12\min\{s_1-s_0,t_0-t_1\}$ small, consider the convex open set $A_\delta :=\Opl \cap \{x \in \O : \; {\rm dist}(x,\partial\O)>\delta\}$  which has the property $A_\delta \subset\subset \O$, hence that $\sigma \in H^1(A_\delta;\R^2)$.  By the choice of $\delta$, the points $g(s_1)$ and $g(t_1)$ belong to $\O^\delta \cap \partial A_\delta$ and the characteristic lines $L_{g(s_1)}$ and $L_{g(t_1)}$ are not tangential to $\partial A_\delta$. We can then apply Merlet's Lemma stated and proved in the Appendix to, in the notation of that Lemma, $A=A_\delta$ and $\mathcal C=g(]s_1,t_1[)$, the open arc joining $g(s_1)$ and $g(t_1)$ in $\Sigma_0$. We conclude that the set $\bigcup_{z \in Z_1 \cap g(]s_1,t_1[)}(L_z \cap A_\delta)$ is $\LL^2$-negligible. Letting $\delta \to 0$, we obtain \eqref{eq:0956}.

Moreover, according to \cite[Theorem 5.6]{BF}, there exists an $\HH^1$-negligible set $Z_2 \subset \Opl$ such that $u$ is constant on $L_x \cap \Opl$ for all $x \in \O \setminus N_2$ where $N_2:=(\bigcup_{z \in Z_2} L_z)\cap \Opl$ is $\LL^2$-negligible.  The resulting set $Z:=Z_1\cup Z_2$ is $\HH^1$-negligible and  $N:=N_1\cup N_2$ is $\LL^2$-negligible. Moreover, since  $\hat u=u^+$ on $\Sigma_0 \setminus Z$, we deduce that $\hat u=u$ in $\omega_0 \setminus N$, hence $\hat u=u$ $\LL^2$-a.e. in $\omega_0$.
\end{proof}

\begin{proof}[Proof of Theorem \ref{thm:cont-u}]
We distinguish several cases.

\medskip

\noindent {\sf Case 1.} If $\Sigma \subset \partial^c \Opl$ then $\Sigma$ is a closed line segment and $L_y \cap \Sigma=\emptyset$ for all $y \in \Opl$. In that case $\omega=\emptyset$ and the conclusion follows. 

\medskip

\noindent {\sf Case 2.} If $\Sigma \cap \partial^c\Opl=\emptyset$, we apply Lemma \ref{lem:cont_u} with $s_0=0$ and $t_0=1$.

\medskip

\noindent {\sf Case 3.} If $\Sigma \cap \partial^c\Opl\neq \emptyset$, by Corollary \ref{cor:char-bound} and Proposition \ref{prop:convx}, the characteristic boundary $\partial^c \Opl$ has at most two connected components which are closed in the relative topology of $\Sigma$, and whose interior is either disjoint from $\Sigma$ or contained in $\Sigma$. Therefore, by Lemma \ref{lem:cont_u}, it is enough to check that if $S=(a,b)$ is a connected component of $\partial^c \Opl$, then $u^+=u^-$ on $S$ and $u$ is continuous in a neighborhood of $S$ in $\O$.

Let $0 \leq t_a \leq t_b \leq 1$ be such that $a=g(t_a)$ and $b=g(t_b)$. Note that we cannot have $t_a=0$ and $t_b=1$, otherwise $a=g(0)$ and $b=g(1)$ and $\Sigma \subset \partial^c \Opl$, corresponding to Case 1 above. We can thus assume without loss of generality that $t_a>0$. We will distinguish two further subcases.

\medskip

{\sf Case 3a.} If $S=[a,b]=\partial^c\mathbf C$ is the characteristic boundary of a connected component $\mathbf C$ of $\mathscr C$ with nonempty interior, by Theorem \ref{prop:structOp2}, there exists a characteristic line segment $L$ such that $L \cap \Opl \subset \partial\mathbf C$. We denote by $p$ and $q$ the two intersection points of $L$ with $\partial\Opl$. We first note that
\begin{equation}\label{eq:starstar} 
\text{$p$ or $q$ does not belong to $\Sigma$.}
\end{equation}
Indeed, assume by contradiction that both $p$ and $q \in \Sigma$. Since $L\cap \overline\O_{pl}=[p,q]$ is a part of the boundary of $\mathbf C$, Remark \ref{rem:bdryfancc} shows that it is the limit for the Hausdorff convergence of boundaries of (boundary) fans $\{\mathbf F_{z_n}\}_{n \in \N}$ and, say, $p$ is the limit of the apexes $\{z_n\}_{n \in \N}$ which belong to $\partial\Opl \cap \partial\O$. But  $p \in \Sigma \subset \O$, so that $z_n \in \O$ for $n$ large enough, which is impossible.

Using the notation of Theorem \ref{prop:structOp2}, let $\Gamma_1$ and $\Gamma_2$ to be the two connected components of $\partial\mathbf C \cap \partial\Opl \setminus \partial^c\mathbf C$. Up to a change of orientation of the parameterization $g$ of $\Sigma$, we may assume that there exists $t_0 \in \; ]0,t_a[$ such that $g([t_0,t_a[) \subset \Gamma_1$. In particular, $g(t_0) \in \Gamma_1$ so that $L_{g(t_0)}$ will intersect $\Gamma_2$. Let $H$ be an open half-plane such that  $\partial H=L_{g(t_0)}$ and $\overline H$ contains $S$. Let us show that $u$ is continuous in $(\O \cap \overline{\mathbf C} \cap \overline H) \cup \Oel$.

From Theorem \ref{prop:structOp2}-(i), for all $x \in (\O \cap \overline{\mathbf C} \cap \overline H) \setminus [a,b]$, there exists a unique characteristic line $L_x$ passing through $x$ and intersecting $g(]t_0,t_a[)$ at a unique point $z_x$. We define
\begin{equation}\label{eq:hatu}
\hat u(x)=
\begin{cases}
u^+(z_x) & \text{ if }{\overline{\mathbf C} \cap \overline H} \setminus [a,b],\\
u^+(x) & \text{ if }x \in [a,b].
\end{cases}
\end{equation}

We first show that the function $\hat u|_{ \Sigma \cap {\overline H}}$ is continuous on $\Sigma \cap {\overline H}$. First, by construction, $\hat u|_{g([t_0,t_b])}=u^+|_{g([t_0,t_b])}$ is continuous on {$g([t_0,t_b])$}. On the other hand, using the function $f$ introduced in Lemma \ref{lem:f_cont}, $\hat u= u^+ \circ f^{-1}$ on $g(]t_b,1[) \cap {\overline H}$ which shows that $\hat u|_{g(]t_b,1[) \cap {\overline H}}$ is continuous on $g(]t_b,1[) \cap {\overline H}$ as the composition of continuous functions. It remains to show the continuity of $\hat u|_{\Sigma \cap {\overline H}}$ at the junction point $b$. For all $y \in g(]t_b,1[) \cap {\overline H}$, there exists a unique $x(y) \in g(]t_0,t_a[)$ such that $f(x(y))=y$. Since $x(y)\to a$ as $y \to b$, we deduce that $\hat u(y)=u^+(x(y))) \to u^+(a)$. If $a=b$, the continuity of $\hat u$ follows since $u^+(b)=u^+(a)=\hat u(a)$. If $a \neq b$, we recall from Lemma \ref{lem:segment} that $\sigma$ is a constant unit vector orthogonal to $[a,b]$. Now $\sigma|_\Oel=\nabla u|_\Oel$, 
so $u$ satisfies
$$
\begin{cases}
\Delta u=0 & \text{ in }\Oel,\\
\partial_\nu u=\e & \text{ on }{]a,b[}
\end{cases}
$$
with $\e=\pm1$.
By elliptic regularity, we infer that $u \in \C^\infty(\Oel \cup \,]a,b[)$. Since $|\nabla u|<1$ in $\Oel$,  $|\nabla u|\leq 1$ on $]a,b[$. Using that $|\partial_\nu u|=1$ on that set, it follows that $\partial_\tau u=0$ on  $]a,b[$.  The  trace $u^+$ of $u|_{\Oel}$ on $\Sigma$ is therefore constant on $]a,b[$, hence, by continuity, 
\begin{equation}\label{eq:constant-trace} u^+ \mbox{ is constant on }[a,b].
\end{equation}
The continuity of $\hat u$ thus follows in that case as well.

We next prove that the function $\hat u$ is continuous in $\O \cap \overline{\mathbf C} \cap \overline H$. In view of Lemma \ref{lem:cont_u}, we get the continuity of $\hat u$ in $(\O \cap \overline{\mathbf C} \cap \overline H) \setminus [a,b]$. To check the continuity of $\hat u$ on $[a,b]$, let us consider a point $x \in \partial^c\mathbf C=[a,b]$ and a sequence $\{x_n\}_{n \in \N}$ in $\O \cap \overline{\mathbf C} \cap \overline H$ such that $x_n \to x$. 
\begin{itemize}
\item If $x_n \in \mathring{\mathbf C}$ for $n$ large enough, then the closed line segment $[a_n,b_n]:=L_{x_n} \cap \overline{\mathbf C}$ converges in the sense of Hausdorff to $[a,b]$. In particular, up to an interchange of $a_n$ with $b_n$, $a_n \to a$ and $b_n \to b$. Moreover,  $a_n=z_{x_n}$ for all $n \in \N$. Thus, using that $u^+$ is continuous on $\Sigma$, we get
$$\hat u(x_n)=u^+(a_n) \to u^+(a)=\hat u(x),$$
where we used that $u^+$ is constant on the segment $[a,b]$ if $a\neq b$.
\item If, for a (not relabeled) subsequence, $x_n \in \Sigma$, using that $\hat u|_{\Sigma \cap \overline H}$  is continuous on $\Sigma \cap  \overline H$ by Step 1, we immediately get that $\hat u(x_n) \to \hat u(x)$.
\end{itemize}

Finally,  Lemma \ref{lem:cont_u} shows that $\hat u$ and $u$ belong to the same equivalence class since the definition \eqref{eq:hatu} of $\hat u$ is consistent with that given in Lemma \ref{lem:cont_u}.

\medskip

{\sf Case 3b.} If $S = \partial^c\mathbf F$ is the characteristic boundary of a {boundary} fan $\mathbf F$, by Proposition \ref{prop:struct-fans} and since $a \in \O$,  it must be that $S=[a,b[$ where $b=g(1)$ is the apex of $\mathbf F$. Let $t_0 \in \; ]0,t_a[$ be such that the characteristic line $L_{g(t_0)}$ passes through the point $b$. We denote by $H$ the {open} half-plane such that $\partial H=L_{g(t_0)}$ and $H$ contains $[a,b[$.

For all $x \in {\overline{\mathbf F} \cap \overline H \setminus \{b\}}$ we denote by $z_x \in g([t_0,t_a])$ the unique intersection point of $L_x$ with {$g([t_0,t_a])$}, and we set
$$\hat u(x):=u^+(z_x).$$
Using the continuity of $u^+$ and arguing as in Case 3a, we infer that $\hat u$ is continuous in $ {\overline{\mathbf F} \cap \overline H \setminus \{b\}}$. Moreover, by Theorem 6.2 and Proposition 6.3 in \cite{BF}, there exists an $\HH^1$-negligible set $Z \subset \mathbf F \cap  H \setminus \{b\}$ such that $u^+=u^-$ on $\Sigma \cap  H \setminus Z$ and $u$ is constant along $L_x \cap  \Sigma \cap  H$ for all $x \in ( \Sigma \cap  H) \setminus \bigcup_{z \in Z} L_z$. Using the change of variable in polar coordinates (with origin given by $b$) together with Fubini's Theorem, we get that $\LL^2((\bigcup_{z \in Z} L_z) \cap \mathbf F)=0$. Moreover, by construction, $u=\hat u$ in $( \Sigma \cap  H) \setminus \bigcup_{z \in Z} L_z$, hence $u=\hat u$ $\LL^2$-a.e. in {$\mathbf F \cap H$}.
\end{proof}

In the sequel, we will identify $u$ with its continuous representative in the set $\omega$.

\begin{remark}
Under the additional assumption that the Dirichlet boundary data $w$ is continuous on $\partial\O$, a possible generalization of Theorem \ref{thm:cont-u} to a global continuity property of $u$ in the entirety of $\O$ will in particular hinge on the feasibility  of extending Merlet's Lemma (see Lemma in Appendix) to exceptional $\HH^1$-negligible sets $Z$ contained in the exterior boundary of $\Opl$, {\it i.e.}, $\partial\O \cap \partial\Opl$.
\hfill\P
\end{remark}

\begin{remark}{
As in Remark \ref{rem:cont-sig--mulcomp}, a  result similar to that of Theorem \ref{thm:cont-u} holds when
$$\Opl=\bigcup_{i=1}^N \O_{pl}^i,$$
where $ \O_{pl}^1,\ldots,\O_{pl}^N$ are open convex sets satisfying $\overline \O_{pl}^i \cap\overline \O_{pl}^j =\emptyset$ for all $1 \leq i \neq j \leq N$. In that case, defining for all $1 \leq i \leq N$
$$\Sigma_i=\partial \O_{pl}^i \cap \partial \Oel \cap \O, \quad \omega_i= \{x \in \O_{pl}^i : \; \exists \, y \in \Sigma_i \setminus \partial^c\O_{pl}^i \text{ such that } x \in L_y\},$$
 $u$ is continuous in $\Oel \cup \bigcup_{i=1}^N(\Sigma_i \cap \omega_i)$. 
\hfill\P}
\end{remark}

\section{Uniqueness for purely Dirichlet boundary conditions}\label{sec.pdbc}

In this  section, we propose to give conditions under which, with the help of the previously acquired results, uniqueness of the minimizer $u$ in \eqref{eq:minim}, hence also of the minimizer in \eqref{eq.rp}, holds true {under pure Dirichlet boundary conditions.}

\medskip

The following uniqueness theorem holds true.
\begin{theorem}\label{thm.uniq}
Let $\O$ be a simply connected bounded $\C^{1,1}$ domain in $\R^2$ and $w\in L^1(\pa \O)$. Assume that the saturation set $\O_1$ defined in \eqref{eq:O1} satisfies hypothesis {\bf (H)} and has nonempty interior.\footnote{The case where {$\O_1$} has empty interior has already been studied in Theorem \ref{thm:BF1}.} Then the functional $\mathcal I:BV(\O) \to \R$ defined by 
\begin{equation}\label{eq:I}
\mathcal I(u):=\iO W(\nabla u)\, dx+ |D^su|(\O)+\int_{\partial \O}|w-u|\, d\h
\end{equation}
has a unique minimizer in $BV(\O)$.
\end{theorem}

Let us once more emphasize that, in contrast with  classical least gradient type problems in which the trace of the minimizer(s) is prescribed on the boundary, the solutions to our Dirichlet problem do not necessarily match the boundary condition $w$. This potential mismatch complicates uniqueness.

The strategy of proof of Theorem \ref{thm.uniq} is  simple: first establish uniqueness in the elastic domain $\Oel$, then in the plastic domain $\Opl$. Although the problem stated in the elastic domain may seem  straightforward ($u$  solves a Poisson equation), uniqueness is far from  obvious because, as already observed in Example \ref{ex:1}, $u$ may fail to match the boundary value $w$ on the external boundary $\partial\Oel \cap \partial \O$. The main difficulty  consists in proving that the boundary value is actually attained on a part of the external boundary with positive $\HH^1$-measure. Assuming the contrary would imply that $u$ should be constant on the external boundary and this leads to a contradiction. Once uniqueness in $\Oel$ is established,  uniqueness in $\Opl$ is obtained through  a detailed analysis of the level sets of $u$ and a reconstruction of $Du$ thanks to the $BV$-coarea formula.

\medskip

The rest of this section is devoted to the proof of Theorem \ref{thm.uniq}. Let $u_1$ and $u_2 \in BV(\O)$ be two minimizers of $\mathcal I$. We denote by $(u_1,\sigma,p_1)$ and $(u_2,\sigma,p_2)$ the associated solutions to the plasticity problem \eqref{eq:plast}{with $\partial_D\O=\partial\O$} (we recall that the stress $\sigma$ is unique). First, we remark that, because $\sigma$ is uniquely defined, the set $\Oel$ is uniquely defined, independently of the minimizer under consideration. In the sequel we will sometimes denote by $u$ either function $u_1$, or $u_2$.

\subsection{Uniqueness in $\Oel$}\label{subsec:un-el}

If $\Oel=\emptyset$, then one should proceed directly to Subsection \ref{sec.unopl}. So, from now onward in this subsection, we assume that $\Oel\ne\emptyset$.

By arguing in each connected component of $\Oel$, there is no loss of generality in assuming that $\Oel$ itself is connected, and consequently that $\Sigma$ is connected as well. Since $\sigma$ is unique and $\sigma=\nabla u_1=\nabla u_2$ in the connected set $\Oel$, there exists and constant $c \in \R$ such that $u_1-u_2=c$. We are thus tasked with showing that $c=0$. 

\medskip

Assume first that there exists an $\HH^1$-measurable set $A \subset \partial\Oel \cap \pa\O$  with $\HH^1(A)>0$ such that $|\sigma\cdot\nu|<1$ $\HH^1$-a.e. on $A$. Then the flow rule in item (i) of Remark \ref{rmk:jump-flow-rule} implies that $u_1=w$ and $u_2=w$ $\HH^1$-a.e. on $A$. Since $u_1-u_2\equiv c$ in $\Oel$, $c=0$ and 
\begin{equation}\label{eq.un-eldom}
u_1=u_2 \mbox{ in }\Oel.
\end{equation} 
We are thus left with the case where 
$$|\sigma\cdot\nu|=1 \;\; \h\text{-a.e. on }\partial\Oel \cap \partial \O.$$
Our goal in the rest of {this Subsection} is to show that the latter never happens.
\bigskip

Since {$\sigma=\nabla u$ in $\Oel$, then $\sigma\cdot\nu=\partial_\nu u$} on $\pa\Oel \cap \partial\O$ and 
\begin{equation}\label{eq.int-nd0}
|\sigma\cdot\nu|=|\partial_\nu u|=1\quad \h\mbox{-a.e. on }\pa\Oel \cap \partial\O.
\end{equation}

In all that follows, $g:[0,1]\to \R^2$ is a Lipschitz parametrization of $\overline\Sigma$ with the property that, if $\overline   \Sigma \cap \partial\O \neq \emptyset$, $g(0)$ and $g(1)$ are the intersection points of $\overline\Sigma$ with $\pa\O$. 

\medskip

The functions $\sigma$ {and $u$} enjoy good properties on $\partial \Oel \setminus \overline \Sigma$. This is the object of the following lemma which is the only time we use the regularity assumptions on $\O$, namely $\C^{1,1}$ and simple connectedness.

\begin{lemma}\label{lem:1135}
Assume that $\O$ is a simply connected bounded open set with $\C^{1,1}$ boundary and that 
$$|\sigma\cdot\nu|=1 \;\; \h\text{-a.e. on }{\pa\Oel \setminus \overline \Sigma}.$$
Then the function $u$ is constant on $\pa\Oel \setminus \overline \Sigma$. Moreover, there exists a constant $\alpha \in \{-1,1\}$ such that
$$\sigma=\alpha\nu \quad \text{ on }\pa\Oel \setminus \overline \Sigma,$$
where $\nu$ is the outer unit normal to $\pa\Oel \setminus \overline \Sigma \subset \partial\O$.
\end{lemma}

\begin{proof}
First note that, since $\O$ is simply connected, then $\pa\Oel \setminus \overline \Sigma$ is an open connected and Lipschitz subset of $\partial\Oel \cap \partial \O$. Since $u$ is harmonic in $\Oel$ and since, by \eqref{eq.int-nd0}, $|\partial_\nu u|=1$ on $\pa\Oel \setminus \overline \Sigma$, Lemma \ref{lem.u=cst} immediately ensures that $u$ is constant on $\pa\Oel \setminus \overline \Sigma$.

Let $B_R(x_0)$ be a ball centered at $x_0 \in\pa\Oel \setminus \overline \Sigma$ such that $\overline B_R(x_0) \cap \overline \Sigma=\emptyset$. Since $u$ is harmonic in $\O \cap B_R(x_0)$ and constant on $\partial \O \cap B_R(x_0)$ which is of class $\C^{1,1}$, elliptic regularity ensures that $u \in H^2(\O \cap B_r(x_0))$ for all $r<R$. Thus, $\sigma=\nabla u \in H^1(\O \cap B_r(x_0);\R^2)$ and, because the exterior normal $\nu$ is Lipschitz continuous, 
$$
\sigma\cdot\nu \in H^{1/2}(B_r(x_0) \cap\pa\O).
$$
Recalling \eqref{eq.int-nd0}, we thus conclude that $\sigma\cdot\nu$ is in $H^{1/2}_{\rm loc}(\pa\Oel \setminus \overline \Sigma;\{-1,1\})$ and thus must remain constant along $\pa\Oel \setminus \overline \Sigma$ (see {\it e.g.} \cite[Lemma A.3]{BF} in the case of a flat boundary). We thus get the desired expression for $\sigma$ on $\pa\Oel \setminus \overline \Sigma$.
\end{proof}

A first consequence of Lemma \ref{lem:1135} is that $g(0) \neq g(1)$, and thus
\begin{equation}\label{eq:g0g1}
\overline\Sigma \cap \partial\O \neq \emptyset.
\end{equation}
Indeed, if $g(0)=g(1)$ then $\overline\Sigma \cap \partial\O \subset \{g(0)\}$ and $\sigma=\alpha\nu$ on $\partial\O \setminus \{g(0)\}$ for some $\alpha \in \{-1,1\}$. As $\sigma$ is divergence free in $\O$, we get that
$$0=\int_\O {\rm div}\sigma\, dx =\int_{\partial\O}\sigma\cdot \nu\, d\HH^1=\alpha \HH^1(\partial\O),$$
which is impossible. It also implies that $g:[0,1] \to \overline \Sigma$ is one to one.

\medskip

In view of \eqref{eq:g0g1} we can conclude that \eqref{eq.int-nd0} cannot hold if $\partial\O$ is of class $\C^{1,1}$. Indeed, {for $\e$ small}, consider the Lipschitz domain
{$\O_{el}^\e:= \Oel\setminus[B_\e(g(0))\cup B_\e(g(1))]$. Its boundary is $\bigcup_{i=1}^4 \partial\O_i^\e$ with
$$\begin{cases} 
\partial\O_1^\e:=(\partial\Oel\cap\partial\O)\setminus [B_\e(g(0))\cup B(g(1))]\\[1mm]
\partial\O_2^\e:= \Sigma\setminus [B_\e(g(0))\cup B(g(1))]\\[1mm]
\partial\O_3^\e:=\partial B_\e(g(0))\cap \Oel\\[1mm]
\partial\O_4^\e:=\partial B_\e(g(1))\cap \Oel.
\end{cases}
$$
In view of Lemma \ref{lem:1135},
$$
\lim_{\e\to 0} \int_{\partial\O_1^\e}\sigma\cdot\nu\ d\h=\alpha \h(\partial\Oel\cap\partial\O),
$$
while, since $\sigma$ is continuous on $\O\setminus \mathcal Z$ and $|\sigma|\le 1$,
$$
\left|\int_{\partial\O_2^\e}\sigma\cdot\nu\ d\h\right| \le \h(\Sigma), \qquad \left|\int_{\partial\O_{i}^\e}\sigma\cdot\nu\ d\h\right| \leq 2\pi \e \to 0
$$
for $i=3$, $4$. Since
$$
0=\lim_{\e\to 0}\int_{\O_{el}^\e} {\rm div}\sigma \ dx=\lim_{\e\to 0}\sum_{i=1,4}\int_{\partial\O_i^\e}\sigma\cdot\nu\ d\h,
$$
we conclude that, for $\alpha = \pm 1$,
$$
\alpha \h(\partial\Oel\cap\partial\O)-\h(\Sigma)\le 0\le \alpha \h(\partial\Oel\cap\partial\O)+\h(\Sigma).
$$
This is however impossible since $\h(\Sigma)<\h(\partial\Oel\cap\partial\O)$ because of the convexity of $\Opl$ and of the non empty character of $\Oel$.}

\bigskip

In conclusion of  Subsection \ref{subsec:un-el}, we have established that it cannot be so that  $|\sigma\cdot\nu|=1$ $\h$-a.e. on $\partial\Oel \cap \partial \O$, 
which, in view of \eqref{eq.un-eldom}, establishes the uniqueness of $u$ in $\Oel$.

\subsection{Uniqueness in $\Opl$}\label{sec.unopl}

If $u_1$ and $u_2$ are two minimizers of the functional $\mathcal I$ defined in \eqref{eq:I}, we have proved so far that $u_1=u_2$ in $\Oel$, so that the trace of ${u_1}|_{\Oel}$ and ${u_2}|_{\Oel}$ coincide at the interface $\Sigma$ between $\Oel$ and $\Opl$. 
We denote by $\widetilde u$ this common value and define
$$\widetilde w:=w {\bf 1}_{\partial\Opl \cap \partial\O} + \widetilde u {\bf 1}_\Sigma \in L^1(\partial\Opl).$$
The functions $u_1$ and $u_2$ are thus solutions to the plasticity problem \eqref{eq:plast} in $\Opl$ with boundary data given by (the unique function) $\widetilde w$. Equivalently, $u_1|_\Opl$ and $u_2|_\Opl$ are minimizers of 
$$v \in BV(\Opl) \mapsto \int_\Opl W(\nabla v)\, dx + |D^s v|(\Opl) + \int_{\partial\Opl}|v-\widetilde w|\, d\HH^1.$$

Using the flow rule (see Remark \ref{rmk:jump-flow-rule}-(i)), since $|\sigma\cdot \nu|<1$ on $\partial\Opl \setminus \partial^c\Opl$, the boundary value $\widetilde w$ is attained on $\partial\Opl \setminus \partial^c\Opl$. In other words, 
\begin{equation}\label{eq:bdary-value}
u_1=u_2=\widetilde w \quad \HH^1\text{-a.e. on }\partial\Opl \setminus \partial^c\Opl,
\end{equation}
where, in the light of Corollary \ref{cor:char-bound}, $\partial^c\Opl$ is made of at most two connected components which are line segments. In order to show that $u_1=u_2$ $\LL^2$-a.e. in $\Opl$ we follow the proof of \cite[Theorem 1.1]{Gorny} which we adapt to our setting. 

\medskip

We first notice that there are at most countably many $\lambda$'s in $\R$ such that $\HH^1(\{\widetilde w=\lambda\} \cap \partial\Opl)>0$. Consider $\lambda \in \R$ such that $\HH^1(\{\widetilde w=\lambda\} \cap \partial\Opl)=0$. For $i=1$, $2$, we define the superlevel sets of $u_i$ by
$$E^i_\lambda=\{u_i >\lambda\} \cap \Opl.$$

\medskip

 Assume first that $0<\LL^2(E_\lambda^i)<\LL^2(\Opl)$ for all $i=1$, $2$. By Proposition \ref{prop:level-set}, the superlevel sets are of the form $E^i_\lambda=H_\lambda^i \cap \Opl$, where $H_\lambda^i$ are {open} half-planes such that $L^i_\lambda=\partial H^i_\lambda$ is a characteristic line and
$$
\begin{cases}
u_i < \lambda \quad \LL^2\text{-a.e. in }\Opl \setminus {\overline H^i_\lambda},\\
u_i > \lambda \quad \LL^2\text{-a.e. in }\Opl \cap { H^i_\lambda}.
\end{cases}
$$
Let us consider the set {$C_\lambda:=E^1_\lambda \setminus \overline E^2_\lambda=(H^1_\lambda \setminus \overline H^2_\lambda) \cap \Opl$} and assume that $\LL^2(C_\lambda)>0$. {This set is nonempty, open, convex,} and its boundary contains the line segments $L_\lambda^1 \cap  \O_{pl}$ and $L_\lambda^2 \cap  \O_{pl}$. Moreover, $\partial C_\lambda \setminus [(L_\lambda^1 \cap  \O_{pl}) \cup (L_\lambda^2 \cap  \O_{pl})]=\partial C_\lambda \cap \partial\Opl$ has at least one connected component $\Gamma_\lambda$ with $\HH^1(\Gamma_\lambda)>0$. Note also that, according to Lemma \ref{lem:char-bound}, $\Gamma_\lambda \cap \partial^c\Opl=\emptyset$ so that $u_1$ and $u_2$ reach the boundary value $\widetilde w$ on $\Gamma_\lambda$, {\it i.e.}
$$u_1=u_2=\widetilde w \quad \HH^1\text{-a.e. on }\Gamma_\lambda.$$
By definition of the set $C_\lambda$, {$u_1 > \lambda$ and $u_2 < \lambda$ $\LL^2$-a.e. in $C_\lambda$}. Therefore, by positivity of the trace operator in $BV$ (see {\it e.g.} \cite[Lemma 2.2]{Gorny}), we infer that $u_1 \geq \lambda$ and $u_2 \leq \lambda$ $\HH^1$-a.e. on $\Gamma_\lambda$. But since $u_1=u_2=\widetilde w$ $\HH^1$-a.e. on $\Gamma_\lambda$, we deduce that $\widetilde w=\lambda$ $\HH^1$-a.e. on $\Gamma_\lambda$ which is against our choice of $\lambda$. We have thus proved that $\LL^2(C_\lambda)=0$ for all but countably many $\lambda \in \R$, and thus that
$\{u_1>\lambda\} \cap \Opl=H_\lambda^1 \cap \Opl=H_\lambda^2 \cap \Opl=\{u_2>\lambda\}\cap \Opl$ up to an $\LL^2$-negligible set, except possibly for countably many $\lambda$'s in $\R$.

\medskip

If either $\LL^2(E^i_\lambda)=0$ or $\LL^2(E^i_\lambda)=\LL^2(\Opl)$ for $i=1$ or $2$,  one of the sets $E^i_\lambda$ must be $\emptyset$ or $\Opl$ {(up to an $\LL^2$-negligible set)} so that, defining $C_\lambda$ appropriately with $\LL^2(C_\lambda)>0$, we obtain, as before,  that  $u_1 > \lambda$ and $u_2 < \lambda$ $\LL^2$-a.e. on $C_\lambda$, thereby reaching  a contradiction by the same argument as before.

\medskip

In all cases, we conclude that $\{u_1>\lambda\} \cap \Opl=\{u_2>\lambda\} \cap \Opl$ up to an $\LL^2$-negligible set, for all but at most countably many $\lambda$'s in $\R$. The Fleming-Rishel coarea formula in $BV$ (see \cite[Theorem 3.40]{AFP}) then yields
$$Du_1 \res \Opl=\int_\R D\chi_{\{u_1>\lambda\} \cap \Opl}\, d\lambda=\int_\R D\chi_{\{u_2>\lambda\} \cap \Opl}\, d\lambda=Du_2 \res \Opl.$$
The set $\Opl$ being convex, {hence connected},  $u_1-u_2=c$ for some constant $c \in \R$. Using \eqref{eq:bdary-value} and because $\HH^1(\partial \Opl \setminus \partial^c\Opl)>0$, we deduce that $c=0$ and that $u_1=u_2$ $\LL^2$-a.e. in $\Opl$.

\medskip

In conclusion of Subsection  \ref{sec.unopl}, we have established  the uniqueness of $u$ in $\Opl$.

\section{Definition and properties of the characteristic flow}\label{sec.cf}

In this last section the continuity result for $\sigma$ established in Theorem \ref{thm:cont_sigma} guarantees the existence (not the uniqueness) of the characteristic curves \eqref{eq:char-curve}. We analyze their local properties, in particular at the interface $\Sigma$ between $\Oel$ and $\Opl$, as well as their topological properties that excludes geometric situations such as loops.

\medskip

Let $x \in \O \setminus\mathcal Z$, where $\mathcal Z$ is the exceptional discontinuity set of $\sigma$ made of at most two points (see Theorem \ref{thm:cont_sigma}). Since, by Theorem \ref{thm:cont_sigma}, $\sigma \in \C^0(\O \setminus \mathcal Z;\R^2)$, the Cauchy-Peano Theorem yields the existence of a maximal open interval $I=I_x=\; ]\alpha_x,\beta_x[$ containing $0$, and of $\gamma=\gamma_x \in \C^1(I_x;\R^2)$ such that $(\gamma_x,I_x)$ is a maximal solution  of the ODE
\begin{equation}\label{eq:ODE}
\begin{cases}
\gamma_x(t) \in \O\setminus \mathcal Z & \text{ for all }t \in \; I_x,\\
\dot \gamma_x(t)=\sigma^\perp(\gamma_x(t))& \text{ for all }t \in \; I_x,\\
\gamma_x(0)=x.&
\end{cases}
\end{equation}

Either $\beta_x=+\infty$, or $\beta_x<+\infty$. In the former case, the solution is global on the right. In the latter case, since $|\sigma|\leq 1$, the mapping $\gamma_x$ is $1$-Lipschitz and therefore, the limit
\begin{equation}\label{eq:ext-cont}
q_x:=\lim_{t\to \beta_x}\gamma_x(t) 
\end{equation}
exists. In addition, since for every compact set $K \subset \O \setminus \mathcal Z$, there exists $T_K <\beta_x$ such that $\gamma_x(t) \in (\O \setminus \mathcal Z) \setminus K$ for all $T_K<t<\beta_x$, it follows that $q_x \in \partial \O \cup \mathcal Z$. The curve $\gamma_x$ can thus be extended by continuity to $\beta_x$ with $\gamma(\beta_x)=q_x$. A similar statement holds for $\alpha_x$. We thus denote by 
$$\Gamma=\Gamma_x:=
\begin{cases}
\gamma_x([\alpha_x,\beta_x]) & \text{ if }-\infty <\alpha_x < \beta_x<+\infty,\\
\gamma_x(]-\infty,\beta_x]) & \text{ if } -\infty =\alpha_x < \beta_x<+\infty,\\ 
\gamma_x([\alpha_x,+\infty[) & \text{ if } -\infty <\alpha_x < \beta_x=+\infty,\\
\gamma_x(\R) & \text{ if } I_x=\R
\end{cases}$$
the image of the resulting curve which will be called a characteristic (curve).\footnote{We will  on occasion drop the $x$-dependence in $\gamma_x$, $\Gamma_x$, $p_x$, $q_x$ for simplicity, unless confusion may ensue.}

\medskip

We further observe that, if $\sigma(x)=0$, then $I_x=\R$ and $\Gamma_x=\{x\}$. Since $|\sigma|=1$ in $(\Opl \cup \Sigma) \setminus \mathcal Z$, the zeroes of $\sigma$ are those of $\nabla u$ in $\Oel$, i.e., the critical points of the harmonic function $u$ in $\Oel$ which is a locally finite set of points. Therefore,
\begin{equation}\label{eq.fnp}
\text{every compact set $K \subset \O \setminus \mathcal Z$ contains finitely many zeroes of $\sigma$}.
\end{equation}

\begin{remark}\label{rem:u-constant}
Note that, if for some interval open interval $J \subset I_x$ we have $\gamma_x(t) \in \Oel$ for all $t \in J$, using that $u \in \C^\infty(\Oel)$ and $\sigma|_\Oel=\nabla u|_\Oel$, the chain rule yields
$$\frac{d}{dt}u(\gamma_x(t))=\nabla u(\gamma_x(t))\cdot \dot\gamma_x(t)=\sigma(\gamma_x(t))\cdot \dot\gamma_x(t)=0 \quad \text{ for all }t \in J.$$
Thus, $u$ is constant along the portion of characteristic $\gamma_x(J) \subset \Gamma_x$.\hfill\P
 \end{remark}

Since $\sigma \in \C^\infty(\Oel;\R^2)$ and $\sigma \in \C_{\rm loc}^{0,1}(\Opl;\R^2)$, it results from the Cauchy-Lipschitz Theorem that the  solutions are unique inside $\Oel$ and $\Opl$. In other words, if $x \neq y$, then
\begin{equation}\label{eq:caracOmega0}
\Gamma_x \cap \Oel=\Gamma_y \cap \Oel \text{ or } \Gamma_x \cap \Gamma_y\cap \Oel=\emptyset,
\end{equation}
and
\begin{equation*}\label{eq:caracOmega1}
\Gamma_x \cap \Opl=\Gamma_y \cap \Opl \text{ or } \Gamma_x \cap \Gamma_y\cap \Opl=\emptyset.
\end{equation*}
In particular $\Gamma_x$ and $\Gamma_y$ (if distinct) can only intersect  on $\Sigma \cup \partial\O$. 

\medskip

In the rest of this section, we  establish several global properties of the characteristic flow. We first examine how a characteristic locally behaves when intersecting $\Sigma$.

\begin{proposition}\label{prop:trav-loc}
Let $x \in \O \setminus \mathcal Z$ and $(\gamma_x,I_x)$ be a maximal solution of \eqref{eq:ODE} such that $\Gamma_x \cap \Sigma \neq \emptyset$, and let $t_0 \in [\alpha_x,\beta_x]$ be such that $x_0:=\gamma_x(t_0) \in \Gamma_x \cap \Sigma$.
\begin{itemize}
\item[(i)] If $x_0 \not\in \partial^c\Opl$, then there exists $\delta>0$ small such that, up to a change of orientation, $\gamma(t) \in \Opl$ for all $t \in\; [t_0-\delta,t_0[$ and $\gamma(t) \in \Oel$ for all $t \in \; ]t_0,t_0+\delta]$. 
\item[(ii)] If $x_0 \in S$ for some connected component $S=(a,b)$ of $\partial^c \Opl$, then 
either $\Gamma_x \cap S=S$  or $\Gamma_x \cap S=\{a\}$ {(and then $x_0=a$)} or $\Gamma_x \cap S=\{b\}$ {(and then $x_0=b$)} . 
\end{itemize}
\end{proposition}

\begin{proof}
{\sf Case 1}. Assume first that $x_0 \not\in \partial^c\Opl$. Then, by definition of the characteristic boundary, there exists $x \in \Opl$ such that $x_0 \in L_x$. Thus by convexity of $\Opl$, up to a change of orientation, we can assume that $\gamma(t) \in \Opl$ for all $t \in [t_0-\delta,t_0[$, for some $\delta>0$ small. Using again the convexity of $\Opl$ and the continuity of $\dot \gamma$, the vector $\dot \gamma(t_0)=\sigma^\perp(x_0)$ is not in the tangent cone to $\overline\O_{pl}$ since otherwise $L_{x} \cap \Opl=\emptyset$.

By contradiction, assume that there exists a sequence $\{t_j\}_{j\in \N}$ such that $t_j \searrow t_0$ and $\gamma(t_j) \in \O \cap \overline \O_{pl} $ for all $j$. For any $\tau \in \mathcal N_{x_0}(\overline \O_{pl} )$, the normal cone to $\overline \O_{pl} $ at $x_0=\gamma(t_0)$, 
$$\tau \cdot (\gamma(t_j)-\gamma(t_0))\leq 0.$$
Dividing the previous inequality by $t_j- t_0\geq 0$ and letting $t_j \to t_0^+$ leads to $\tau\cdot \dot \gamma(t_0)\leq 0$. Similarly since $\gamma(t) \in \overline \O_{pl} $ for all $t \leq t_0$, then 
$$\tau \cdot (\gamma(t)-\gamma(t_0))\leq 0.$$
Dividing  by $t-t_0\leq 0$ and letting $t \to t_0^-$ yields $\tau\cdot \dot \gamma(t_0)\geq 0$. Thus $\tau\cdot \dot \gamma(t_0)= 0$, which proves that $\dot \gamma(t_0)$ is orthogonal to $\mathcal N_{x_0}(\overline \O_{pl} )$. In other words, $\dot \gamma(t_0)$ is in the tangent cone to  $\overline\O_{pl}$ which is a contradiction. As a consequence, at the possible expense of decreasing $\delta>0$ if necessary, $\gamma(t) \in \Oel$ for all $t \in \; ]t_0,t_0+\delta]$.

\medskip

{\sf Case 2.} Assume next that $x_0 \in S$ for some connected component $S=(a,b)$ of $\partial^c\Opl$. 

Let us show that either $\Gamma \cap S=S$, or $\Gamma \cap S=\{a\}$ or $\Gamma \cap S=\{b\}$. For this, there is no loss of generality in assuming that $a\neq b$. By Lemma \ref{lem:segment}, $\sigma=\e\nu$ for some $\e=\pm 1$, where $\nu$ is the unit normal to $S$ oriented from $\Oel$ to $\Opl$. 

Assume first that $x_0 \in \; ]a,b[$, and let $S':=\; ]a',b'[$ be an open line segment such that $[a',b'] \subset \; ]a,b[$ and $x_0 \in \; ]a',b'[$. Let $U'$ be a smooth, open set such that $\overline{U}' \subset \Omega$, $\overline{U}' \cap \Sigma = \; [a',b']$ and $V'=U' \cap H$, where $H=\{y \in \R^2: \; (y-x_0)\cdot \nu< 0\}$ is the open half plane disjoint from $\Opl$. Since $V' \subset \Oel$, $u$ satisfies 
$$\begin{cases}
\Delta u=0 & \text{ in } V',\\
\partial_\nu u=\e & \text{ on } S',
\end{cases}$$
and since $S' \subset \partial V'$ is flat, elliptic regularity ensures that $u \in \C^\infty(V'\cup S')$. In particular, $\sigma=\nabla u$ is Lipschitz in $V' \cup S'$, and the Cauchy-Lipschitz Theorem shows the uniqueness of the solution of \eqref{eq:ODE} in $V' \cup S'$. Since $x_0 \in S'$ and $\sigma=\e\nu$ is constant in $[a,b]$, there exists a maximal  interval $J \subset I$ containing $t_0$ such that
$$\gamma(t)=x_0+\e (t-t_0) \nu^\perp \quad \text{ for all }t \in J$$
with $\gamma(J) = {(a,b)}$, and thus $\Gamma \cap S=S$.

If $\Gamma \cap S=\{a,b\}$ with $a\ne b$,  we can, up to a change of orientation,  find $\alpha \leq t_a < t_b \leq \beta$  such that $a=\gamma(t_a)$ and $b=\gamma(t_b)$. Note that $\gamma(]t_a,t_b[) \subset \Oel$ otherwise $\Gamma$ would have to cross the segment $]a,b[$ which is not possible because $]a,b[\; \subset \partial^c\Opl$. Since $\gamma([t_a,t_b]) \cup [a,b]$ is a closed Jordan curve, Jordan's Theorem shows that there is a bounded connected component $W$ of $\R^2 \setminus (\gamma([t_a,t_b]) \cup [a,b])$ such that $\partial W=\gamma([t_a,t_b]) \cup [a,b]$. Since $W \subset \Oel$, then $u$ is harmonic in $W$. Moreover, by Remark \ref{rem:u-constant}, there is $c_1 \in \R$ such that $u\equiv c_1$ on $\gamma([t_a,t_b])$. On the other hand, because $\nabla u=\e\nu$ on $[a,b]$,  appealing to  Lemma \ref{lem.u=cst}, $\partial_\tau u=0$ on $[a,b]$ and $u$ is also constant on $[a,b]$. Thus there exists $c_2 \in \R$ such that $u\equiv c_2$ on $[a,b]$. Since by \eqref{eq.u-cont-el+sig} $u$ is continuous on $\Oel \cup \Sigma$,  $c_1=c_2$ because {$x_0 \in \{a,b\}$ and $x_0 \in \Sigma$}. Finally, $u$ being harmonic in $W$ and constant on $\partial W$, the maximum principle shows that $u$ is constant in $W$, hence  $\sigma|_W=\nabla u|_W=0$. But this is against the continuity of $\sigma$ in $]a,b[$ because $|\sigma|=|\e \nu|=1$ on $]a,b[$. 

In conclusion we get that either $\Gamma \cap S=S$, or $\Gamma \cap S=\{a\}$ or $\Gamma \cap S=\{b\}$.
\end{proof}

 Then, we show  that, if a characteristic curve connects in $\Oel$ two points of the interface $\Sigma$, then the portion of $\Sigma$ in between those two points must contain characteristic boundary points of $\Opl$.

\begin{lemma}\label{lem:confinement}
Let $(\gamma_x,I_x)$ be a maximal solution of \eqref{eq:ODE} with $x\in \O \setminus \mathcal Z$. Assume that {$\alpha_x< t_0<t_1 < \beta_x$} are such that
\begin{equation}\label{eq.gam_in_el_dom}
\begin{cases}
\gamma_x(t_0),\, \gamma_x(t_1) \text{ belong to the same connected component of }  { \Sigma},\\
\gamma_x(t_0)\neq \gamma_x(t_1),\\
\gamma_x(]t_0,t_1[) \subset \Oel,
\end{cases}
\end{equation}
and denote by $\mathcal C$ the open arc in $\Sigma$ joining $\gamma_x(t_0)$ and $\gamma_x(t_1)$. Then, at least one connected component $S$ of $\partial^c\Opl$ is such that $S \cap \mathcal C\neq \emptyset$ and $S \subset \overline{\mathcal C}$. Further, if $S$ is reduced to a single point, then $S \subset \mathcal C$.
\end{lemma}

\begin{proof}
Set $x_0=\gamma(t_0)$ and $x_1=\gamma(t_1)$. According to \cite[Proposition C-30.1]{david}, there exists a Lipschitz mapping $g : [0,1] \to \R^2$ such that $g(0)=x_0$, $g(1)=x_1$ and $g([0,1])$ is a curve in $\Sigma$ joining $x_0$ and $x_1$. 

\medskip

Assume that $\partial^c \Opl \cap g(]0,1[) = \emptyset$, so that, for all $s \in \; ]0,1[$, there exists $x_s \in \Opl$ such that $g(s) \in L_{x_s}$. 
Consider the closed Jordan curve made of the union of $g([0,1])$ and $\gamma([t_0,t_1])$. By Jordan's Theorem, $\R^2 \setminus (g([0,1])\cup \gamma([t_0,t_1]))$ has a bounded connected component $U\subset \Oel$ such that $\partial U=g([0,1])\cup \gamma([t_0,t_1])$. {Moreover, $K=\overline U$ is a compact set contained in $\O \setminus \mathcal Z$.} By \eqref{eq:caracOmega0} and since characteristic curves cannot intersect in $\Oel$, for all $s \in [0,1]$,  $\Gamma_{x_s}\cap \Oel \subset U$. Define
$$r^-=\sup\{s \in \;]0,1[\;  : \exists \; t \in \; ]s,1[ \text{ such that } g(t) \in \Gamma_{x_s}\}$$
and
$$r^+=\inf\{s \in \; ]0,1[\;  : \exists \; t \in \; ]0,s[ \text{ such that } {g(t) \in \Gamma_{x_s}}\}.$$
We have $0<r^- \leq r^+ <1$. 
Assume first that $r^-<r^+$ and let $r \in \; ]r^-,r^+[$. Then, since $g(r) \not\in \partial^c\Opl$, there exists $x_r \in \Opl$ such that $g(r) \in L_{x_r}$. By continuity of $\dot \gamma$ and convexity of $\Opl$, the curve $\Gamma_{x_r}$ intersects $\Sigma$ at $g(r)$ and, by definition of $r^-$ and $r^+$, $\Gamma_{x_r}$ cannot intersect $g([0,1])$ elsewhere. 

The first possibility is that $\Gamma_{x_r}$ forms a loop in $\Oel$, that is that there exist $\alpha_{x_r}\leq t'<s'\leq \beta_{x_r}$ such that $\gamma_{x_r}(t')=\gamma_{x_r}(s')=g(r) \in \Sigma \setminus \partial^c\Opl$ and $\gamma_{x_r}(]t',s'[) \subset \Oel$. Defining {$V$} to be the bounded connected component of $\R^2\setminus \gamma_{x_r}([t',s'])$ (because $ \gamma_{x_r}([t',s'])$ is a closed Jordan curve),  
 Remark \ref{rem:u-constant} shows that $u$ remains constant on $\partial U$, thus, by the maximum principle, on all of {$V$}, hence {on the connected component $\omega$ of $\Oel$ containing $V$}. But then, $\sigma\equiv0$ on {$\omega$} which contradicts the fact that  $|\sigma|=1$ on $\Sigma \setminus \mathcal Z$ by Remark \ref{rem:cont-|sigma|-Sigma}.

The second possibility is that $\Gamma_{x_r}$ leaves every compact which would imply that $\Gamma_{x_r}$ intersects $\gamma([t_0,t_1]) \subset \partial U$, again a contradiction with \eqref{eq:caracOmega0}. 
 
 The third possibility is that $\Gamma_{x_r}$ terminates at a point $y\in U$ which must thus be a zero of $\sigma$. Note that, in such a case, $\alpha_x=-\infty$ or $\beta_x=+\infty$. According to \eqref{eq.fnp} there is at most a finite number of such points {in $K$}. Thus, varying $r\in \; ]r^-,r^+[$, we conclude the existence of $\bar r^-, \bar r^+$ with
$r^-<\bar r^-<\bar r^+<r^+$ such that $\Gamma_{x_r}$ terminates at a fixed point $y\in U$ for all $r\in [\bar r^-, \bar r^+]$. But, since by Theorem \ref{thm:cont-u}, $u$ is continuous on {the open set} $\bigcup_{r\in \,]\bar r^-, \bar r^+[}\Gamma_{x_r}\cap\O$ while, in view of Remark \ref{rem:u-constant}, $u$ is constant along $\Gamma_{x_r}\cap\O$, we conclude that $u$ must be constant on the domain $\bigcup_{r\in \, ]\bar r^-, \bar r^+[}\Gamma_{x_r}\cap\Oel$, which is impossible.

Thus we must have that $r^-=r^+=:r$. If $\Gamma_{x_r} \cap g([0,1])=\{g(r)\}$, the same argument as before leads to a contradiction. Therefore, there exists $r' \neq r$ such that $g(r') \in \Gamma_{x_r}$. Without loss of generality, we can assume that $r'>r$. By definition of $r=r^-$, for every $\tau \in \; ]r,r'[$, there exists $\tau' < \tau$ such that $g(\tau') \in \Gamma_{x_\tau}$. But then, $g(\tau) \in \Gamma_{x_{\tau'}}$, in contradiction with the fact that $\tau>r^-=r$.

So  $\partial^c \Opl \cap g(]0,1[)$ cannot be empty and there exists a connected component of $\partial^c\Opl$, namely a line segment $S=(a,b)$, such that $S \cap g(]0,1[)\ne \emptyset$. If $a=b$, $S=\{a\}$ and $a \in  g(]0,1[)$. If $a \neq b$ and {\it e.g.} $x_1 \in \; ]a,b[$, then, by {Proposition \ref{prop:trav-loc}-(ii)}, $]a,b[\; \subset \Gamma_x$. Thus, there exists $\delta>0$ such that $\gamma_x(t) \in S \subset \Sigma$ for all $t \in \; ]t_1-\delta,t_1+\delta[$ which is impossible since $\gamma_x(]t_0,t_1[) \subset \Oel$.  As a consequence $x_1 \not\in \; ]a,b[$ and the same goes for $x_0$. It results that $]a,b[\; \subset g(]0,1[)$, hence $S \subset g([0,1])$.
\end{proof}

\begin{remark}\label{rem:charac_in_Sigma}
 The counterpart of the previous Lemma also holds true, namely that, if $L$ a characteristic line intersecting (the same connected component of) $\Sigma$ at two {distinct} points $g(s)$ and $g(t)$ of $\Sigma$, then there exists a line segment $S=[a,b] \subset \partial^c\Opl$ such that $S \subset g(]s,t[)$.

Indeed, denoting by $H$ the open half-space such that $[g(s),g(t)] \subset \partial H$ and containing $g(]s,t[)$, then the convex set $\Opl \cap H$ is contained in a connected component $\mathbf C$ of $\mathscr C$ with nonempty interior and $g(]s,t[) \subset \partial \Opl \cap \partial\mathbf C$. But then, {Theorem} \ref{prop:structOp2}-(i) implies the result.

This remark, together with  Lemma \ref{lem:confinement}, establishes that if a characteristic curve intersects the same connected component of $\Sigma$ twice at different points, then the closure of that component must contain one of the two connected components  of $\partial^c\Opl$.
\hfill\P
\end{remark}

Finally , we  establish that there can be no closed loops.

\begin{theorem}\label{thm:no-loop}
Let $(\gamma_x,I_x)$ be a maximal solution of \eqref{eq:ODE} with $x\in \O \setminus \mathcal Z$. Then $\Gamma_x$ cannot contain a closed loop.
\end{theorem}

\begin{proof}
Assume by contradiction that there exist $\alpha\le t_0<s_0\le \beta$ such that $x_0:=\gamma(t_0)=\gamma(s_0)$ and define $\Gamma'=\gamma([t_0,s_0])$. Then, $\gamma:[t_0,s_0] \to \R^2$ is a closed Jordan curve, and, by Jordan's Theorem, we can consider the bounded connected component $U$ of $\R^2\setminus \gamma([t_0,s_0])$ such that $\partial U= \gamma([t_0,s_0])$.

{We distinguish two cases which will both lead to a contradiction.}

\medskip

{\sf Case 1:} If $\Gamma'$ remains inside $\Oel \cup \Sigma$, then $U \subset \Oel$ hence $u$ is harmonic in $U$. By Corollary \ref{cor:char-bound} {and Lemma \ref{lem:confinement}}, there are at most three connected components of $\Gamma' \cap \Sigma$. Moreover, by Proposition \ref{prop:trav-loc}, such connected components must be either single points or line segments with non empty interior. On the one hand, by Remark \ref{rem:u-constant}, $u$ is constant on the three connected components of $\Gamma' \setminus \Sigma=\Gamma' \cap \Oel$. On the other hand, if $S$ is a connected component of $\Gamma' \cap \Sigma$ with {relative} non-empty interior, { Proposition \ref{prop:trav-loc} implies that} $S \subset \partial^c\Opl$. According to Lemma \ref{lem:segment}, $\sigma|_{S}=\nabla u|_{S}=\e\nu$ for some $\e=\pm 1$, where $\nu$ is the unit normal to $S$ oriented from $\Oel$ to $\Opl$. As a consequence $\partial_\tau u=0$ on $S$, hence $u$ is constant on $S$. Using that $u$ is continuous in $\Oel \cup \Sigma$, we infer that $u$ is constant on $\partial U$. By the maximum principle on $U$, $u$ is constant on the connected component $\omega$ of $\Oel$ containing $U$. But then, $\sigma=\nabla u \equiv0$ on $\omega$ which contradicts the fact that  $|\sigma|=1$ on $\Sigma \setminus \mathcal Z$ by Remark \ref{rem:cont-|sigma|-Sigma}.

\medskip

{\sf Case 2:} If  $\Gamma'$ intersects $\Opl$, denoting by $\nu$ the outer normal to $U$, we know that $\sigma$ is always parallel to $\nu$ on $\partial U\cap \Opl$ and thus, by continuity of $\sigma$ on $\overline\O_{pl}\setminus\mathcal Z$, $\sigma\cdot\nu$ can only change sign in $\Oel$.  Take $y_0 \in \partial U \cap \Oel$ such that $\sigma(y_0)\cdot \nu(y_0)=0$. By the definition \eqref{eq:ODE} of the characteristics, $\sigma$ is orthogonal to $\Gamma'=\partial U$ and, consequently, $\sigma(y_0)=0$. Let $r_0 \in [\alpha,\beta]$ be such that $y_0=\gamma(r_0)$. Using  \eqref{eq:ODE} again yields $\dot \gamma(r_0)=0$. Since $y_0 \in \Oel$ and $\sigma \in \C^\infty(\Oel;\R^2)$, the Cauchy-Lipschitz Theorem ensures the existence of $\delta>0$ such that the ODE
$$
\begin{cases}
X(r) \in \Oel & \text{ for all } r \in [-\delta,\delta],\\
\dot X(r)=\sigma^\perp(X(r)) & \text{ for all } r \in [-\delta,\delta],\\
X(0)=y_0&
\end{cases}$$
has a unique local solution which must satisfy $X(r)=\gamma(r_0+r)=y_0$ for all $ r \in [-\delta,\delta]$. Thus the characteristic curve $\Gamma'$ (hence $\Gamma$) stops at $y_0 \in \Oel$, which contradicts the fact that $\Gamma'$ is a closed loop.   We can thus assume without loss of generality that   $\sigma\cdot\nu>0$ on $\partial U$. Finally, according to the divergence theorem, 
$$0=\int_U \Div \sigma\, dx = \int_{\partial U} \sigma\cdot\nu\, d\HH^1,$$
which is  impossible. 
\end{proof}

\medskip

\section*{Acknowledgements}

\noindent JFB's work was supported by a public grant from the Fondation Math\'ematique Jacques Hadamard. The authors thank Jean-Jacques Marigo for pointing out  the MacClintock construction of Example \ref{ex.Mc} as a possibly tractable computation, and Vincent Millot for directing them to reference \cite{brezis.nirenberg}.
They also greatfully acknowledge Beno\^it Merlet for providing a proof of the Lemma  in the Appendix which is thus baptized ``Merlet's Lemma" in the text.
 
 \section*{Appendix}

\setcounter{theorem}{0}
\setcounter{equation}{0}
\setcounter{section}{1}

\renewcommand{\thesection}{\Alph{section}}
\newcommand\ee{{\vec e}}

In our context, revisiting the proof of Theorem 3.10 (see \cite[Theorem 5.6]{BF}), it can be seen that characteristic lines in $\Opl$ originating from an $\HH^1$-negligible set of points in (the interior of) $\Opl$ form an $\LL^2$-negligible set (see also the proof of Proposition \ref{prop:level-set}).  Specifically, if $Z \subset  \Opl$ is an $\HH^1$-negligible set, then $\bigcup_{z \in Z} (L_z \cap \Opl)$ is $\LL^2$-negligible. Unfortunately, the proof of this property does not easily extend to the case where the exceptional set $Z$ leaves on the boundary of $\Opl$ because $\sigma$ is only locally Lipschitz in $\Opl$. {For example, the apex $\hat z$ of a boundary fan is a set of zero $\HH^1$ measure (since it is a singleton), but the union of the characteristics originated from $\hat z$ is precisely the boundary fan $\mathbf F_{\hat z}$, hence a set of positive $\LL^2$ measure. However, the $H^1_{\rm loc}(\O;\R^2)$-regularity of $\sigma$ will permit us to  consider exceptional sets $Z \subset \O \cap \partial \Opl$, thanks to the following lemma (see the proof of  Lemma \ref{lem:cont_u} for the details).

\begin{theorem*}[Merlet's Lemma] 
 Consider a convex open set $A\subset\R^2$ and $m \in H^1(A;\R^2)\cap \C^0(\overline A;\R^2)$ with
 $$
 \begin{cases}
 {\rm div\, } m=0\\
 |m|=1
 \end{cases}
 \quad\text{ in }A.$$ 
Define $L_x=x+\R m^\perp(x)$ to be the characteristic line with director $m^\perp(x)$ going through $x \in A$. Consider two points $x_0$ and $x_1$ in $\partial A$ 
 such that $L_{x_0}$ (resp. $L_{x_1}$) is not in the tangent cone to $\partial A$ at $x_0$ (resp. $x_1$), and let $\C$ the open arc joining $x_0$ and $x_1$ in $\partial A$.   
If $Z \subset \C$ is such that $\HH^1(Z)=0$, then $\LL^2\left(\bigcup_{z\in Z}(L_z\cap A)\right)=0$.
 \end{theorem*}
 
\begin{proof}
In the following proof,  angles  are non-oriented.

By a result first derived in \cite[Proposition 3.2]{JOP} (see also \cite[Proposition 5.4]{BF} in our specific context) together with the continuity of $m$ on $\overline A$, $m$ remains constant along $L_x\cap\overline A$ for all $x \in \overline A$. In particular,  $L_x\cap L_y \cap \overline A=\emptyset$ for $x$, $y\in \overline A$ with $x\ne y$, otherwise $m$ would not belong to $H^1(A;\R^2)$ (see {\it e.g.} \cite[Theorem 6.2]{BF} for the case $x=y \in\pa A$).

Denote by $g:[0,1]\to {\overline \C}$ a one-to-one Lipschitz parametrization of $\overline \C$ with $g(0)=x_0$ and $g(1)=x_1$. {Let $H_0$ (resp. $H_1$) be} the open half-plane such that $\partial H_0=L_{x_0}$ (resp. $\partial H_1=L_{x_1}$) {and} containing $L_{x_1} \cap A$ (resp. $L_{x_0} \cap A$). Then 
$$U :=  H_0 \cap H_1 \cap A$$ 
is a (nonempty) convex open subset of $A$ which has the property 
$$(L_{x_0} \cap \ol A) \cup (L_{x_1} \cap\ol A ) \cup \C \subset \partial U.$$ 

Let $\C'=\partial U \setminus [(L_{x_0} \cap \ol A) \cup (L_{x_1} \cap\ol A ) \cup \C]$.  
We first notice that the length of any line segment joining a point $x \in \C$ to a point $y \in \C'$ must stay between two positive constants, {\it i.e.,} 
\begin{equation}\label{a-1}
0<\alpha \le \h([x,y])=|x-y|\le {\rm diam}(A)\quad \text{ for all }(x,y)\in \C \times \C',
\end{equation}
for some $\alpha>0$ only depending on $A$.

Let us denote by $y_0$ (resp. $y_1$) the other intersection point of $L_{x_0}$ (resp. $L_{x_1}$) with $\partial A$. 
 Because of the convex character of $U$ and since $L_{x_0}$ (resp. $L_{x_1}$) is not in the tangent cone to $\pa A$ at $x_0$ (resp. $x_1$), the angle
 between any chord $[x,x']$ joining two points $x$, $x' \in \C$ with a segment $[x,y]$ joining $x \in \C$ to $y \in \C'$ is such that 
{ \begin{equation}\label{a12}
\eta_0 \leq \angle([x,x'],[x,y]) \le \pi-\eta_0,
 \end{equation}
 for some angle $0<\eta_0 \leq \min\{\angle{([x_0,x_1],[x_0,y_1])},\angle{([x_0,x_1],[y_0,x_1])}\}<\pi$ only depending on $A$.}
 
 \medskip
 
{\sf Step 1: Bound from above on the area of a portion $V \subset U$ bounded by two characteristic lines.} 
Consider a small arc $g([s,s'])$ of $\C$ joining $a:=g(s)$ to $b=g(s')$. We denote by $\theta$ the angle between the lines $L_a$ and $L_b$ and assume that
\begin{equation}\label{a15}
\theta \le \frac{\eta_0}{2}.
\end{equation}
This is possible if $s$ and $s'$ are close enough because $m$ is continuous.  Let $c$ (resp. $d$) be the intersection {point} of $L_{a}$ (resp. $L_{b}$) with $\partial A$ distinct from $a$ (resp. $b$). Note that $c\ne d$, otherwise $L_{a}$ would be intersecting $L_{b}$ at $c=d$ which would contradict the fact  that $v\in H^1(A;\R^2)$.  Whenever $\theta\ne 0$ we consider the intersection point $z_\theta$ of the lines $L_{a}$ and $L_{b}$; it must lie outside $\ol A$, lest, once more, $m$ not be $H^1(A;\R^2)$.

Let $V$ be the convex subdomain of $U$ bounded by $g([s,s'])$, $L_{a} \cap A$, $L_{b} \cap A$. Consider the open half-plane $H$ passing through the points $x_0$ and $x_1$, and such that $\C \subset H$. {If $\theta\ne 0$ and $z_\theta \in H$, by convexity of $A$, the area of $V$ can be estimated as 
\begin{equation}\label{a5}
\LL^2(V)\le C (|a-b|+\sin\theta),
\end{equation}
while, if $\theta=0$ or $z_\theta \not\in H$, that area is immediately seen to be controlled by
\begin{equation}\label{a50}\LL^2(V)\le C |a-b|,\end{equation}
for some constant $C>0$ only depending on $A$. }

Let us zero in on the case where $\theta\ne 0$ and $z_\theta \in H$ since, otherwise, estimate \eqref{a50} will suffice for our purpose as seen later. 

 \medskip
 
 {\sf Step 2: Bound of  $\|\nabla m\|_{L^2(V)}$ from below when $\theta\ne 0$ and $z_\theta\in H$.}
Assume, without loss of generality, that $|a-z|\le |b-z|$ and define $T$ to be the trapeze with boundary $[a,\bar c]\cup[\bar c,\bar d]\cup[b,\bar d]\cup [a,b]$ where $\bar c \in L_a$, $\bar d \in L_b$, $]\bar c,\bar d[ \; \subset A$ is parallel to $]a,b[$ and either $\bar c=c$, or $\bar d=d$. The convexity of $V$ ensures that $T \subset V$.  We call $\eta$ the angle between $[a,b]$ and $L_a$. We further denote by $h$ the distance between the line segments $[a,b]$ and $[\bar c, \bar d]$ and note that, in the notation of \eqref{a-1},
\begin{equation}\label{a21}
h\ge \alpha \min\{\sin \eta,\sin(\eta-\theta)\} \geq \alpha \sin\eta \sin(\eta-\theta);
\end{equation}
see figure \ref{fig:a}.

 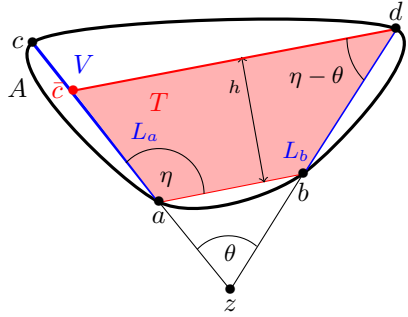
\begin{figure}[hbtp]
\scalebox{1}{\begin{tikzpicture}

\draw[style=very thick] plot[smooth cycle] coordinates {(-1.18,1.19) (3.65,1.37)(2.42,-0.55) (.5,-.93)};

\draw (-1.38,0.59)  node{$A$}; 

\draw[color=blue][style=very thick] (2.42,-0.55) -- (3.65,1.37) ;

\draw[color=blue][style=very thick] (-1.18,1.19) -- (.5,-.93) ; 

\draw (-.5,.9) node[color=blue]{$V$}; 

\draw (.5,-.93) -- (1.45,-2.1) ; \draw (1.45,-2.1) -- (2.42,-0.55) ;

\draw[color=red][style=thick]  (2.42,-0.55) --  (.5,-.93);
\draw (1.8,-1.52) arc (49:137:.57); \draw (1.45,-1.6) node{{\small $\theta$}};

\fill[color=red!30]  (2.42,-0.55) --  (.5,-.93)--(-.64,.55)--(3.65,1.37)-- (2.42,-0.55) ;
\draw (.5,.4) node[color=red]{$T$};
\draw(0.,0.) node[right][color=blue]{\small $L_{a}$}; \draw(2.02,-.26) node[right][color=blue]{\small $L_{b}$};

 \draw[<->] (1.9, -.67)-- (1.6, 1.0); \draw (1.3,.6) node[right]{{\scriptsize $h$}};

 \draw (1.1,-.82) arc (0:137:.6);\draw (0.6,-0.4) node[below]{\small{$\eta$}};

\draw  (-1.18,1.19) node{{\small $\bullet$}}; \draw (-1.18,1.19) node[left]{$c$};

\draw  (.5,-.93) node{{\small $\bullet$}};  \draw (.5,-.93) node[below]{$a$};

\draw (1.45,-2.1) node{{\small $\bullet$}};  \draw (1.45,-2.1) node[below]{$z$};

\draw  (2.42,-0.55) node{{\small $\bullet$}}; \draw (2.42,-0.55) node[below]{$b$};

\draw[color=red][style=thick]  (-.64,.55) -- (3.65,1.37) ;

\draw  (3.65,1.37) node{{\small $\bullet$}}; \draw (3.65,1.37) node[above]{$d$};

\draw  (-.64,.55) node[color=red]{{\small $\bullet$}};  \draw (-.65,.55) node[left][color=red]{$\bar c$};

\draw (3,1.25) arc (180:224:0.8);  \draw (2.6,1) node[below]{\small{$\eta-\theta$}};

\end{tikzpicture}}
\caption{\small  The sets $A$, $V$ and $T$.}
\label{fig:a}
\end{figure}

 Let us consider an orthonormal basis $\{e_1,e_2\}$ of $\R^2$ where the origin is set at the point $a$, the first vector $e_1=\frac{b-a}{|b-a|}$ is oriented in the direction of the vector $b-a$, and $e_2$ (which is orthogonal to $e_1$) is such that $T \subset \{x \in \R^2 : x \cdot e_2 \geq 0\}$. Denoting by $(x_1,x_2)$ the coordinates of $x$ in that basis, purely geometric considerations lead to
$$T= \Bigg\{(x_1,x_2) \in \R^2 : \; 0 \leq x_2 \leq h, \; a(x_2) \leq x_1 \leq b(x_2)\Bigg\},$$
where, for $x_2 \in 0,h]$, $(a(x_2),x_2) \in [a,\bar c]$, $(b(x_2),x_2) \in [b,\bar d]$ and the length of the section of $T$ at height $x_2 \in [0,h]$ is given by
$$b(x_2)-a(x_2)=|a-b|+x_2\left(\tan\left(\frac{\pi}{2}+\theta-\eta\right)-\tan\left(\frac{\pi}{2}-\eta\right)\right).$$
 Then, Fubini's theorem together with Jensen's inequality yields
\begin{eqnarray}
\int_V |\nabla m|^2\, dx \geq \int_T |\nabla m|^2\, dx & \ge & \int_0^h \left[\int_{a(x_2)}^{b(x_2)}|\pa_1 m|^2\ dx_1\right]dx_2\nonumber\\
& \ge & \int_0^h \frac{|m(b(x_2),x_2)-m(a(x_2),x_2)|^2}{b(x_2)-a(x_2)} dx_2.\label{eq:a?}
\end{eqnarray}
 Using that $m$ is constant along $[a,\bar c] \subset L_a$ and $[b,\bar d] \subset L_b$ together with the fact that $\theta=\angle(L_a,L_b)$, we infer that
 $|m(b(x_2),x_2)-m(a(x_2),x_2)|^2=2(1-\cos\theta)=4\sin^2(\frac\theta2)\geq \sin^2\theta$.  {Hence,  in view of \eqref{a21} together with the trigonometric formula
 $$\tan\left(\frac{\pi}{2}+\theta-\eta\right)-\tan\left(\frac{\pi}{2}-\eta\right)=\frac{\sin\theta}{\sin\eta \sin(\eta-\theta)},$$
 \eqref{eq:a?} becomes
 \begin{eqnarray}\label{a22}
\int_V |\nabla m|^2\, dx & \ge & \int_0^{ \alpha \sin\eta \sin(\eta-\theta)} \frac{\sin^2\theta}{|a-b|+x_2\left(\tan\left(\frac{\pi}{2}+\theta-\eta\right)-\tan\left(\frac{\pi}{2}-\eta\right)\right)} dx_2\nonumber\\
& = & \sin\theta\sin\eta\sin(\eta-\theta)\ln\left(1+\frac{\alpha \sin\theta}{|a-b|}\right).
\end{eqnarray}
}
The angle relations \eqref{a12}, \eqref{a15} imply that
 $$
\sin\eta\sin(\eta-\theta)\ge \sin\eta_0\sin\left(\frac{\eta_0}{2}\right):=C_0,$$
where $C_0>0$ is a constant only depending on $A$,  and \eqref{a22} finally becomes, 
\begin{equation}\label{a27}
\int_V |\nabla m|^2\, dx\ge C_0 \sin\theta \ln\left(1+\frac{\alpha \sin\theta}{|a-b|}\right).  
\end{equation}

 \medskip

{\sf Step 3: Conclusion.} Since $\HH^1(Z)=0$, by definition of the Hausdorff measure, for all $\delta>0$, there exist an at most countable set $I \subset \N$ and $\{B_i\}_{i \in I}$ such that $Z  \subset \bigcup_{i \in I} B_i$ and
\begin{equation}\label{eq:haus}
\sum_{i \in I} {\rm diam}(B_i) \leq \delta.
\end{equation}
Moreover, there is no loss of generality in assuming that the sets $B_i=B_{\varrho_i}(z_i)$ are discs centered at a point $z_i \in Z$ with radius {$\varrho_i >0$. We can further suppose that $\overline B_i \setminus \overline \C=\emptyset$ and $\theta_i \leq \eta_0/2$  (see \eqref{a12}--\eqref{a15}) at the expense of decreasing  $\delta$ and thanks to the continuity of $m$.}

For each $i \in I$, we set
$\{a_i,b_i\}=\partial B_i \cap \C,$
and define the associated $V_i$, $\theta_i$, $H_i$ and $z_{\theta_i}$.  {We partition $I$ into 
$$I^+:=\{i\in I: \; \theta_i\ne 0\mbox{ and }z_{\theta_i}\in H_i\}, \quad I^-:=\{i\in I:\; \theta_i=0\mbox{ or }z_{\theta_i}\notin H_i\}.$$
Using \eqref{a50},
\begin{equation}\label{a4bis}
\LL^2(V_i)\le C{\rm diam}(B_i), \;i \in I^-,
\end{equation}
while \eqref{a5} ensures that
\begin{equation}\label{a5bis}
\LL^2(V_i)\le  C ({\rm diam}(B_i)+\sin\theta_i),\; i \in I^+.
\end{equation}
Here $C>0$ is a constant only depending on $A$. In view of \eqref{a27},
\begin{equation}\label{a27bis}
\int_{V_i} |\nabla m|^2\, dx \ge C_0\sin\theta_i     \ln\left(1+\frac{\alpha\sin\theta_i}{{\rm diam}(B_i)}\right),\; i \in I^+.
\end{equation}
}

Define the disjoint sets of indices 
$$I^+_0:=\left\{i\in I^+: \;  \frac{\sin\theta_i}{{\rm diam}(B_i)}< 1\right\}, \quad \; I^+_j:=\left\{i\in I^+: \; 2^{j-1}\leq \frac{\sin\theta_i}{{\rm diam}(B_i)}< 2^j\right\} \text{ for } j\ge 1,$$
{so that}
$$\sum_{i\in I^+} \sin\theta_i= \sum_{j=0}^{j_0}\sum_{i\in I^+_j}\sin\theta_i+\sum_{j>j_0}\sum_{i\in I^+_j}\sin\theta_i.$$
Now, 
\begin{equation}\label{a100}
{\sum_{j=0}^{j_0}\sum_{i\in I^+_j}\sin\theta_i} \le \sum_{j=0}^{j_0}\sum_{i\in I^+_j}2^j{\rm diam}(B_i)\le 2^{j_0} (j_0+1)\sum_{{i \in I}}{\rm diam}(B_i),
\end{equation}
while, appealing to \eqref{a27bis}, 
\begin{eqnarray}\label{a101}
{\sum_{j>j_0}\sum_{i\in I^+_j}\sin\theta_i} & \le & \sum_{j>j_0} \frac1{\ln(1+\alpha 2^{j_0})}\sum_{i\in I^+_j}\sin\theta_i\ln\left(1+\alpha \frac{\sin\theta_i}{{\rm diam}(B_i)}\right)\nonumber\\
& \le & \frac{1}{C_0\ln(1+\alpha 2^{j_0})} \sum_{{i\in I^+}}\int_{V_i}|\nabla m|^2\ dx\nonumber\\
& \le &\frac{1}{C_0\ln(1+\alpha 2^{j_0})} \int_A|\nabla m|^2\ dx,
\end{eqnarray}
where we used that, since characteristic lines cannot intersect inside $A$, the sets $\{V_i\}_{{i \in I^+}}$ are pairwise disjoint. {Gathering \eqref{eq:haus}, \eqref{a4bis}, \eqref{a5bis}, \eqref{a100} and \eqref{a101}, we  obtain that 
\begin{eqnarray}\label{a4}
\sum_{i\in I} \LL^2(V_i) & = & \sum_{i\in I^-} \LL^2(V_i)+\sum_{i\in I^+} \LL^2(V_i)\nonumber\\
& \le & C\sum_{i\in I}{\rm diam}(B_i)+ C\sum_{i\in I^+}\sin\theta_i \nonumber\\
& \le & C\left(\delta + 2^{j_0} (j_0+1)\delta+ \frac{C_0}{\ln(1+\alpha 2^{j_0})}\int_A|\nabla m|^2 \, dx\right).
\end{eqnarray}

Given} $\e>0$, choosing first $j_0 \in \N$ large enough and then $\delta>0$ small enough, we conclude that 
$$\sum_{i \in I} \LL^2(V_i)\le\e.$$ 
Since $\bigcup_{z\in Z}(L_z\cap A) \subset \bigcup_{i\in I} V_i$, this proves that the set $\bigcup_{z\in Z}(L_z\cap A)$ is $\LL^2$-negligible.
\end{proof}

\end{document}